\documentclass[preprint,12pt]{elsarticle}
\usepackage{color}
\usepackage{amsfonts}
\usepackage{amsthm}
\usepackage{amsmath}
\usepackage{amssymb}
\usepackage{graphicx}
\providecommand{\U}[1]{\protect\rule{.1in}{.1in}}

\def\subneq{\mathop{\raise 0.7ex \hbox{$\subset$}}\!\!\!\!\!\!{\raise -0.6ex\hbox{$\neq$}}\,}

\def\div{\mathop{\rm div}\nolimits}

\def\u{\underline}

\def\RR{\mathbb{R}}

\def\1{{1\hskip-0.25em{\rm l}}}

\def\sobre#1#2{\lower 1ex \hbox{ $#1 \atop #2 $ } }
\def\bajo#1#2{\raise 1ex \hbox{ $#1 \atop #2 $ } }

\def\Ae{A^{0,\ep}}
\def\Be{B^{0,\ep}}
\def\Ce{C^{0,\ep}}
\def\De{D^{0,\ep}}

\def\div{{\rm div}}
\def\ep{\varepsilon}

\def\p{\partial}

\def\O{\Omega}
\def\Oe{\Omega^\ep}

\def\u2{{u^\ep \over \ep^2 }}
\def\u3{{\displaystyle {\bar u}^\ep \over \ep^2 }}

\def\div{\mathop{\rm div}\nolimits}

\def\u{\underline}

\def\sobre#1#2{\lower 1ex \hbox{ $#1 \atop #2 $ } }
\def\bajo#1#2{\raise 1ex \hbox{ $#1 \atop #2 $ } }

\def\div{{\rm div}}

\def\ep{\varepsilon}

\def\p{\partial}

\def\u2{{u^\ep \over \ep^2 }}
\def\u3{{\displaystyle {\bar u}^\ep \over \ep^2 }}
\def\p{\partial}
\def\pej{\mbox{\rm Pe}_{j}}
\def\pei{\mbox{\rm Pe}_{i}}
\def\dsp{\displaystyle}
\def\x{x}
\def\tlambda{\tilde{\lambda}}
\def\fin{{\rm final}}
\newtheorem{theorem}{Theorem}

\newtheorem{lemma}[theorem]{Lemma}

\newtheorem{proposition}[theorem]{Proposition}
\newtheorem{remark}[theorem]{Remark}

\begin{document}

\begin{frontmatter}

\title{Role of  non-ideality for the  ion transport in  porous media: derivation of the macroscopic equations using upscaling}

\author[GA]{Gr\'egoire Allaire\fnref{fn1,fn2}}
\ead{gregoire.allaire@polytechnique.fr}
\author[GA]{Robert Brizzi\fnref{fn2}}
\ead{robert.brizzi@polytechnique.fr}
\author[JFD]{Jean-Fran\c{c}ois Dufr\^eche\fnref{fn2}}
\ead{jean-francois.dufreche@univ-montp2.fr}
\author[AM]{Andro Mikeli\'c\corref{cor1}\fnref{fn2}}
  \ead{mikelic@univ-lyon1.fr}
\cortext[cor1]{Corresponding author}
\author[AP]{Andrey Piatnitski\fnref{fn2}}
\ead{andrey@sci.lebedev.ru}
\fntext[fn1]{G. A. is a member of the DEFI project at INRIA Saclay Ile-de-France.}
\fntext[fn2]{This research was partially supported by the project DYMHOM (De la dynamique mol\'eculaire, via l'homog\'en\'eisation, aux mod\`eles macroscopiques de poro\'elasticit\'e et \'electrocin\'etique) from the program NEEDS (Projet f\'ed\'erateur Milieux Poreux MIPOR),
GdR MOMAS and GdR PARIS.  The authors would like to thank O. Bernard, V. Marry, P. Turq and B. Rotenberg from the laboratory Physicochimie des Electrolytes, Colloides et Sciences Analytiques (PECSA), UMR CNRS 7195, Université P. et M. Curie, for helpful discussions.}
\address[GA]{ CMAP, Ecole Polytechnique,
F-91128 Palaiseau, France}
\address[JFD]{Universit\'e de Montpellier 2, Laboratoire Mod\'elisation M\'esoscopique et Chimie Th\'eorique (LMCT),
Institut de Chimie S\'eparative de Marcoule ICSM
UMR 5257, CEA / CNRS / Universit\'e de Montpellier 2 / ENSCM
Centre de Marcoule, B\^at. 426 BP 17171,
30207 Bagnols sur C\`eze Cedex, France}
\address[AM]{Universit\'e de Lyon, Lyon, F-69003, France; Universit\'e
Lyon 1, Institut Camille Jordan, UMR 5208, B\^at. Braconnier,  43, Bd du 11
novembre 1918, 69622 Villeurbanne Cedex,
France}
\address[AP]{ Narvik University College, Postbox 385, Narvik 8505, Norway;  Lebedev Physical Institute, Leninski prospect 53, 119991, Moscow, Russia }




\begin{abstract}
This paper is devoted to the homogenization (or upscaling) of a system of
partial differential equations describing the non-ideal transport of a N-component
electrolyte in a dilute Newtonian solvent through a rigid porous medium.
Realistic non-ideal effects are taken into account by an approach based on the mean spherical approximation
(MSA) model which takes into account finite size ions and screening effects.
We first consider equilibrium solutions in the absence of external forces.
In such a case, the velocity and diffusive fluxes vanish and the equilibrium
electrostatic potential is the solution of a variant of Poisson-Boltzmann
equation coupled with algebraic equations. Contrary to the ideal case,
this nonlinear equation has no monotone structure. However, based on
invariant region estimates for Poisson-Boltzmann equation and for small
characteristic value of the solute packing fraction, we prove existence
of at least one solution. To our knowledge this existence result is new
at this level of generality.
When the motion is governed by a small static electric
field and a small hydrodynamic force, we generalize O'Brien's argument
to deduce a linearized model. Our second main result is the rigorous
homogenization of these linearized equations and the proof that the
effective tensor satisfies Onsager properties, namely is symmetric
positive definite. We eventually make numerical comparisons with the
ideal case. Our numerical results show that the MSA model confirms
qualitatively the conclusions obtained using the ideal model but there are
quantitative differences arising that can be important at high charge or high concentrations.
\end{abstract}


\end{frontmatter}


\section{Introduction}

The quasi-static transport of an electrolyte through an electrically charged porous medium is an important and well-known multiscale problem in geosciences and porous materials modeling. An $N$-component electrolyte is a dilute solution of $N$ species of charged particles, or ions, in a fluid which saturates a rigid charged porous medium \cite{Lyklema}. The macroscopic dynamics of such a system is controlled by several phenomena. First the global hydrodynamic flow, which is commonly modelled by the Darcy's law depends on the geometry of the pores and also on the charge distributions of the system. Second, the migration of ions because of an electric field can be quantified by the conductivity of the system. Third, the diffusion motion of the ions is modified by the interaction with the surfaces, but also by the interactions between the solute particles. Lastly, electrokinetic phenomena are due to the {\bf electric double layer} (EDL) which is formed as a result of the interaction of the electrolyte solution which neutralizes the charge of the solid phase at the pore solid-liquid interface. Thus, an external electric field can generate a so-called electro-osmotic flow and reciprocally, when a global hydrodynamic flow is applied, an induced streaming potential is created in the system.

The EDL can be split into several parts, depending on the strength of the electrostatic coupling. There is a condensed layer of ions of typical size $l_G$ for which the attraction energy with the surface $e \Sigma/\mathcal{E}l_G$ (with $\Sigma$ the surface charge and $e$ the elementary charge) is much more than the thermal energy $k_BT$ (with $k_B$ Boltzmann's constant and $T$ the temperature). The corresponding characteristic length $l_G=\mathcal{E}k_B T/\Sigma e$ (Gouy length) is typically less than one nanometer. Consequently, the layer of heavily adsorbed ions practically depends on the molecular nature of the interface and it is generally known as the {\bf Stern layer}. After the Stern layer the {\bf electrostatic diffuse layer} or {\bf Debye's layer} is formed, where the ion density varies. The EDL is the union of Stern and diffuse layers. The thickness of the diffuse layer is predicted by the Debye length $\lambda_D$ which depends on the electrolyte concentration. For low to moderate electrolyte concentrations $\lambda_D$ is in the nanometric range. Outside Debye's layer, in the remaining bulk fluid, the solvent can be considered as electrically neutral.

The large majority of theoretical works are concerned with a simple (so-called {\bf ideal}) model for which the departure of ideality of ions are neglected (see later in this introduction a precise definition of ideality). Thus the macroscopic descriptions of charged porous media such as the ones using finite element methods \cite{AM:03}, homogenization approaches \cite{MM:02} or lattice-Boltzmann methods \cite{Rotenberg} are commonly based on the Poisson-Nernst-Planck formalism for which the local activity coefficients of ions are neglected and the transport properties are modelled solely from the mobility at infinite dilution. In addition, the boundary condition for the electrostatic interaction between the two phases is very often simplified by replacing the bare surface charge $\Sigma$, which corresponds to the chemistry of the system, by surface potential $\Psi$. Its boundary value at the no slip plane is known as the zeta potential $\zeta$. In fact, it is rather the surface charge density $\Sigma$, proportional to the normal derivative of $\Psi$, than $\zeta$, which is the relevant parameter (this is confirmed by an asymptotic analysis in \cite{PoissBoltz:12}).

A few studies do not model the details of the EDL. Under the presence of an external electric field $\mathbf{E}$, the charged fluid may acquire a plug electro-osmotic flow velocity which is proportional to $\mathbf{E} \zeta$ and given by the so-called Smoluchowski's formula. In the case of porous media with large pores, the electro-osmotic effects are modeled by introducing an effective slip velocity at the solid-liquid interfaces, which comes from the Smoluchowski formula. In this setting, the effective behavior of the charge transport through spatially periodic porous media was studied by Edwards in \cite{E:95}, using the volume averaging method. These methods for which the transport beyond the EDL is uncoupled from the one in the EDL are not valid for numerous systems, such as clays because the characteristic pore size is also of the order of the EDL size (a few hundreds of nanometers or even less). Therefore the Debye's layer fills largely the pores and its effect cannot anymore be modeled by an effective slip boundary condition at the liquid-solid interface.

In the present paper, we consider continuum equations (such as the Navier-Stokes or the Fick equations) as the right model for the description of porous media at the pore scale where the EDL phenomena and the pore geometry interact. The typical length scale for which these continuous approach are valid is confirmed to be both experimentally (see e.g. \cite{CH:85}) and theoretically \cite{Osmo1,Osmo2} close to 1 nanometer. Therefore, we consider continuum equations at the microscopic level and, more precisely, we couple the incompressible Stokes equations for the fluid with the electrokinetic model made of a global electrostatic equation and one convection-diffusion equation for each type of ions.

The most original ingredient of the model is the treatment of the departure from ideality. Electrolyte solutions are not ideal anymore as fas as the ion concentration is not dilute \cite{Barthel}. Typically simple 1-1 electrolyte, such as NaCl in water have an activity coefficient which is close to 0.6 at molar concentrations {  (while it is equal to 1 by definition in the ideal case)} and the non ideality effects is even more important for the transport coefficients \cite{Dufreche05,VanDamme}. Thus any ideal model can only be in semi-quantitative agreement with a more rigorous model if departure from ideality are neglected. In the present article, we use a new approach based on the Mean Spherical Approximation (MSA), for which the ions are considered to be charged hard spheres \cite{Hansen,Blum77}. This model is able to describe the properties of the solutions up to molar concentrations. In addition, a generalization of the Fuoss-Onsager theory based on the Smoluchowski equation has been developped \cite{EbelMSA,BernardMSA,Dufreche05,DufrecheDB,Dufreche02,VanDamme} by taking into account this model, and it is possible to predict the various transport coefficients of bulk electrolyte solutions up to molar concentrations. This MSA transport equations extend the well known Debye-Huckel-Onsager limiting law to the domain of concentrated solutions. They have also been proved to be valid \cite{Jardat} for confined solutions in the case of clays by comparing their predictions to molecular and Brownian dynamics simulations.

A more detailed, mathematically oriented, presentation of the fundamental concepts of electroosmotic flow in nanochannels can be found in the book \cite{KBA:05} by Karniadakis et al., pages 447-470, from which we borrow the notations and definitions in this introduction.
We now describe precisely our stationary model, describing at the pore scale the electro-chemical interactions of an $N$-component electrolyte in a dilute Newtonian solvent.
All quantities are given in {\bf SI units}. We start with the following mass conservation laws
\begin{gather}
\mbox{ div } \bigg( {\bf j}_i + \mathbf{u} n_i \Big) =0 \quad \mbox{in } \quad \Omega_p ,\quad i=1, \dots , N,
\label{AM7}
\end{gather}
where $\O_p$ is the pore space of the porous medium $\O$, $i$ denotes the solute species,
$\mathbf{u}$ is the hydrodynamic velocity and $n_i$ is the $i$th species concentration.
For each species $i$, $\mathbf{u} n_i $ is its convective flux and ${\bf j}_i$ its
migration-diffusion flux.

The solute velocity is given by the incompressible Stokes equations
with a forcing term made of an exterior hydrodynamical force $\mathbf{f}$
and of the electric force applied to the fluid thanks to the charged species
\begin{gather}
\eta \Delta \mathbf{u} = \mathbf{f} + \nabla p +  e \sum^N_{j=1} z_j n_j \nabla \Psi \qquad \mbox{in } \quad \Omega_p ,
\label{AM5} \\
\mbox{ div } \mathbf{u} =0 \qquad \mbox{in } \quad \Omega_p ,
\label{AM6} \\ \mathbf{u} =0 \qquad \mbox{on } \quad \p \Omega_p \setminus \p \O,
\label{AM8}
\end{gather}
where $\eta>0$ is the shear viscosity, 
$p$ is the pressure, $e$ is the elementary charge, $z_i$ is the charge number
of the species $i$ and $\Psi$ is the electrostatic potential. The pore space boundary
is $\p \O_p$ while $\p \O$ is the outer boundary of the porous medium $\O$.
On the fluid/solid boundaries $\p \Omega_p \setminus \p \O$ we assume the
no-slip boundary condition (\ref{AM8}). {  For simplicity, we shall assume
that $\Omega$ is a rectangular domain with periodic boundary conditions
on $\p \O$. Furthermore, in order to perform a homogenization process, we assume that the pore distribution is periodic in $\O_p$. }

We assume that all valencies $z_j$ are different. If not, we lump together
different ions  with the same valency. Of course, for physical reasons,
all valencies $z_j$ are integers. We rank them by increasing order and
we assume that they are both anions and cations, namely positive and negative
valencies,
\begin{equation}
\label{valence}
z_1 < z_2 < ... < z_N , \quad z_1<0<z_N ,
\end{equation}
and we denote by $j^+$ and $j^-$ the sets of positive and negative valencies.

The migration-diffusion flux ${\bf j}_i$ is given by the following linear relationship
\begin{gather}
{\bf j}_i = -\sum_{j=1}^N  L_{ij}  (n_1 , \dots , n_N ) \big( \nabla \mu_j
+ z_j e \nabla \Psi \big), \quad  i=1, \dots , N  ,
\label{electroflux}
\end{gather}
where $ L_{ij}  (n_1 , \dots , n_N ) $ is the Onsager coefficient between $i$ and $j$ and $\mu_j$ is the chemical potential of the species $j$ given by
\begin{gather}
\mu_j = \mu_j^0 + k_B T \ln n_j + k_B T \ln \gamma_j (n_1 , \dots , n_N ) , \quad  j=1, \dots , N , \label{Chempot}
\end{gather}
with $\gamma_j$ being the {\bf activity coefficient} of the species $j$,
$k_B$ is the Boltzmann constant, $\mu_j^0$ is the standard chemical potential
expressed at infinite dilution and $T$ is the absolute temperature.
{   The sum of all fluxes ${\bf j}_i$ is not zero because the solvent is
not considered here and ${\bf j}_i$ is a particle flux.}
The Onsager tensor $L_{ij}$ is made of the linear Onsager coefficients
between the species $i$ and $j$. It is symmetric and positive definite.
Furthermore, on the fluid/solid interfaces a no-flux condition holds true
\begin{equation}
\label{AM9}
\mathbf{j}_i \cdot \nu =0 \quad \p \Omega_p \setminus \p \O, \quad i=1, \dots , N.
\end{equation}
The electrostatic potential is calculated from Poisson equation with the
electric charge density as bulk source term
\begin{gather}
\mathcal{E} \Delta \Psi = - e \sum^N_{j=1} z_j n_j \qquad \mbox{in } \quad \Omega_p ,
\label{AM3} \end{gather}
where $\mathcal{E} = \mathcal{E}_0 \mathcal{E}_r$ is the dielectric constant
of the solvent. The surface charge $\Sigma$ is assumed to be given at the
pores boundaries, namely the boundary condition reads
\begin{gather}
\mathcal{E} \nabla \Psi \cdot \nu =-\Sigma \qquad \mbox{on } \quad \p \Omega_p \setminus \p \O,
\label{AM4}
\end{gather}
where $\nu$ is the unit exterior normal to $\Omega_p$.

The activity coefficients $\gamma_i$  and the Onsager coefficients $L_{ij}$
depend on the electrolyte. At infinite dilution the solution can be considered
{\bf ideal} and we have $\gamma_i =1$ and $L_{ij} = \delta_{ij} n_i D_i^0 / (k_B T)$,
where $D^0_i>0$ is the diffusion coefficient of species $i$ at infinite dilution.
At finite concentration, these expressions which correspond to the Poisson-Nernst-Planck
equations are not valid anymore. {\bf Non-ideal effects} modify the ion transport
and they are to be taken into account if quantitative description of the system
is required. Different models can be used. Here we choose the {\bf Mean Spherical
Approximation (MSA)} in simplified form \cite{Dufreche05} which is valid if the
diameters of the ions are not too different. The activity coefficients read
\begin{gather}
\ln \gamma_j = -\frac{L_B \Gamma z_j^2}{1+ \Gamma \sigma_j}+ \ln \gamma^{HS} , \quad j=1, \dots , N ,
\label{gamma-j}
\end{gather}
where {  $\sigma_j$ is the $j$-th ion diameter,}
$L_B$ is the Bjerrum length given by $L_B = e^2 / (4\pi \mathcal{E} k_B T)$,
{  $\gamma^{HS}$ is the hard sphere term defined by (\ref{Hardsphere}),}
and $\Gamma$ is the {\bf MSA screening parameter} defined by
\begin{gather}
\Gamma^2 = \pi L_B \sum_{k=1}^N \frac{n_k z_k^2}{(1+ \Gamma \sigma_k )^2} .
\label{Gamma}
\end{gather}
For dilute solutions, i.e., when all $n_j$ are small, we have
$$
2 \Gamma \approx \kappa = \frac{1}{\lambda_D} \quad \mbox{ with } \quad
\lambda_D = \dsp \sqrt{\frac{\mathcal{E} k_B T}{ e^2 \sum_{k=1}^N n_k z_k^2 }} ,
$$
where $\lambda_D$ is the Debye length. Thus, $1/2\Gamma$ generalizes $\lambda_D$
at finite concentration and it represents the size of the ionic spheres when the
ion diameters $\sigma_i$ are different from zero.
(In the sequel we shall use a slightly different definition
of the Debye length, relying on the notion of characteristic
concentration, see Table \ref{Data}.)
In (\ref{gamma-j}) $\gamma^{HS}$ is the hard sphere term which is independent of
the type of species and is given by
\begin{equation}
\label{Hardsphere}
\ln \gamma^{HS} = p(\xi) \equiv \xi \frac{8-9\xi + 3\xi^2}{(1-\xi)^3} ,
\quad \mbox{ with } \quad \xi = \frac{\pi}{6} \sum_{k=1}^N n_k \sigma_k^3  ,
\end{equation}
where $\xi$ is the solute packing fraction.

The Onsager coefficients $L_{ij}$ are given by
\begin{gather}
L_{ij}  (n_1 , \dots , n_N )  = n_i \bigg( \frac{D_i^0}{k_B T} \delta_{ij} + {\bf\Omega}_{ij} \bigg)
\bigg( 1+ \mathcal{R}_{ij} \bigg), \; i,j =1, \dots , N,
\label{Offdiag}
\end{gather}
where ${\bf\Omega}_{ij} = {\bf\Omega}_{ij}^c + {\bf\Omega}_{ij}^{HS}$ stands for the
hydrodynamic interactions in the MSA formalism and there is no summation for repeated
indices in (\ref{Offdiag}). It is divided into two terms: the Coulombic part is
\begin{gather}
{\bf\Omega}_{ij}^{c}=-\frac{1}{3\eta}\frac{z_{i}z_{j}L_{B} n_{j}}
{(1+\Gamma\sigma_{i})(1+\Gamma\sigma_{j})\left(\Gamma+{\displaystyle \sum_{k=1}^N n_{k}
\frac{\pi L_{B} z_{k}^{2}\sigma_{k}}{(1+\Gamma\sigma_{k})^{2}}}\right)} ,
\label{Omegac}
\end{gather}
and the hard sphere part is
\begin{gather}
{\bf\Omega}_{ij}^{HS}=-\frac{\left(\sigma_{i}
+\sigma_{j}\right)^{2}}{12 \eta} n_{j}\frac{1-\tilde{X}_{3}/5
+(\tilde{X}_{3})^{2}/10}{1+2\tilde{X}_{3}},
\label{OmegaHS}
\end{gather}
with
\begin{gather}
\tilde{X}_{3}=\frac{\pi}{6}\sum_{i=1}^N n_{i} \frac{3X_{1}X_{2} +X_{3} X_0}{4X_{0}^2}   \quad  \mbox{and } \; X_{k}=\frac{\pi}{6}\sum_{i=1}^N n_{i} \sigma_{i}^{k} .
\label{Diag}
\end{gather}
In (\ref{Offdiag}) $\mathcal{R}_{ij}$ is the electrostatic relaxation term given by
\begin{gather}
     \mathcal{R}_{ij}=
\frac{ \kappa_q^2 e^2  z_i z_j   }{3 \mathcal{E} k_B T (\sigma_{i} +\sigma_j ) ( 1 + \Gamma \sigma_i ) ( 1 + \Gamma \sigma_j ) } \;\frac{  1 - e^{ - 2 \kappa_{q}  (\sigma_{i} +\sigma_j )} } { \kappa_{q}^2 + 2 \Gamma \kappa_{q} + 2 \Gamma^2
  - 2 \pi L_B \displaystyle \sum_{k=1}^N n_k
\frac{z_{k}^{2} e^{- \kappa_{q}\sigma_k}}{(1+\Gamma\sigma_{k})^{2}} }
\label{Relax}
\end{gather}
where $\kappa_q>0$ is defined by
\begin{gather}
\kappa_q^2 =\frac{e^2}{\mathcal{E} k_{\mathrm{B}} T}
\frac{\sum_{i=1}^N n_i z_i^2 D_i^{0}}{\sum_{i=1}^N D_i^{0}} .
\label{Relaxk}
\end{gather}
All these coefficients $\gamma_j, \Gamma, {\bf\Omega}_{ij}, \mathcal{R}_{ij}$
are varying in space since they are functions of the concentrations $n_j$.
The $N\times N$ tensor $(L_{ij})$ is easily seen to be symmetric. However,
to be coined "Onsager tensor" it should be positive too, which is not
obvious from the above formulas. The reason is that they are only approximations
for not too large concentrations. Nevertheless, when the concentrations $n_j$
are small, all entries $L_{ij}$ are first order perturbations of the ideal
values $\delta_{ij} n_i D_i^0 / (k_B T)$ and thus the Onsager tensor is
positive at first order.
The various parameters appearing in (\ref{AM7})-(\ref{Relaxk}) are defined in Table \ref{Data}.

\begin{table}[ht]
\centerline{\begin{tabular}{|l|l|l|} \hline\hline
\emph{} & \emph{ QUANTITY }   &
{\emph{CHARACTERISTIC VALUE}}
  \\
\hline   e &  electron charge  &  $1.6$e$-19$ C (Coulomb) \\
\hline    $D_i^0$ & diffusivity of the $i$th species  & $D_i^0 \in (1.333 , 2.032)$e$-09 \, m^2/s$ \\
\hline    $k_B$  &  Boltzmann constant   & $1.38$e$-23 \, J/K$  \\
\hline    $n_c$   & characteristic concentration  & $( 6.02 \, 10^{24}, 6.02 \, 10^{26})$ particles$/ m^3$ \\
\hline    $T$  & temperature   & $293^\circ K$ (Kelvin)\\
\hline    $\mathcal{E}$  &   dielectric constant  &  $6.93$e$-10\, C/(mV)$   \\
\hline  $\eta$   & dynamic viscosity    & $1$e$-3\, kg/(m \, s)$ \\
\hline   $\ell$   & pore size    & $5$e$-9$ m \\
\hline    $\lambda_D$  & Debye's length    & $\dsp \sqrt{\mathcal{E} k_B T / (e^2 n_c)} \in (0.042 , 0.42 )$ nm \\
\hline     $z_j$   &  $j$-th electrolyte valence    &  given integer\\
\hline    $\Sigma$   & surface charge density     & $0.129 C/m^2$ (clays)  \\
\hline    $\mathbf{f}$  &  given applied force & $N/m^3$ \\
\hline    $\sigma_j$  & $j$-th   hard sphere diameter  & $2$e$-10$ m \\
\hline    $\Psi_c$   & characteristic electrokinetic potential &  $0.02527$ V (Volt) \\
\hline    $L_B$  &   Bjerrum length  & $7.3$e$-10$ m \\
 \hline
\end{tabular}
} \caption{{\it Data  description}\label{Data}}
\end{table}

{  As already said we consider a rectangular domain
$\O = \prod_{k=1}^d(0,L_k)^d$ ($d=2,3$ is the space dimension),
$L_k>0$ and at the outer boundary $\p \Omega$ we set}
\begin{equation}\label{BC1}
    \Psi + \Psi^{ext}(x)  \ , \, n_i \ ,  \, \mathbf{u}
 \, \mbox{ and } \, p
 \, \mbox{ are } \, \Omega-\mbox{periodic.}
\end{equation}
The applied exterior potential $\Psi^{ext}(x)$ can typically be linear,
equal to $\mathbf{E}\cdot \x$, where $\mathbf{E}$ is an imposed electrical field.
Note that the applied exterior force $\mathbf{f}$ in the Stokes equations (\ref{AM5})
can also be interpreted as some imposed pressure drop or gravity force.
Due to the complexity of the geometry and of the equations, it is necessary
for engineering applications to upscale the system (\ref{AM7})-(\ref{AM4})
and to replace the flow equations with a Darcy type law, including electro-osmotic effects.

It is a common practice to assume that the porous medium has a periodic microstructure.
For such media, and in the ideal case, formal two-scale asymptotic analysis of system
(\ref{AM7})-(\ref{AM4}) has been performed in many previous papers. Many of these works
rely on a preliminary linearization of the problem which is first due to O'Brien et al.
\cite{OBW:78}. Let us mention in particular the work of Looker and Carnie in \cite{LC:06}
that we rigorously justify in \cite{AMP} and for which we provided numerical experiments
in \cite{ABDMP}. Other relevant references include \cite{AM:03}, \cite{A:01}, \cite{AS:81}, \cite{CSTA:96}, \cite{GCA:06},  \cite{MSA:01}, \cite{MM:02}, \cite{MM:03}, \cite{MM:06}, \cite{MM:06a}, \cite{MM:08}, \cite{RPA:06}, \cite{Rayetal:11}, \cite{Rayetal:12}, \cite{schmuck}.

Our goal here is to generalize these works for the non-ideal MSA model.
More specifically, we extend our previous works \cite{AMP}, \cite{ABDMP}
and provide the homogenized system for a linearized version of
(\ref{AM7})-(\ref{AM4}) in a rigid periodic porous medium (the linearization
is performed around a so-called equilibrium solution which satisfies the
full nonlinear system (\ref{AM7})-(\ref{AM4}) with vanishing fluxes).
{
The homogenized system is an elliptic system of $(N+1)$ equations
$$
- \div_x \mathcal{M} \nabla (p^0, \{\mu_j\}_{1\leq j\leq N}) = \mathcal{S} \quad \mbox{ in } \; \O ,
$$
where $p^0$ is the pressure, $\mu_j$ the chemical potential of the $j$-th species,
$\mathcal{M}$ the Onsager homogenized tensor and $\mathcal{S}$ a source term.
The $(N+1)$ equations express the conservation of mass for the fluid and the
$N$ species. More details will be given in Section \ref{Passlimit}.}

In Section \ref{sec2} we provide a dimensionless version of system (\ref{AM7})-(\ref{AM4}).
We also explain in Lemma \ref{lem.ideal} how the ideal case can be recovered
from the non-ideal MSA model in the  limit of small characteristic value
of the solute packing fraction.
Section \ref{sec3} is concerned with the definition of so-called equilibrium solutions
when the external forces are vanishing (but not the surface charge $\Sigma$).
Computing these equilibrium solutions amounts to solve a MSA variant of the
nonlinear Poisson-Boltzmann equation for the potential. Existence of a solution
to such a Poisson-Boltzmann equation is established in Section \ref{secexist}
under a smallness assumption for a characteristic value of the solute packing fraction.
To our knowledge this existence result is the first one at this level of generality.
Let us mention nevertheless that, in a slightly simpler setting (two species only
and a linear approximation of $p(\xi)$) and with a different method (based on
a saddle point approach in the two variables $\Psi$ and $\{n_j\}$), a previous
existence result was obtained in \cite{EJL}.
In Section \ref{sec4} we give a linearized version of system (\ref{AM7})-(\ref{AM4}).
We generalize the seminal work of O'Brien et al. \cite{OBW:78}, which was devoted
to the ideal case, to the present setting of the MSA model. Under the assumption
that all ions have the same diameter $\sigma_j$ we establish in Proposition \ref{prop.linear}
and Lemma \ref{lem.existvf} that the linearized model is well-posed and that
its solution satisfies uniform a priori estimates. This property is crucial
for homogenization of the linearized model which is performed in Section \ref{Passlimit}.
Following our work \cite{AMP} in the ideal case, we rigorously obtained
the homogenized problem in Theorem \ref{1.15} and derive precise formulas
for the effective tensor in Proposition \ref{prop.eff}.
Furthermore we prove that the so called {\bf Onsager relation} (see e.g. \cite{Ons})
is satisfied, namely the full homogenized tensor $\mathcal{M}$ is symmetric positive definite.

Eventually Section \ref{secnum} is devoted to a numerical study of the obtained homogenized
coefficients, including their sensitivities to various physical parameters and a systematic
comparison with the ideal case.

\section{Non-dimensional form}
\label{sec2}

Before studying its homogenization, we need a {\bf dimensionless} form of the equations
(\ref{AM7})-(\ref{AM4}). We follow the same approach as in our previous works \cite{AMP},
\cite{ABDMP}.
The known data are the characteristic pore size $\ell$, {  the characteristic
domain size $L$,} the surface charge density $\Sigma$
(having the characteristic value $\Sigma_c$), the static electrical potential $\Psi^{ext}$ and
the applied fluid force $\mathbf{f}$. As usual, we introduce a small parameter $\ep$ which
is the ratio between the pore size and the medium size, $\ep = \ell / L <<1$.

We are interested in characteristic concentrations $n_c$ taking on typical values in the range
$(10^{-2}, 1)$ in Mole/liter, that is $( 6.02 \, 10^{24}, 6.02 \, 10^{26})$ particles per $m^{3}$.
From Table \ref{Data}, we write
$\lambda_D= \sqrt{\mathcal{E} k_B T / (e^2 n_c)}$ and we find out that $\lambda_D$
varies in the range $(0.042 , 0.42 )$ nm.

Following \cite{KBA:05}, we introduce the characteristic potential $\zeta= k_B T/ e$
and the  parameter $\beta$ related  to the Debye-H\"uckel parameter $\kappa = 1/ \lambda_D$, as follows
\begin{equation}
\label{def.beta}
\beta = \left(\frac{\ell}{\lambda_D}\right)^2 .
\end{equation}
{  Next we rescale the space variable by setting $x'=x/L$
and $\Omega'=\Omega/L=\prod_{k=1}^d(0,L'_k)^d$ (we shall drop the primes
for simplicity in the sequel). The rescaled dimensions $L'_k$ are assumed
to be independent of $\ep$.
Similarly, the pore space becomes $\Oe= \O_p /L$ which is a periodically
perforated domain with period $\ep$.}
{   Still following \cite{KBA:05}, we define other characteristic quantities}
$$
\Gamma_c =\sqrt{\pi L_B n_c } , \quad
p_c = n_c k_B T , \quad u_c = \ep^2 \frac{k_B T n_c L}{\eta }  ,
$$
{   where $\Gamma_c$, in terms of $n_c$, is deduced from
(\ref{Gamma}), $p_c$ is a pressure equilibrating the electrokinetic forces in
(\ref{AM5}) and $u_c$ is the velocity corresponding to a Poiseuille
flow in a tube of diameter $\ell$, length $L$ and pressure drop $p_c$.
We also introduce adimensionalized forcing terms}
$$
\Psi^{ext,*}= \frac{e\Psi^{ext}}{k_B T} , \quad
\mathbf{f}^* = \frac{\mathbf{f} L}{p_c} , \quad
\Sigma^* = \frac{\Sigma}{\Sigma_c} , \quad
N_\sigma= \frac{e \Sigma_c \ell}{\mathcal{E} k_B T}  ,
$$
and adimensionalized unknowns
$$
\Gamma^\ep = \frac{\Gamma}{\Gamma_c} , \;
p^\ep = \frac{p}{p_c} , \; \mathbf{u}^\ep =\frac{\mathbf{u}}{u_c} , \;
\Psi^\ep = \frac{e\Psi}{k_B T} , \; n_j^\ep = \frac{n_j}{n_c} , \;
{\bf j}_j^\ep = \frac{{\bf j}_j L}{n_c D^0_j} .
$$
The dimensionless equations for hydrodynamical and electrostatic part are thus
\begin{gather}
\ep^2 \Delta \mathbf{u}^\ep - \nabla p^\ep = \mathbf{f}^* +
\sum_{j=1}^N z_j n_j^\ep (x) \nabla\Psi^\ep   \ \mbox{ in } \Oe , \label{EPPR1} \\
\mathbf{u}^\ep =0 \ \mbox{ on } \, \p \Oe \setminus \p\O , \quad
\div \ \mathbf{u}^\ep =0 \quad \mbox{in } \, \Oe, \label{EPPR6}\\
 -\ep^2 \Delta \Psi^\ep = \beta \sum_{j=1}^N  z_j n_j^\ep (x) \quad \mbox{in } \ \Oe ;  \label{EPPR3b}\\
\ep \nabla \Psi^\ep \cdot \nu = -N_\sigma \Sigma^*  \; \mbox{ on } \, \partial\Oe \setminus \partial\O , \label{EPPR2} \\
( \Psi^\ep + \Psi^{ext,*} ) , \quad
\mathbf{u}^\ep \ \mbox{ and } \ p^\ep \quad \mbox{are } \ \Omega-\mbox{periodic in } \; x . \label{EPPR6bc}
\end{gather}
{   (Recall that $\Omega=\prod_{k=1}^d(0,L_k)^d$ so that periodic
boundary conditions make sense for such a rectangular domain.)}
Furthermore, from (\ref{gamma-j}) and (\ref{Gamma}) we define
\begin{gather}
\gamma_j^\ep = \gamma^{HS}_\ep \exp \{ - \frac{L_B \Gamma^\ep \Gamma_c z^2_j}{(1+\Gamma^\ep \Gamma_c \sigma_j)} \} \quad \mbox{ and } \quad (\Gamma^\ep)^2 =
\sum_{k=1}^N \frac{n^\ep_k z_k^2}{(1+ \Gamma_c \Gamma^\ep \sigma_k )^2}.
\label{gamma-eps}
\end{gather}
The solute packing fraction $\xi$ is already an adimensionalized quantity
(taking values in the range $(0,1)$). However, introducing a characteristic
value $\xi_c$ we can adimensionalize its formula (\ref{Hardsphere}) as
\begin{equation}
\label{xic}
\xi_c = \frac{\pi}{6} n_c \sigma_c^3 , \quad
\xi = \xi_c \sum_{k=1}^N n^\ep_k (\frac{\sigma_k}{\sigma_c})^3 ,
\end{equation}
where $\sigma_c$ is the characteristic ion diameter.
We note that $\Gamma_c \in (0.117, 1.17) \, 10^9$ m$^{-1}$,
$\Gamma_c \sigma_c \in (0.023, 0.23)$,
$L_B \Gamma_c \in (0.0857 , 0.857)$ and
$\xi_c \in (0.252, 25.2)\, 10^{-4}$ which is a small parameter.
Concerning ${\bf\Omega}^c_{ij}$ which has to be compared with $D^0_i/(k_B T)$, we find out that
$$
L_B n_c k_B T /(3 \eta \Gamma_c D^0_i ) = \Gamma_c k_B T /(3 \pi \eta  D^0_i ) \in (0.005415 , 0.5415) ,
$$
while ${\bf\Omega}^{HS}_{ij}$ is slightly smaller and $\mathcal{R}_{ij}$ looks negligible.
Concerning the transport term, the Peclet number for the $j$-th species is
$$
\pej = \frac{u_c L}{D^0_j}= \frac{\ell^2 k_B T n_c}{\eta D_j^0} \in (0.01085 , 1.085) .
$$
After these considerations we obtain the dimensionless form of equation (\ref{AM7}):
\begin{gather}
\div \bigg( {\bf j}_i^\ep + \pei n_i^\ep \mathbf{u}^\ep \bigg) =0
\quad \mbox{ in } \quad \Oe , \; i =1, \dots , N , \label{Nernst1}\\
{\bf j}_i^\ep \cdot \nu = 0  \; \mbox{ on } \, \partial\Oe \setminus \partial\O ,
 \; i =1, \dots , N , \label{Nernst1bc} \\
{\bf j}_i^\ep = -\sum_{j=1}^N n_i^\ep K_{ij}^\ep \nabla M_j^\ep \quad  \mbox{ and } \quad
M_j^\ep = \ln \left(n_j^\ep \gamma_j^\ep e^{z_j \Psi^\ep } \right) , \label{electroflux2}\\
K_{ij}^\ep = \bigg( \delta_{ij} + \frac{k_B T}{D_i^0} {\bf\Omega}_{ij} \bigg) \bigg( 1+ \mathcal{R}_{ij} \bigg), \; i,j =1, \dots , N . \label{OffdiagSD}
\end{gather}

{  Eventually the porous medium $\Omega^\ep$ is assumed to be an
$\ep$-periodic smooth open subset of $\Omega=\prod_{k=1}^d(0,L_k)^d$ and $L_k /\ep $ are integers for every $k$ and every $\ep$.
It is built from $\Omega$ by removing a periodic distributions of
solid obstacles which, after rescaling by $1/\ep$, are all similar
to the unit obstacle $Y_S$.}
More precisely, we consider a smooth partition of the unit periodicity
cell $Y=Y_S \cup Y_F$ where $Y_S$ is the solid part and $Y_F$ is the
fluid part. The liquid/solid interface is $S=\partial Y_S \setminus \p Y$.
The fluid part is assumed to be a smooth connected open subset (no
assumption is made on the solid part). We define $Y_\ep^j=\ep (Y_F +j)$
and $\Omega^\ep = \bigcup\limits_{j\in\mathbb Z^d} Y_\ep^j \cap \O$.

We also assume a periodic distribution of charges $\Sigma^* \equiv \Sigma^* (x/ \ep)$.
This will imply that, at equilibrium (in the absence of other forces), the solution
of system (\ref{EPPR1})-(\ref{OffdiagSD}) is also periodic of period $\ep$.


We recall that the ideal model (see e.g. \cite{KBA:05}) corresponds to the following
values of the activity coefficient, $\gamma_i =1$, and of the Onsager tensor
$L_{ij} = \delta_{ij} n_i D_i^0 / (k_B T)$. In view of our previous dimensional
analysis it is interesting to see in which sense the present non-ideal MSA model
is close to the ideal case. We shall make this connection in the limit of a
small parameter going to zero. More precisely we rely on the characteristic value
$\xi_c$ of the solute packing fraction, defined by (\ref{xic}). The smallness of
$\xi_c$ (which means a low concentration, weighted by the ion size)
will be a crucial assumption in Theorem \ref{thm.exist3} that establishes the existence
of equilibrium solutions to the MSA model. It is therefore a natural parameter
to study the limit ideal case.
With this goal in mind we introduce two additional dimensionless numbers:
the Bjerrum's parameter (also called the Landau plasma parameter)
\begin{equation}
bi = \frac{L_B}{\sigma_c}  , \label{bi}
\end{equation}
and the ratio appearing in Stokes' formula for the drag hydrodynamic force
\begin{equation}
S=\frac{k_B T}{\eta D^0_c \sigma_c}  , \label{Sform}
\end{equation}
where $D^0_c$ is the characteristic value for the diffusivities $D^0_i$, $1\leq i\leq N$.
According to the numerical values of Table \ref{Data}, we assume that
\begin{equation}
\label{Assumpordre}
    bi \quad \mbox{and} \quad S \quad \mbox{are of order one}.
\end{equation}
More precisely, it is enough to assume that $bi$ and $S$ are bounded
quantities when $\xi_c$ becomes infinitely small (they can tend to zero too).

\begin{lemma}
\label{lem.ideal}
Under assumption (\ref{Assumpordre}), the ideal case is the limit of our
non-ideal MSA model for small solute packing fraction $\xi_c$, namely
\begin{equation}
K^\ep_{ij} = \delta_{ij} + O(\sqrt{\xi_c}), \quad \mbox{and} \quad
\ln \gamma_j^\ep = O(\sqrt{\xi_c}) .
\label{Approwksi}
\end{equation}
\end{lemma}

Hence the MSA model is a regular $O(\sqrt{\xi_c})$ perturbation of the idealized model.
Theorem \ref{thm.exist3} in Section \ref{sec3} gives the equilibrium MSA solution as an
$O(\sqrt{\xi_c})$ perturbation of the equilibrium idealized solution. The arguments
from Section \ref{secexist} could be extended to interpret the MSA variant of
Poisson-Boltzmann equation as an $O(\sqrt{\xi_c})$ perturbation of the classical
(ideal) Poisson-Boltzmann equation.

\begin{proof}
In view of formula (\ref{Hardsphere}) we find
$$
\ln \gamma^{HS} = O(\xi_c) .
$$
From its definition (\ref{Gamma}) and for small $\xi_c$ we deduce that
$$
\Gamma = O(\sqrt{L_B n_c}) .
$$
Using assumption (\ref{Assumpordre}), $bi=O(1)$, yields
$$
\Gamma \sigma_j = O(\sqrt{bi\xi_c}) = O(\sqrt{\xi_c})
\quad \mbox{and} \quad
\Gamma L_B = O(\sqrt{bi^3\xi_c}) = O(\sqrt{\xi_c}) ,
$$
which implies from (\ref{gamma-j})
$$
-\frac{L_B \Gamma z^2_j}{1+ \Gamma \sigma_j} = O(\sqrt{\xi_c})
\quad \mbox{and thus} \quad
\ln \gamma_j = O(\sqrt{\xi_c}) .
$$
Turning to the Onsager coefficients, we obtain from (\ref{Relaxk}) that
$$
\kappa_q \sigma_c = O( \sqrt{L_B n_c} \sigma_c ) = O(\sqrt{\xi_c}) ,
$$
which implies after some algebra that
$$
\mathcal{R}_{ij} = O( L_B \kappa_q) = O(\sqrt{\xi_c}) .
$$
Using the second part of assumption (\ref{Assumpordre}), $S=O(1)$, yields
$$
\Omega^c_{ij} \frac{k_B T}{D^0_i} = O(\frac{k_B T}{\eta D^0_c \sigma_c}
\sqrt{bi \xi_c}) = O(S \sqrt{bi \xi_c}) = O(\sqrt{\xi_c}) .
$$
Similarly
$$
\Omega^{HS}_{ij} \frac{k_B T}{D^0_i} = O( S \xi_c) = O(\xi_c) ,
$$
which eventually yields
$$
L_{ij}=\frac{n_i D^0_i}{k_B T} (\delta_{ij} + O(\sqrt{\xi_c})) ,
$$
from which we infer the conclusion (\ref{Approwksi}).
Note that a similar approximation holds for the chemical potential
$$
\mu_j = \mu_j^0 + k_B T (\ln n_j + O(\sqrt{\xi_c})).
$$
\end{proof}

\section{Equilibrium solution}
\label{sec3}

The goal of this section is to find a so-called {\bf equilibrium solution}
of system (\ref{EPPR1})-(\ref{OffdiagSD}) when the exterior forces are
vanishing $\mathbf{f}=0$ and $\Psi^{ext}=0$. However, the surface charge
density $\Sigma^*$ is not assumed to vanish or to be small. This equilibrium
solution will be a reference solution around which we shall linearize
system (\ref{EPPR1})-(\ref{OffdiagSD}) in the next section. Then we
perform the homogenization of the (partially) linearized system.
We denote by $n^{0,\ep}_i, \Psi^{0,\ep}, \mathbf{u}^{0,\ep}, M^{0,\ep}_i,
p^{0,\ep}$ the equilibrium quantities.

In the case $\mathbf{f}=0$ and $\Psi^{ext}=0$, one can find an equilibrium
solution by choosing a zero fluid velocity and taking all diffusion fluxes
equal to zero. More precisely, we require
\begin{gather}
\mathbf{u}^{0, \ep} = 0  \quad \mbox{ and } \quad \nabla M^{0,\ep}_j=0 ,
\label{LINN}
\end{gather}
which obviously implies that ${\bf j}_i^{0,\ep}=0$ and (\ref{Nernst1})-(\ref{Nernst1bc})
are satisfied. The Stokes equation (\ref{EPPR1}) shall give the corresponding
value of the pressure satisfying
$$
\nabla p^{0,\ep}(x) = - \sum_{j=1}^N z_j n_j^{0,\ep} (x) \nabla\Psi^{0,\ep}(x) ,
$$
for which an explicit expression is given below (see (\ref{E_j-pres})).
From $\nabla M^{0,\ep}_j=0$ and (\ref{electroflux2}) we deduce
that there exist constants $n_j^0(\infty)>0$ and $\gamma_j^0(\infty)>0$ such that
\begin{gather}
n_j^{0,\ep} (x) = n_j^0(\infty) \gamma_j^0(\infty)
\frac{\exp \{ - z_j \Psi^{0, \ep} (x)  \}}{\gamma_j^{0,\ep}(x)} .
\label{NJ00}
\end{gather}
The value $n_j^0(\infty)$ is {  the reservoir concentration (also
called the infinite dilute concentration) which will be
later assumed to satisfy the bulk electroneutrality condition for zero potential.
The value $\gamma_j^0(\infty)$ is the reservoir} activity coefficient which
corresponds to the value of $\gamma_j^{0,\ep}$
for zero potential (see (\ref{equi.activity}) below for its precise formula).
Before plugging (\ref{NJ00}) into Poisson equation (\ref{EPPR3b})
to obtain the MSA variant of Poisson-Boltzmann equation for the potential
$\Psi^{0,\ep}$, we have to determine the value of the activity coefficient
$\gamma_j^{0,\ep}$.

From the first equation of (\ref{gamma-eps}) we have
\begin{gather}
\gamma_j^{0,\ep} = \gamma^{HS}(\xi)
\exp \{ - \frac{L_B \Gamma^{0,\ep} \Gamma_c z_j^2}{1+ \Gamma^{0,\ep} \Gamma_c \sigma_j} \}
= \exp \{ p(\xi) - \frac{L_B \Gamma^{0,\ep} \Gamma_c z_j^2}{1+ \Gamma^{0,\ep} \Gamma_c \sigma_j} \} ,
\label{NJ001}
\end{gather}
where, for $\xi \in [0,1)$, $p(\xi)$ is a polynomial defined by (\ref{Hardsphere}) and,
recalling definition (\ref{xic}) of the characteristic value $\xi_c$,
the solute packing fraction $\xi$ is
\begin{equation}
\label{gammaeps0}
\xi =  \xi_c \sum_{j=1}^N n_j^{0,\ep} (\frac{\sigma_j}{\sigma_c})^3 .
\end{equation}

The second equation of (\ref{gamma-eps}) gives a formula for the MSA screening
parameter
\begin{equation}
\label{Gammaeps0}
(\Gamma^{0,\ep})^2 =
\sum_{k=1}^N \frac{n^{0,\ep}_k z_k^2}{(1+ \Gamma_c \Gamma^{0,\ep} \sigma_k )^2}.
\end{equation}
Let us explain how to solve the algebraic equations (\ref{NJ00}), (\ref{NJ001}),
(\ref{gammaeps0}) and (\ref{Gammaeps0}).

Combining (\ref{NJ00}), (\ref{NJ001}) and (\ref{gammaeps0}),
for given potential $\Psi^{0,\ep}$ and screening parameter $\Gamma^{0,\ep}$,
the solute packing fraction $\xi$ is a solution of the algebraic equation
\begin{equation}
\label{sol.xi0}
\xi =  \exp \{ -p(\xi) \} \xi_c \sum_{j=1}^N (\frac{\sigma_j}{\sigma_c})^3
n_j^0(\infty) \gamma_j^0(\infty) \exp \left\{ - z_j \Psi^{0,\ep} +
\frac{L_B \Gamma^{0,\ep} \Gamma_c z_j^2}{1+ \Gamma^{0,\ep} \Gamma_c \sigma_j} \right\} .
\end{equation}
Once we know $\xi\equiv\xi(\Psi^{0,\ep},\Gamma^{0,\ep})$, solution of (\ref{sol.xi0}),
combining (\ref{NJ00}) and (\ref{Gammaeps0}), $\Gamma^{0,\ep}$ is a solution of the
following algebraic equation, depending on $\Psi^{0,\ep}$,
\begin{gather}
(\Gamma^{0,\ep})^2 =  \sum^N_{j=1} n_j^0(\infty) \gamma_j^0(\infty)
\frac{z_j^2}{(1+\Gamma^{0,\ep} \Gamma_c \sigma_j)^2}
\exp \left\{ - z_j \Psi^{0,\ep} +
\frac{L_B \Gamma^{0,\ep} \Gamma_c z_j^2}{1+ \Gamma^{0,\ep} \Gamma_c \sigma_j}
- p\left(\xi(\Psi^{0,\ep},\Gamma^{0,\ep})\right) \right\} .
\label{GammaL0}
\end{gather}
All in all, solving the two algebraic equations (\ref{sol.xi0}) and (\ref{GammaL0})
yields the values $\Gamma^{0,\ep}(\Psi^{0,\ep})$ and
$\tilde\xi(\Psi^{0,\ep}) \equiv \xi\Big(\Psi^{0,\ep},\Gamma^{0,\ep}(\Psi^{0,\ep})\Big)$
(see Section \ref{secexist} for a precise statement).

Then the electrostatic equation (\ref{EPPR3b}) reduces to the following MSA variant
of Poisson-Boltzmann equation which is a nonlinear partial differential equation
for the sole unknown $\Psi^{0,\ep}$
\begin{equation}
\label{BP0}
\left\{ \begin{array}{ll}
\dsp -\ep^2 \Delta \Psi^{0,\ep} =\beta  \sum_{j=1}^N z_j n_j^0(\infty) \gamma_j^0(\infty)
\exp \left\{ - z_j \Psi^{0,\ep} + \frac{L_B \Gamma^{0,\ep}(\Psi^{0,\ep}) \Gamma_c z_j^2}{1+ \Gamma^{0,\ep}(\Psi^{0,\ep}) \Gamma_c \sigma_j} -p(\tilde\xi(\Psi^{0,\ep}) ) \right\} \; \mbox{ in } \ \Oe , &  \\
\dsp  \ep \nabla \Psi^{0,\ep} \cdot \nu = -N_\sigma \Sigma^* \ \mbox{ on } \, \partial\Oe \setminus \partial\O ,\
 \Psi^{0,\ep} \; \mbox{ is } \Omega-\mbox{periodic}.&
\end{array} \right.
\end{equation}

In Section \ref{secexist} (see Theorem \ref{thm.main}) we shall prove the following
existence result. Unfortunately we are unable to prove uniqueness.

\begin{theorem}
\label{thm.exist3}
Assuming that the surface charge distribution $\Sigma^*$ belongs
to $L^\infty(\p\Oe)$, that the ions are not too small, namely
\begin{equation}
\label{Bound0}
L_B < (6+4\sqrt{2})  \min_{1\leq j\leq N} \frac{\sigma_j}{z_j^2}
\quad \mbox{ with } \quad 6+4\sqrt{2} \approx 11.656854 ,
\end{equation}
and that the characteristic value $\xi_c$ is small enough,
there exists at least one solution of (\ref{BP0}) $\Psi^{0,\ep} \in H^1(\Oe)\cap L^\infty(\Oe)$.
\end{theorem}

Introducing the primitive $E_j(\Psi)$ of
\begin{equation}
\label{E_j}
E_j^\prime(\Psi) = z_j n^0_j(\infty) \gamma_j^0(\infty) \exp \{ -z_j \Psi +
\frac{L_B \Gamma^{0,\ep} (\Psi ) \Gamma_c z_j^2}{1+ \Gamma^{0,\ep} (\Psi ) \Gamma_c \sigma_j}
-p(\tilde\xi(\Psi)) \} ,
\end{equation}
the equilibrium pressure in Stokes equations (corresponding to a zero velocity) is given
(up to an additive constant) by
\begin{equation}
\label{E_j-pres}
p^{0, \ep} = \sum^N_{j=1} E_j(\Psi^{0,\ep}) .
\end{equation}

\begin{remark}
\label{rem.E_j}
In the ideal case, i.e., when $\gamma^{0,\ep}_j=1$, the function $E_j(\Psi^{0,\ep})$ defined by
(\ref{E_j}) is simply equal to $n_j^{0,\ep}= n^0_j(\infty) \exp \{ -z_j \Psi^{0,\ep} \}$.
\end{remark}

From a physical point of view, it is desired that the solution of (\ref{BP0})
vanishes, i.e., $\Psi^{0,\ep} =0$, when the surface charges are null, i.e.,
$\Sigma^*=0$. Therefore, following the literature,  we impose the
{\bf bulk electroneutrality condition}
\begin{equation}
\label{Neutrality}
\sum_{j=1}^N E^\prime_j(0) = -
\sum_{j=1}^N z_j n^0_j(\infty) \gamma_j^0(\infty) \exp
\{ \frac{L_B \Gamma^{0,\ep}(0) \Gamma_c z_j^2}{1+ \Gamma^{0,\ep}(0) \Gamma_c \sigma_j}
-p(\tilde\xi(0)) \}  = 0,
\end{equation}
where $\Gamma^{0,\ep}(0)$ is the solution of (\ref{GammaL0}) for $\Psi^{0,\ep} =0$.

Defining the equilibrium activity coefficient by
\begin{equation}
\label{equi.activity}
\gamma_j^0(\infty) = \exp
\{ p(\tilde\xi(0)) - \frac{L_B \Gamma^{0,\ep}(0) \Gamma_c z_j^2}{1+ \Gamma^{0,\ep}(0) \Gamma_c \sigma_j} \} ,
\end{equation}
the bulk electroneutrality condition (\ref{Neutrality}) reduces to its usual
form
$$
\sum_{j=1}^N z_j n^0_j(\infty)  = 0 .
$$
Formula (\ref{equi.activity}) is an implicit algebraic equation for $\gamma_j^0(\infty)$
since $\tilde\xi(0)$ and $\Gamma^{0,\ep}(0)$ depend themselves on the $\gamma_k^0(\infty)$'s.
The next Lemma proves that it is a well-posed equation.

\begin{lemma}
\label{lem.gamma0}
There always exists a unique solution $\gamma_j^0(\infty)$
of the algebraic equation (\ref{equi.activity}).
\end{lemma}

\begin{proof}
Assume that there exists $\gamma_j^0(\infty)$ satisfying (\ref{equi.activity})
and plug this formula in (\ref{GammaL0}). It yields
$$
\left(\Gamma^{0,\ep}(0)\right)^2 =  \sum^N_{j=1}
\frac{n_j^0(\infty) z_j^2}{(1+\Gamma^{0,\ep}(0) \Gamma_c \sigma_j)^2} ,
$$
which admits a unique solution $\Gamma^{0,\ep}(0)>0$ since the left hand side is
strictly increasing and the right hand side is decreasing.
On the same token, using (\ref{equi.activity}) in (\ref{sol.xi0})
leads to
$$
\tilde\xi(0) = \xi_c \sum_{j=1}^N (\frac{\sigma_j}{\sigma_c})^3 n_j^0(\infty) .
$$
We have thus found explicit values for $\Gamma^{0,\ep}(0)$ and $\tilde\xi(0)$
which do not depend on the $\gamma_k^0(\infty)$'s. Using them in
(\ref{equi.activity}) gives its unique solution $\gamma_j^0(\infty)$.
\end{proof}

\begin{remark}
\label{rem.xic}
From the proof of Lemma \ref{lem.gamma0} it is clear that $\Gamma^{0,\ep}(0)$
does not depend on $\xi_c$, while $\tilde\xi(0) = O(\xi_c)$, which
implies that $\gamma_j^0(\infty) = O(1)$ for small $\xi_c$.
\end{remark}

\begin{remark}
The bulk electroneutrality condition (\ref{Neutrality}) is not
a restriction. Actually all our results hold under
the much weaker assumption (\ref{valence}) that all valencies $z_j$
do not have the same sign. Indeed, if (\ref{Neutrality}) is
not satisfied, we can make a change of variables in the
Poisson-Boltzmann equation (\ref{BP0}), defining a new
potential $\tilde\Psi^{0,\ep}=\Psi^{0,\ep}+C$ where $C$ is a
constant reference potential. Since the function
$$
C \ \to \ \Phi(C) =\sum_{j=1}^N z_j n^0_j(\infty) \gamma_j^0(\infty) \exp \{ - z_j C +
\frac{L_B \Gamma^{0,\ep}(C) \Gamma_c z_j^2}{1+ \Gamma^{0,\ep}(C) \Gamma_c \sigma_j}
- p(\tilde\xi(C))\}
$$
is continuous and admits opposite infinite limits when $C$
tends to $\pm$, there exists at least one value $C$ such that $\Phi(C) =0$.
This change of variables for the potential leaves (\ref{GammaL0}) and
(\ref{BP0}) invariant if we change the constants $n_j^0(\infty) \gamma_j^0(\infty)$ in new
constants
$$
\tilde n_j^0(\infty) \tilde \gamma_j^0(\infty) = n_j^0(\infty) \gamma_j^0(\infty) \exp \{ - z_j C +
\frac{L_B \Gamma^{0,\ep}(C) \Gamma_c z_j^2}{1+ \Gamma^{0,\ep}(C) \Gamma_c \sigma_j} -
\frac{L_B \Gamma^{0,\ep}(0) \Gamma_c z_j^2}{1+ \Gamma^{0,\ep}(0) \Gamma_c \sigma_j}
- p(\tilde\xi(C)) + p(\tilde\xi(0))\} .
$$
These new constants satisfy the bulk electroneutrality condition (\ref{Neutrality}).
\end{remark}

\section{Linearization}
\label{sec4}

We now proceed to the linearization of electrokinetic equations (\ref{EPPR1}-\ref{OffdiagSD})
around the equilibrium solution computed in Section \ref{sec3}. We therefore assume
that the external forces, namely the static electric potential $\Psi^{ext}(x)$ and
the hydrodynamic force $\mathbf{f}(x)$, are small. However, the surface charge density
$\Sigma^*$ on the pore walls is not assumed to be small since it is part of the
equilibrium problem studied in Section \ref{sec3}. Such a linearization process
is classical in the ideal case (see the seminal paper \cite{OBW:78} by O'Brien et al.)
but it is new and slightly more complicated for the MSA model.
For small exterior forces, we write the perturbed electrokinetic unknowns as
\begin{gather}
n_i^\ep (x) = n^{0,\ep}_i (x) + \delta n_i^\ep (x), \quad  \Psi^\ep (x) = \Psi^{0,\ep} (x) + \delta \Psi^\ep (x), \notag \\
\mathbf{u}^\ep(x) = \mathbf{u}^{0, \ep}(x) + \delta\mathbf{u}^\ep(x), \quad  p^\ep (x) = p^{0, \ep} (x) + \delta p^\ep (x), \notag
\end{gather}
where $n^{0,\ep}_i, \Psi^{0,\ep}, \mathbf{u}^{0,\ep}, p^{0,\ep}$ are the equilibrium quantities, corresponding
to $\mathbf{f}=0$ and $\Psi^{ext}=0$. The $\delta$ prefix indicates a perturbation.
Since the equilibrium velocity vanishes $\mathbf{u}^{0,\ep}=0$, we identify in the sequel
$\mathbf{u}^\ep=\delta \mathbf{u}^\ep$.

Motivated by the form of the Boltzmann equilibrium distribution and the calculation of $n_i^{0,\ep}$, we follow the lead of \cite{OBW:78} and introduce a so-called ionic potential $\Phi_i^\ep$ which is defined in terms of $n_i^\ep$ by
\begin{equation}
\label{BP2}
n_i^\ep (x) \gamma_i^\ep (x) = n_i^0(\infty) \gamma_j^0(\infty) \exp \{ - z_i (\Psi^\ep (x)  + \Phi_i^\ep (x) + \Psi^{ext,*}(x) )  \} ,
\end{equation}
where
\begin{equation}
\label{BP2b}
\gamma_i^\ep = \exp \{ p(\xi) - \frac{L_B \Gamma^\ep \Gamma_c z^2_i}{1+\Gamma^\ep \Gamma_c \sigma_i} \}
\quad \mbox{ and } \quad
(\Gamma^\ep)^2 = \sum_{k=1}^N \frac{n_k^\ep z_k^2}{(1+ \Gamma^\ep \Gamma_c \sigma_k )^2} ,
\end{equation}
with
$$
p(\xi) = \xi \frac{8-9\xi + 3\xi^2}{(1-\xi)^3}  \quad \mbox{ and } \quad
\xi = \xi_c \sum_{j=1}^N n_j^{\ep} (\frac{\sigma_j}{\sigma_c})^3  .
$$
Since $\Phi_i^{0,\ep}=0$ by virtue of formula (\ref{NJ00}) for $n_i^{0,\ep}$,
we identify $\delta\Phi_i^\ep$ with $\Phi_i^\ep$.

\begin{lemma}
\label{lem.lin}
The linearization of (\ref{BP2}-\ref{BP2b}) yields
\begin{gather}
\label{BP301}
\delta n_i^\ep(x) = \sum_{k=1}^N z_k \alpha_{i,k}^{0,\ep}(x)
\Big( \delta \Psi^\ep(x) + \Phi_k^\ep(x) + \Psi^{ext,*}(x) \Big) ,
\end{gather}
with
\begin{gather}
\label{BP302}
\alpha_{i,k}^{0,\ep} = - n_i^{0,\ep} \delta_{ik} + \Be n_i^{0,\ep} n_k^{0,\ep} \sigma_k^3
- \frac{L_B \Gamma_c}{\Ae}  n_i^{0,\ep} n_k^{0,\ep} \\
\times
\left( \frac{z^2_i}{(1+\Gamma^{0,\ep} \Gamma_c \sigma_i)^2} -\Be \Ce \right)
\left( \frac{z^2_k}{(1+\Gamma^{0,\ep} \Gamma_c \sigma_k)^2} -\Be \De \sigma_k^3 \right) ,  \notag
\end{gather}
where
$$
\Be = \frac{\frac{\pi}{6} n_c p^\prime(\xi)}{1 + \xi p^\prime(\xi)}
\ , \quad
\Ce = \sum_{k=1}^N \frac{z^2_k \sigma_k^3 n_k^{0,\ep}}{(1+\Gamma^{0,\ep} \Gamma_c \sigma_k)^2}
\ , \quad
\De = \sum_{k=1}^N \frac{z^2_k n_k^{0,\ep}}{(1+\Gamma^{0,\ep} \Gamma_c \sigma_k)^2}
$$
and
\begin{gather}
\Ae = 2 \Gamma^{0,\ep} + 2\Gamma_c \sum_{k=1}^N \frac{n_k^{0,\ep} z_k^2 \sigma_k}{(1+ \Gamma^{0,\ep} \Gamma_c \sigma_k )^3} - L_B \Gamma_c \sum_{k=1}^N \frac{n_k^{0,\ep} z_k^4}{(1+ \Gamma^{0,\ep} \Gamma_c \sigma_k )^4}  \notag \\
+ L_B \Gamma_c \Be \Ce \sum_{k=1}^N \frac{n_k^{0,\ep} z_k^2}{(1+ \Gamma^{0,\ep} \Gamma_c \sigma_k )^2} . \notag
\end{gather}
Under assumption (\ref{Bound0}) of Theorem \ref{thm.exist3} the coefficient $\Ae$ is positive.

Furthermore, at equilibrium, if we consider $n_i^{0,\ep}$ as a function of $\Psi^{0,\ep}$,
we have
\begin{gather}
\label{BP311}
\frac{d n_i^{0,\ep}}{d \Psi^{0,\ep}} = \sum_{k=1}^N z_k \alpha_{i,k}^{0,\ep} .
\end{gather}
If all ions have the same diameter ($\sigma_k=\sigma_i$ for all $i,k$), then the
coefficients $\alpha_{i,k}^{0,\ep}=\alpha_{k,i}^{0,\ep}$ are symmetric.
\end{lemma}

\begin{remark}
At equilibrium, the concentrations $n_i^{0,\ep}$, as well as the screening
parameter $\Gamma^{0,\ep}$, depend
only on $\Psi^{0,\ep}$ through the
algebraic equations (\ref{NJ00}), (\ref{NJ001}),
(\ref{gammaeps0}) and (\ref{Gammaeps0}). However, outside equilibrium the
concentrations $n_i^{\ep}$ and the screening parameter $\Gamma^{\ep}$
depend through (\ref{BP2}-\ref{BP2b}) on the entire family
$( \delta \Psi^\ep + \Phi_k^\ep + \Psi^{ext,*} )$, $1\leq k\leq N$.
\end{remark}

\begin{proof}
Linearizing (\ref{BP2}) leads to
$$
\delta n_i^\ep = \frac{-\delta \gamma_i^\ep}{(\gamma_i^{0,\ep})^2} n_i^0(\infty) \gamma_j^0(\infty)
\exp \{ - z_i \Psi^{0,\ep} \} - \frac{z_i n_i^0(\infty)\gamma_j^0(\infty)}{\gamma_i^{0,\ep}}
\exp \{ - z_i \Psi^{0,\ep} \} \Big( \delta \Psi^\ep + \Phi_i^\ep + \Psi^{ext,*} \Big)
$$
which is equivalent to
\begin{gather}
\label{BP303}
\delta n_i^\ep = \frac{-\delta \gamma_i^\ep}{\gamma_i^{0,\ep}} n_i^{0,\ep}
- z_i n_i^{0,\ep} \Big( \delta \Psi^\ep + \Phi_i^\ep + \Psi^{ext,*} \Big) .
\end{gather}
Linearization of the first equation of (\ref{BP2b}) yields
\begin{gather}
\label{BP304}
\frac{\delta \gamma_i^\ep}{\gamma_i^{0,\ep}}
= p^\prime(\xi) \delta\xi - \frac{L_B \Gamma_c z^2_i}{(1+\Gamma^{0,\ep} \Gamma_c \sigma_i)^2} \delta \Gamma^\ep
\quad \mbox{ with } \quad \delta\xi = \frac{\pi}{6} n_c \sum_{k=1}^N \sigma_k^3 \delta n_k^\ep .
\end{gather}
{  Multiplying (\ref{BP303}) by $\sigma_i^3$ and (\ref{BP304})
by $\sigma_i^3n_i^{0,\ep}$, then summing up to eliminate
$\delta \gamma_i^\ep / \gamma_i^{0,\ep}$, gives}
$$
\left( \sum_{k=1}^N \sigma_k^3 \delta n_k^\ep \right)
\left( 1 + \frac{\pi}{6} n_c p^\prime(\xi)  \sum_{k=1}^N \sigma_k^3 n_k^{0,\ep} \right)
= - \sum_{k=1}^N \sigma_k^3 z_k n_k^{0,\ep} \Big( \delta \Psi^\ep + \Phi_k^\ep + \Psi^{ext,*} \Big)
+ \delta \Gamma^\ep \sum_{k=1}^N \frac{L_B \Gamma_c z^2_k \sigma_k^3 n_k^{0,\ep}}{(1+\Gamma^{0,\ep} \Gamma_c \sigma_k)^2} ,
$$
from which, together with (\ref{BP303}), we deduce
\begin{gather}
\label{BP3}
\delta n_i^\ep(x) = - z_i n_i^{0,\ep}(x)
\Big( \delta \Psi^\ep(x) + \Phi_i^\ep(x) + \Psi^{ext,*}(x) \Big)
+ \frac{L_B \Gamma_c z^2_i n_i^{0,\ep}(x)}{(1+\Gamma^{0,\ep}(x) \Gamma_c \sigma_i)^2} \delta \Gamma^\ep(x) \\
+ \Be n_i^{0,\ep}(x) \sum_{k=1}^N \sigma_k^3 z_k n_k^{0,\ep}(x) \Big( \delta \Psi^\ep(x) + \Phi_k^\ep(x) + \Psi^{ext,*}(x) \Big)
- \Be \Ce n_i^{0,\ep}(x) L_B \Gamma_c \delta \Gamma^\ep(x) . \notag
\end{gather}
Next, we linearize the second formula of (\ref{BP2b}) to obtain
$$
2 \Gamma^{0,\ep} \delta\Gamma^\ep = \sum_{k=1}^N \left(
\frac{z_k^2 \delta n_k^\ep}{(1+ \Gamma^{0,\ep} \Gamma_c \sigma_k )^2} -
\frac{2 n_k^{0,\ep} z_k^2 \sigma_k \Gamma_c}{(1+ \Gamma^{0,\ep} \Gamma_c \sigma_k )^3}
\delta\Gamma^\ep \right) .
$$
Combining it with (\ref{BP3}) leads to
\begin{gather}
\label{BP3b}
\Ae(x) \, \delta \Gamma^\ep(x)  = -
\sum_{k=1}^N \frac{n_k^{0,\ep}(x) z_k^3}{(1+ \Gamma^{0,\ep}(x) \Gamma_c \sigma_k )^2}
\Big( \delta \Psi^\ep(x) + \Phi_k^\ep(x) + \Psi^{ext,*}(x) \Big) \\
+ \Be(x) \De(x) \sum_{k=1}^N n_k^{0,\ep}(x) z_k\sigma_k^3
\Big( \delta \Psi^\ep(x) + \Phi_k^\ep(x) + \Psi^{ext,*}(x) \Big) . \notag
\end{gather}
Eventually, plugging (\ref{BP3b}) into (\ref{BP3}) yields (\ref{BP301}) and (\ref{BP302}).

Since we divide by $\Ae$ we check that it does not vanish in some range of the
physical parameters. Using  definition
(\ref{BP2b}) of $(\Gamma^{0,\ep})^2$ in the equality $2\Gamma^{0,\ep}=2(\Gamma^{0,\ep})^2/\Gamma^{0,\ep}$
allows us to rewrite the coefficient $\Ae$ as
\begin{gather}
\Ae(x) = \Gamma_c \sum_{k=1}^N \frac{n_k^{0,\ep} z_k^2}{(1+ \Gamma^{0,\ep} \Gamma_c \sigma_k)^2}
\left( \frac{2}{\Gamma^{0,\ep} \Gamma_c} + \frac{2 \sigma_k}{(1+ \Gamma^{0,\ep} \Gamma_c \sigma_k)}
- \frac{L_B z_k^2}{(1+ \Gamma^{0,\ep} \Gamma_c \sigma_k )^2} \right) \notag \\
+ L_B \Gamma_c \Be \Ce \sum_{k=1}^N \frac{n_k^{0,\ep} z_k^2}{(1+ \Gamma^{0,\ep} \Gamma_c \sigma_k )^2} , \notag
\end{gather}
where each term in the sum of the first line is positive under the same condition (\ref{Bound0})
and same proof as in Lemma \ref{lem.gamma2}.

The computation leading to (\ref{BP311}) is completely similar. Finally, the symmetry
relation $\alpha_{i,k}^{0,\ep}=\alpha_{k,i}^{0,\ep}$ is obvious from formula (\ref{BP302})
when $\sigma_k=\sigma_i$ for all $i,k$.
\end{proof}

\begin{remark}
In the ideal case, $\gamma_i^\ep\equiv1$, Lemma \ref{lem.lin} simplifies
a lot since $\alpha_{i,k}^{0,\ep}= - n_i^{0,\ep} \delta_{ik}$ which implies
there is no coupling between the various ionic potentials in the definition
of a single species concentration.
\end{remark}

Thanks to the definition (\ref{BP2}) of the ionic potential,
the linearization of the convection-diffusion equation (\ref{Nernst1}) is easy
because the diffusive flux simplifies as
$$
M_j^\ep = \ln \left(n_j^\ep \gamma_j^\ep e^{z_j \Psi^\ep } \right)
= \ln \left( n_j^0(\infty)\gamma_j^0(\infty) \right) - z_j ( \Phi_j^\ep + \Psi^{ext,*} ) .
$$
Furthermore, the equilibrium solution satisfies $\nabla M_j^{0,\ep} = 0$, which implies
\begin{gather}
\mbox{div } \bigg( \sum_{j=1}^N n_i^{0,\ep} K_{ij}^{0,\ep} z_j \nabla (\Phi_j^\ep + \Psi^{ext,*} )
+ \pei n_i^{0,\ep} \mathbf{u}^\ep  \bigg) =0 \quad \mbox{ in } \quad \Oe
\label{lin1}\\
K_{ij}^{0,\ep} = \bigg( \frac{D_i^0}{k_B T} \delta_{ij} + \Omega_{ij}^0 \bigg) \bigg( 1+ \mathcal{R}_{ij}^0 \bigg), \; i,j =1, \dots , N . \label{lin2}
\end{gather}
The linearization of the Stokes equation (\ref{EPPR1}) is more tricky.
We first get
\begin{gather}
\ep^2 \Delta \mathbf{u}^\ep - \nabla \delta p^\ep = \mathbf{f}^* +
\sum_{j=1}^N z_j \left( \delta n_j^\ep \nabla\Psi^{0,\ep} + n_j^{0,\ep} \nabla\delta\Psi^\ep \right) \; \mbox{in} \; \O^\ep , \label{lin3} \\
\div \mathbf{u}^\ep =0 \quad \mbox{ in }  \Omega^\ep , \quad
\mathbf{u}^\ep =0 \ \mbox{ on } \p \Omega^\ep \setminus \p \O . \notag
\end{gather}
We rewrite the sum on the right hand side of (\ref{lin3}) as
\begin{gather}
\label{lin71}
\nabla \left( \sum_{j=1}^N z_j n_j^{0,\ep} \delta\Psi^\ep \right) + S^\ep
\quad \mbox{ with } \quad S^\ep = \sum_{j=1}^N z_j
\left( \delta n_j^\ep \nabla\Psi^{0,\ep} - \delta\Psi^\ep \nabla n_j^{0,\ep} \right) .
\end{gather}
Since $\nabla n_j^{0,\ep} = \frac{d n_j^{0,\ep}}{d \Psi^{0,\ep}} \nabla\Psi^{0,\ep}$
at equilibrium, from Lemma \ref{lem.lin} we deduce
\begin{gather}
\label{lin72}
S^\ep = \sum_{j,k=1}^N z_j \left( z_k \alpha_{j,k}^{0,\ep}
\Big( \delta \Psi^\ep + \Phi_k^\ep + \Psi^{ext,*} \Big)
- z_k \alpha_{j,k}^{0,\ep} \delta\Psi^\ep  \right) \nabla\Psi^{0,\ep} \notag \\
= \sum_{j,k=1}^N z_j z_k \alpha_{j,k}^{0,\ep}
\Big( \Phi_k^\ep + \Psi^{ext,*} \Big)  \nabla\Psi^{0,\ep}
\end{gather}
If all ions have the same diameter, the coefficients $\alpha_{j,k}^{0,\ep}$ are
symmetric, i.e. $\alpha_{j,k}^{0,\ep} = \alpha_{k,j}^{0,\ep}$, so we deduce
$$
S^\ep = \sum_{k=1}^N z_k \frac{d n_k^{0,\ep}}{d \Psi^{0,\ep}}
\Big( \Phi_k^\ep + \Psi^{ext,*} \Big)  \nabla\Psi^{0,\ep}
= \sum_{k=1}^N z_k \Big( \Phi_k^\ep + \Psi^{ext,*} \Big) \nabla n_k^{0,\ep} .
$$
Thus, we rewrite (\ref{lin3}) as
\begin{gather}
\ep^2 \Delta \mathbf{u}^\ep - \nabla P^\ep = \mathbf{f}^* -
\sum_{j=1}^N z_j n_j^{0,\ep} \nabla \left( \Phi_j^\ep + \Psi^{ext,*} \right) , \label{lin4}
\end{gather}
where the new pressure $P^\ep$ is defined by
$$
P^\ep = \delta p^\ep +
\sum_{j=1}^N z_j n_j^{0,\ep} \left( \delta\Psi^\ep + \Phi_j^\ep + \Psi^{ext,*} \right)  .
$$

\begin{remark}
When the ion diameters are different, we can merely introduce
nonlinear functions $F_j$ (defined by their derivatives) such that
$$
S^\ep = \sum_{j=1}^N z_j \left( \Phi_j^\ep + \Psi^{ext,*} \right)
\nabla \Big( F_j(\Psi^{0,\ep}) \Big) .
$$
In general it is not clear whether $F_j=n_j$.
\end{remark}

Of course, one can deduce a linearized equation for $\delta\Psi^\ep$ from
the non-linear Poisson equation (\ref{EPPR3b}) too. But, since $\delta\Psi^\ep$
does not enter the previous equations (upon redefining the pressure $P^\ep$),
it is {\bf decoupled} from the main unknowns $\mathbf{u}^\ep$, $P^\ep$ and $\Phi_i^\ep$.
Therefore it is not necessary to write its equation in details.

To summarize, we have just proved the following result.

\begin{proposition}
\label{prop.linear}
Assume that all ions have the same diameter.
The linearized system, around the equilibrium solution of Section \ref{sec3},
of the electrokinetic equations (\ref{EPPR1}-\ref{OffdiagSD}) is
\begin{gather}
\ep^2 \Delta \mathbf{u}^\ep - \nabla P^\ep = \mathbf{f}^* -
\sum_{j=1}^N z_j n_j^{0,\ep} \nabla \left( \Phi_j^\ep + \Psi^{ext,*} \right) \;
\mbox{in} \; \O^\ep , \label{lin4b} \\
\div \mathbf{u}^\ep =0 \quad \mbox{ in }  \Omega^\ep , \quad
\mathbf{u}^\ep =0 \ \mbox{ on } \p \Omega^\ep \setminus \p \O , \label{lin4c} \\
\div \, n_i^{0,\ep} \bigg( \sum_{j=1}^N K_{ij}^{0,\ep} z_j \nabla (\Phi_j^\ep + \Psi^{ext,*} )
+ \pei \mathbf{u}^\ep  \bigg) =0 \quad \mbox{ in } \Oe , \ i=1, \dots , N , \label{lin1b}\\
\sum_{j=1}^N  K_{ij}^{0,\ep} z_j \nabla (\Phi_j^\ep + \Psi^{ext,*} ) \cdot \nu = 0
\ \mbox{ on } \p \Omega^\ep \setminus \p \O , \label{bc.phijeps} \\
\mathbf{u}^\ep , \ P^\ep , \ \Phi_j^\ep \ \mbox{ are } \; \Omega-\mbox{periodic} , \label{LIN4}
\end{gather}
where the coefficients $n_j^{0,\ep}$ and $K_{ij}^{0,\ep}$ (defined by (\ref{lin2}))
are evaluated at equilibrium.
\end{proposition}

This is the system of equations that we are going to homogenize in the next
sections. It is the extension to the non-ideal case of a similar ideal system
previously studied in \cite{AMP}, \cite{ABDMP}, \cite{AM:03}, \cite{A:01},
\cite{CSTA:96}, \cite{GCA:06}, \cite{MSA:01}, \cite{RPA:06}, \cite{LC:06}.
The mathematical structure of system (\ref{lin4b})-(\ref{LIN4}) is essentially
the same as in the ideal case. The only difference is the coupling
of the diffusion equations through the tensor $K_{ij}^{0,\ep}$.
Note that the tensor $K_{ij}^{0,\ep}$ is related to the original Onsager tensor
$L_{ij}$, defined in (\ref{Offdiag}): upon adimensionalization and evaluation
at equilibrium, $L_{ij}$ {  becomes $L_{ij}^{0,\ep} = n_i^{0,\ep} K_{ij}^{0,\ep} D_i^0 /(k_B T)$.
In particular, the tensor $L_{ij}^{0,\ep}$ inherits from the symmetry of $L_{ij}$
(it is thus symmetric positive definite).}

Next, we establish the variational formulation of (\ref{lin4b})-(\ref{LIN4})
and prove that it admits a unique  solution.
The functional spaces related to the velocity field are
$$
W^\ep =\{ \mathbf{v} \in H^1 (\Oe)^d , \; \mathbf{v}=0 \mbox{ on } \p \Oe \setminus \p\O , \;
\Omega-\mbox{periodic in } \; x  \}
$$
and
$$
H^\ep = \{ \mathbf{v} \in W^\ep , \; \div \,\mathbf{v}=0 \; \hbox{ in } \; \Oe \} .
$$
The variational formulation of (\ref{lin4b})-(\ref{LIN4}) is:
find $\mathbf{u}^\ep \in H^\ep$ and $\{ \Phi_j^\ep \}_{j=1, \dots , N}  \in H^1 (\Oe)^N$,
$\Phi_j^\ep$ being $\Omega$-periodic, such that, for any test functions $\mathbf{v} \in H^\ep$ and
$\{ \phi_j \}_{j=1, \dots , N} \in H^1 (\Oe)^N$, $\phi_j$ being $\Omega$-periodic,
$$
a\left( ( \mathbf{u}^\ep , \{ \Phi_j^\ep \} ) , ( \mathbf{v}, \{ \phi_j \} ) \right) =
\langle \mathcal{L}, ( \mathbf{v}, \{ \phi_j \} ) \rangle ,
$$
where the bilinear form $a$ and the linear form $\mathcal{L}$ are defined by
\begin{gather}
a\left( ( \mathbf{u}^\ep , \{ \Phi_j^\ep \} ) , ( \mathbf{v}, \{ \phi_j \} ) \right) := \ep^2 \int_{\Oe} \nabla \mathbf{u}^\ep : \nabla \mathbf{v} \, dx
+ \sum_{i,j=1}^N \frac{z_iz_j}{\pei} \int_{\Oe} n_i^{0,\ep} K_{ij}^{0,\ep} \nabla \Phi_j^\ep \cdot \nabla \phi_i \, dx
\notag \\
 + \sum_{j=1}^N z_j \int_{\Oe} n_j^{0,\ep} \big( \mathbf{u}^\ep \cdot \nabla \phi_j -
\mathbf{v} \cdot \nabla \Phi_j^\ep \big) \ dx  \notag \\
\langle \mathcal{L}, ( \mathbf{v}, \{ \phi_j \} ) \rangle :=
\sum_{j=1}^N z_j \int_{\Oe} n_j^{0,\ep} \mathbf{E}^* \cdot \mathbf{v} \, dx
-\! \sum_{i,j=1}^N \frac{z_iz_j}{\pei}  \int_{\Oe} n_i^{0,\ep} K_{ij}^{0,\ep} \mathbf{E}^*
\cdot \nabla \phi_i \, dx -\! \int_{\Oe} \mathbf{f}^*\cdot\mathbf{v} \, dx ,
\label{VAREP}
\end{gather}
where, for simplicity, we denote by $\mathbf{E}^*$ the electric field corresponding
to the potential $\Psi^{ext,*}$, i.e.,
$
\mathbf{E}^*(x) = \nabla \Psi^{ext,*}(x) .
$

\begin{lemma}
\label{lem.existvf}
For sufficiently small values of $n_c>0$ and $\xi_c>0$, and under
assumption (\ref{Bound0}), there exists a
unique solution of (\ref{lin4b})-(\ref{LIN4}), $\mathbf{u}^\ep \in H^\ep$ and
$\{ \Phi_j^\ep \}_{j=1, \dots , N} \in H^1(\Oe)^N$, $\Phi_j^\ep$ being $\Omega$-periodic.
Furthermore, there exists a positive constant $C$, independent of $\ep$, such that
\begin{gather}
\| \mathbf{u}^\ep \|_{L^2 (\Oe)^d} +
\ep \|\nabla\mathbf{u}^\ep \|_{L^2 (\Oe)^{d^2}} +
\max_{1\leq j\leq N} \| \Phi_j^\ep \|_{H^1 (\Oe)} \leq C
\bigg( \| \mathbf{E}^* \|_{L^2 (\O )^d} + \| \mathbf{f}^* \|_{L^2 (\O )^d} \bigg).
\label{Apriori}
\end{gather}
\end{lemma}

\begin{proof}
Assumption (\ref{Bound0}) and $\xi_c>0$ small implies that the potential $\Psi^{0,\ep}$
is bounded in $L^\infty(\Oe)$ (see Theorem \ref{thm.main}).
The same holds true for $\xi$ and $\Gamma^{0,\ep}$
which are algebraic functions of $\Psi^{0,\ep}$. Thus, the concentrations
$n_j^{0,\ep}$, defined by (\ref{NJ0HS}) are uniformly positive and bounded
in $L^\infty(\Oe)$. Due to the structure of ${\bf\Omega}^c_{ij}$, ${\bf\Omega}^{HS}_{ij}$
and $\mathcal{R}_{ij}$, these coefficients, evaluated at equilibrium, are arbitrary
small in $L^\infty(\Oe)$ for small $n_c$. Consequently, the tensor $K_{ij}^{0,\ep}$
is positive definite (as a perturbation of the identity) and the bilinear form $a$
is coercive for $n_c \leq n_c^{\rm cr}$. The rest of the proof, including the
a priori estimates, is similar to
the ideal case, studied in \cite{AMP}, where we had $K_{ij}^{0,\ep}=\delta_{ij}$.
\end{proof}

\section{Homogenization}
\label{Passlimit}

In the previous sections \ref{sec3} and \ref{sec4} we did not use our
assumption that the porous medium and the surface charge distribution
are $\ep$-periodic (see the end of
section \ref{sec2}). Our further analysis relies  crucially  on this $\ep$-periodicity
hypothesis. Theorem \ref{thm.exist3} gives the existence of a solution
to the Poisson-Boltzmann equation {  (\ref{BP0})} but not its uniqueness.
Nevertheless, we can define a particular solution of (\ref{BP0}), which
is $\ep$-periodic,
\begin{equation}
\label{eq.psi0}
\Psi^{0,\ep} (x) = \Psi^{0}(\frac{x}{\ep}) ,
\end{equation}
where $\Psi^{0}(y)$ is a solution of the unit cell Poisson-Boltzmann
equation
\begin{equation}
\label{BP122}
\left\{ \begin{array}{ll}
\dsp - \Delta_y \Psi^{0}(y) = \beta \sum_{j=1}^N z_j n^0_j(y) & \mbox{ in } Y_F , \\
\dsp \nabla_y \Psi^{0} \cdot \nu = -N_\sigma \Sigma^*(y) & \mbox{ on } \partial Y_F
\setminus \partial Y ,\\
\dsp y \to \Psi^{0}(y) \; \mbox{ is } 1-\mbox{periodic,} & \\
\dsp n^0_j(y) = n_j^0(\infty)\gamma_j^0(\infty) \frac{\exp\left\{ - z_j \Psi^{0}(y) \right\} }{\gamma^0_j(y)} ,&
\end{array} \right.
\end{equation}
with the activity coefficient defined by
$$
\gamma^0_j(y) = \gamma^{HS}(y) \exp \{ - \frac{L_B \Gamma^0(y) \Gamma_c z^2_j}{(1+\Gamma^0(y) \Gamma_c \sigma_j)} \} \quad \mbox{ and } \quad (\Gamma^0(y))^2 =
\sum_{k=1}^N \frac{n^0_k(y) z_k^2}{(1+ \Gamma_c \Gamma^0(y) \sigma_k )^2} ,
$$
and
$$
\gamma^{HS} = \exp \{ p(\xi) \} \quad \mbox{ with } \quad
p(\xi) = \xi \frac{8-9\xi + 3\xi^2}{(1-\xi)^3}  \quad \mbox{ and } \quad
\xi(y) = \frac{\pi n_c}{6} \sum_{k=1}^N n_k^0(y) \sigma_k^3  .
$$

The formal two-scale asymptotic expansion method \cite{blp}, \cite{hornung},
\cite{SP80} can be applied to system (\ref{lin4b})-(\ref{LIN4}) as in
the ideal case studied by \cite{LC:06}, \cite{MM:02}, \cite{MM:03},
\cite{MM:06}, \cite{MM:08}, \cite{AMP} and \cite{ABDMP}.
Introducing the fast variable $y=x/\ep$, it assumes that the solution of
(\ref{lin4b})-(\ref{LIN4}) is given by
{
\begin{equation}
\label{eq.ansatz}
   \left\{
     \begin{array}{l}
       \dsp \mathbf{u}^\ep (x) = \mathbf{u}^0 (x,x/\ep) + \ep \mathbf{u}^1 (x,x/\ep) +\dots ,  \\
       \dsp P^\ep (x) = p^0 (x) + \ep p^1 (x,x/\ep) + \dots , \\
       \dsp \Phi^\ep_j (x) = \Phi^0_j (x) + \ep \Phi^1_j (x,x/\ep) +\dots .
     \end{array}
   \right.
\end{equation}
}
We then plug this ansatz in the equations (\ref{lin4b})-(\ref{LIN4}) and
use the chain-rule lemma for a function $\phi(x,\frac{x}{\ep})$
$$
\nabla \left( \phi(x,\frac{x}{\ep}) \right) =
\left( \nabla_x \phi + \frac{1}{\ep} \nabla_y \phi \right) (x,\frac{x}{\ep}) .
$$
Identifying the various powers of $\ep$ we obtain a cascade of
equations from which we retain only the first ones that constitute
the following two-scale homogenized problem. This type of calculation
is classical and we do not reproduce it here. It can be made rigorous
thanks to the notion of two-scale convergence \cite{All92}, \cite{NGU}.

\begin{proposition}
\label{thm.2sc}
From each bounded sequence $\{w^\ep\}$ in $L^2(\Omega^\ep)$ one can extract a
subsequence which two-scale converges to a limit $w\in L^2(\Omega\times Y_F)$
in the sense that
$$
\lim_{\ep\rightarrow 0}\int_{\Omega^\ep} w^\ep(x)\varphi\left(x, \frac{x}{\ep}\right)\, dx
= \int_{\Omega} \int_{Y_F} w(x,y)\varphi(x,y)\,dy\,dx
$$
for any $\varphi\in L^2\big(\Omega;C_{\rm per}(Y)\big)$ (``per'' denotes $1$-periodicity).
\end{proposition}

{   For sequences of functions ${w^\ep}$  defined in the perforated domain $\Omega^\ep$ and satisfying uniform in $\ep$  $H^1$-bounds, it is well-known \cite{hornung} that
one can build extensions to the entire domain $\Omega$ satisfying
the same uniform bounds. We implicitly assume such extensions in
the theorem below but do not give details which are classical and
may be found in \cite{AMP}.}

\begin{theorem}
\label{1.15}
Under the assumptions of Lemma \ref{lem.existvf} the solution
of (\ref{lin4b})-(\ref{LIN4}) converges in the following sense
\begin{gather*}
\mathbf{u}^\ep \to  \mathbf{u}^0 (x,y)  \quad \hbox{ in the two-scale sense } \label{1.69} \\
\ep \nabla \mathbf{u}^\ep  \to  \nabla_y \mathbf{u}^0
 (x,y)  \quad \hbox{ in the two-scale sense }
\label{1.70}\\
P^\ep \to  p^0 (x) \quad \hbox{ strongly in } \, L^2(\O) \label{1.71} \\
\Phi_j^\ep  \to  \Phi_j^0(x) \quad \hbox{ weakly in } \, H^1 (\Omega)
\mbox{ and strongly in }  L^2 (\Omega)
\label{convPhi}\\
\nabla \Phi_j^\ep  \to \nabla_x \Phi_j^0(x) + \nabla_y \Phi_j^1(x,y) \quad
\hbox{ in the two-scale sense }
\label{convgradPhi}
\end{gather*}
where $(\mathbf{u}^0 , p^0) \in L^2 (\O ; H^1_{per}(Y)^d) \times L^2_0 (\O)$
and $\{ \Phi_j^0 , \Phi_j^1 \}_{j=1, \dots , N} \in \left(H^1 (\O)\times L^2(\O; H^1_{per}(Y))\right)^N$ is the unique solution of the two-scale homogenized problem
\begin{gather}
- \Delta_y \mathbf{u}^0(x,y) + \nabla_y p^1(x,y) = - \nabla_x p^0(x) - \mathbf{f}^*(x) \notag \\
+ \sum_{j=1}^N z_j n_j^{0}(y) \left(\nabla_x \Phi_j^0(x) + \nabla_y \Phi_j^1(x,y) + \mathbf{E}^*(x)\right)
\ \mbox{ in } \, \O\times Y_F , \label{Stokes1} \\
\div_y \mathbf{u}^0 (x,y) =0 \ \mbox{ in } \, \O \times Y_F , \quad\quad \mathbf{u}^0 (x,y) =0 \ \mbox{on } \, \O \times S , \label{Stokes2}\end{gather} \begin{gather}
\div_x \left(\int_{Y_F}\mathbf{u}^0(x,y)  \, dy \right) =0 \, \mbox{ in } \O , \label{VAREP51}\\
- \div_y n^0_i(y) \bigg( \sum_{j=1}^N  K_{ij}(y) z_j \big( \nabla_y \Phi^1_j(x,y)  + \nabla_x \Phi_j ^0 (x) +  \mathbf{E}^*(x)\big)  + \pei \mathbf{u}^0(x,y) \bigg) =0  \ \mbox{ in } \O\times Y_F ,\label{Diff1}\\
\sum_{j=1}^N  K_{ij}(y) z_j \big( \nabla_y \Phi^1_j  + \nabla_x \Phi_j ^0  + \mathbf{E}^* \big) \cdot \nu(y) =0 \ \mbox{ on } \O\times S, \label{Diff2} \\
-\div_x  \int_{Y_F} n^0_i(y) \bigg( \sum_{j=1}^N  K_{ij}(y) z_j \big( \nabla_y \Phi^1_j(x,y)  + \nabla_x \Phi_j ^0 (x) +  \mathbf{E}^*(x)\big)  + \pei \mathbf{u}^0(x,y)  \bigg) \, dy  =0  \; \mbox{in } \, \O  ,\label{Diffg1}\\
\Phi_i^0 \ , \; \int_{Y_F}\mathbf{u}^0 \, dy \ \mbox{ and } \ p^0 \mbox{ being $\Omega$-periodic in } x , \label{Diffg2}
\end{gather}
with periodic boundary conditions on the unit cell $Y_F$ for all functions depending on $y$
and $S=\partial Y_S \setminus \p Y$.
\end{theorem}

The limit problem introduced in Theorem \ref{1.15} is called the two-scale
and two-pressure homogenized problem, following the terminology of \cite{hornung},
\cite{LIO:81}. It features two incompressibility constraints
(\ref{Stokes2}) and (\ref{VAREP51}) which are exactly dual to the two pressures $p^0(x)$
and $p^1(x,y)$ which are their corresponding Lagrange multipliers.
Remark that equations (\ref{Stokes1}), (\ref{Stokes2}) and (\ref{Diff1})
are just the leading order terms in the ansatz of the original equations.
On the other hand, equations (\ref{VAREP51}) and (\ref{Diffg1}) are averages
on the unit cell $Y_F$ of the next order terms in the ansatz.
For example, (\ref{VAREP51}) is deduced from
$$
\div_y \mathbf{u}^1 (x,y) + \div_x \mathbf{u}^0 (x,y)=0 \ \mbox{ in } \, \O \times Y_F
$$
by averaging on $Y_F$, recalling that $\mathbf{u}^1(x,y) =0$ on $\O\times S$.

\begin{proof}
The proof of convergence and the derivation of the homogenized system
is completely similar to the proof of Theorem 1 in \cite{AMP} which
holds in the ideal case. The only point which deserves to be made
precise here is the well-posedness of the two-scale homogenized problem.

Following section 3.1.2 in \cite{All97}, we introduce the functional space for the velocities
$$
V= \{ \mathbf{u}^0(x,y) \in L^2_{per}\left(\O ; H^1_{per}(Y_F)^d\right) \; \mbox{ satisfying }
(\ref{Stokes2})-(\ref{VAREP51}) \} ,
$$
which is known to be orthogonal in $L^2_{per}\left(\O ; H^1_{per}(Y_F)^d\right)$ to
the space of gradients of the form $\nabla_x q(x) + \nabla_y q_1(x,y)$ with $q(x)\in
H^1_{per}(\O)/\RR$ and $q_1(x,y)\in L^2_{per}\left(\O ; L^2_{per}(Y_F)/\RR\right)$.
We define the functional space
$X=V \times H^1_{per} (\O)/ \RR \times L^2_{per}(\O; H^1_{per}(Y_F)^d /\RR)$
and the variational formulation of (\ref{Stokes1})-(\ref{Diffg2}) is to find
$( \mathbf{u}^0 , \{ \Phi_j^0 , \Phi_j^1 \} ) \in X$ such that,
for any test functions $(\mathbf{v} , \{ \phi^0_j , \phi^1_j \} ) \in X$,
\begin{equation}
\label{VAR0}
a\left( ( \mathbf{u}^0 , \{ \Phi_j^0 , \Phi_j^1 \} ) , ( \mathbf{v}, \{ \phi^0_j , \phi^1_j \} ) \right) =
\langle \mathcal{L}, ( \mathbf{v}, \{ \phi^0_j , \phi^1_j \} ) \rangle ,
\end{equation}
where the bilinear form $a$ and the linear form $\mathcal{L}$ are defined by
\begin{gather}
a\left( ( \mathbf{u}^0 , \{ \Phi_j^0 , \Phi_j^1 \} ) , ( \mathbf{v}, \{ \phi^0_j , \phi^1_j \} ) \right) := \int_{\O} \int_{Y_F} \nabla_y \mathbf{u}^0 : \nabla \mathbf{v} \, dx \, dy \notag \\
+ \sum_{i,j=1}^N \frac{z_iz_j}{\pei} \int_{\O} \int_{Y_F} n_i^{0} K_{ij} (\nabla_x \Phi_j^0
+ \nabla_y \Phi_j^1 ) \cdot (\nabla_x \phi_j^0 + \nabla_y \phi_j^1 ) \, dx \, dy
\label{VAR1} \\
 + \sum_{j=1}^N z_j \int_{\O} \int_{Y_F} n_j^{0} \big( \mathbf{u}^0
\cdot (\nabla_x \phi_j^0 + \nabla_y \phi_j^1 ) -
\mathbf{v} \cdot (\nabla_x \Phi_j^0 + \nabla_y \Phi_j^1 ) \big) \ dx  \, dy \notag
\end{gather}
and
\begin{gather}
< \mathcal{L}, ( \mathbf{v}, \{ \phi_j \} ) > :=
\sum_{j=1}^N z_j \int_{\O} \int_{Y_F} n_j^{0} \mathbf{E}^* \cdot \mathbf{v} \, dx  \, dy
- \int_{\O} \int_{Y_F} \mathbf{f}^*\cdot\mathbf{v} \, dx \, dy \notag \\
- \sum_{i,j=1}^N \frac{z_iz_j}{\pei} \int_{\O} \int_{Y_F} n_i^{0} K_{ij} \mathbf{E}^*
\cdot (\nabla_x \phi_j^0 + \nabla_y \phi_j^1 ) \, dx \, dy  ,\notag
\end{gather}
We apply the Lax-Milgram lemma to prove the existence and uniqueness of
the solution in $X$ of (\ref{VAR0}).
The only point which requires to be checked is the coercivity of the bilinear form.
We take $\mathbf{v} =\mathbf{u}^0$, $\phi_j^0 = \Phi_j^0$ and $\phi_j^1 = \Phi_j^1$ as the test functions in (\ref{VAR0}).

We define a local diffusion tensor
\begin{equation}
\label{local.tensor}
\tilde K(y) = \left( \frac{z_iz_j}{\pei} K_{ij}(y)n^0_i(y) \right)_{1\leq i,j\leq N}
= \left(  \frac{k_B T}{u_c L} z_iz_j L_{ij}(y) \right)_{1\leq i,j\leq N} ,
\end{equation}
which is symmetric since $(L_{ij})$ is symmetric too. As already remarked in the
proof of Lemma \ref{lem.existvf}, $\tilde K$ is uniformly coercive for small enough
$n_c>0$ and $\xi_c>0$.
Therefore, the second integral on the right hand side of (\ref{VAR1}) is positive.
The third integral, being skew-symmetric, vanishes, which proves the coercivity of $a$.
\end{proof}

Of course, one should extract from (\ref{Stokes1})-(\ref{Diffg2}) the macroscopic
homogenized problem, which requires to separate the fast and slow scale. In the ideal
case, Looker and Carnie in \cite{LC:06} proposed a first approach which was further
improved in \cite{AMP} and  \cite{ABDMP}.

The main idea is to recognize in the two-scale homogenized problem
(\ref{Stokes1})-(\ref{Diffg2}) that there are two different macroscopic
fluxes, namely $(\nabla_x p^0 (x)+\mathbf{f}^*(x))$ and
$\{\nabla_x \Phi_j^0(x) + \mathbf{E}^*(x)\}_{1\leq j\leq N}$.
Therefore we introduce two family of cell problems, indexed by $k\in\{1,...,d\}$
for each component of these fluxes. We denote by $\{\mathbf{e}^k\}_{1\leq k\leq d}$
the canonical basis of $\RR^d$.

The first cell problem, corresponding to the macroscopic pressure gradient, is
\begin{gather}
- \Delta_y \mathbf{v}^{0,k}(y) + \nabla_y \pi^{0,k}(y) = \mathbf{e}^k +
\sum_{j=1}^N z_j n_j^{0}(y) \nabla_y \theta^{0,k}_j(y) \; \mbox{ in }  Y_F \label{StokesAux0}\\
\div_y \mathbf{v}^{0,k}(y) =0 \quad \mbox{ in }  Y_F , \quad  \mathbf{v}^{0,k}(y) =0 \, \mbox{ on } S , \label{divAux0}\\
-\div_y n^0_i(y) \bigg( \sum_{j=1}^N  K_{ij}(y) z_j \nabla_y \theta^{0,k}_j(y) +
\pei \mathbf{v}^{0,k}(y) \bigg) = 0  \; \mbox{ in }  Y_F \label{DiffAux0}\\
\sum_{j=1}^N  K_{ij}(y) z_j \nabla_y \theta^{0,k}_j(y) \cdot \nu =0  \; \mbox{ on } S. \label{bcAux0}
\end{gather}
The second cell problem, corresponding to the macroscopic diffusive flux, is
for each species $l\in\{1,...,N\}$
\begin{gather}
- \Delta_y \mathbf{v}^{l,k}(y) + \nabla_y \pi^{l,k}(y) = \sum_{j=1}^N z_j n_j^{0}(y) (
\delta_{lj} \mathbf{e}^k  +  \nabla_y \theta^{l,k}_j(y) ) \; \mbox{ in }  Y_F \label{StokesAuxi}\\
\div_y \mathbf{v}^{l,k}(y) =0 \quad \mbox{ in }  Y_F , \quad \mathbf{v}^{l,k}(y) =0 \quad \mbox{ on } S , \label{divAuxi} \\
-\div_y n^0_i(y) \bigg( \sum_{j=1}^N  K_{ij}(y) z_j \big( \delta_{lj} \mathbf{e}^k +\nabla_y \theta^{l,k}_j(y) \big) +
\pei \mathbf{v}^{l,k}(y) \bigg) =0  \; \mbox{ in }  Y_F \label{DiffAuxi}\\
\sum_{j=1}^N  K_{ij}(y) z_j \big( \delta_{lj} \mathbf{e}^k +\nabla_y \theta^{l,k}_j(y) \big) \cdot \nu =0  \; \mbox{ on } S, \label{bcAuxi}
\end{gather}
where $\delta_{ij}$ is the Kronecker symbol. As usual the cell problems
are complemented with periodic boundary conditions.

Then, we can decompose the solution of (\ref{Stokes1})-(\ref{Diffg2}) as
\begin{gather}
\mathbf{u}^0 (x,y) = \sum_{k=1}^d \left( - \mathbf{v}^{0,k}(y)
\left( \frac{\p p^0}{\p x_k} +
f^*_k \right)(x) +
\sum_{i=1}^N \mathbf{v}^{i,k}(y) \left( E^*_k + \frac{\p \Phi^0_i}{\p x_k} \right)(x) \right) \label{micro1}\\
 p^1 (x,y) = \sum_{k=1}^d  \left( - \pi^{0,k}(y)
\left( \frac{\p p^0}{\p x_k} +
f^*_k \right)(x) +
\sum_{i=1}^N \pi^{i,k}(y) \left( E^*_k + \frac{\p \Phi^0_i}{\p x_k} \right)(x) \right) \label{micro2}\\
 \Phi_j^1 (x,y) = \sum_{k=1}^d  \left( - \theta^{0,k}_j(y)
\left( \frac{\p p^0}{\p x_k} +
f^*_k \right)(x) +
\sum_{i=1}^N \theta^{i,k}_j(y) \left( E^*_k + \frac{\p \Phi^0_i}{\p x_k} \right)(x) \right) .\label{micro3}
\end{gather}
We average (\ref{micro1})-(\ref{micro3}) in order to get a purely
macroscopic homogenized problem.
We define the homogenized quantities: first, the electrochemical potential
\begin{equation}
\label{qhom1}
\mu_j(x) = - z_j (\Phi_j^0 (x) + \Psi^{ext,*}(x) ) ,
\end{equation}
then, the ionic flux of the $j$th species
\begin{equation}
\label{qhom2}
\mathbf{j}_j(x) = \frac{1}{|Y_F |} \int_{Y_F}
n^0_j(y) \bigg( \sum_{l=1}^N  K_{jl}(y) \frac{z_l}{\pej} \big( \nabla_y \Phi^1_l(x,y)  + \nabla_x \Phi_l ^0 (x) +  \mathbf{E}^*(x)\big)  + \mathbf{u}^0 \bigg) dy ,
\end{equation}
and finally the filtration velocity
\begin{equation}
\label{qhom3}
\mathbf{u}(x) =  \frac{1}{|Y_F|} \int_{Y_F} \mathbf{u}^{0} (x,y) \, dy .
\end{equation}
From (\ref{micro1})-(\ref{micro3}) we deduce the homogenized or upscaled equations
for the above effective fields.

\begin{proposition}
\label{prop.eff}
Introducing the flux
$\mathcal{J}(x) = (\mathbf{u},\{\mathbf{j}_j\}_{1\leq j\leq N})$
and the gradient
$\mathcal{F}(x) =(\nabla_x p^0, \{\nabla_x \mu_j\}_{1\leq j\leq N})$,
the macroscopic equations are
\begin{gather}
\div_x \mathcal{J} =0 \quad \mbox{in } \; \O , \label{hom1}\\
\mathcal{J} = - \mathcal{M} \mathcal{F} - \mathcal{M} (\mathbf{f}^* , \{ 0\}) , \label{hom2}
\end{gather}
with a homogenized tensor $\mathcal{M}$ defined by
\begin{equation}
\label{Onsager}
   \mathcal{M}  = \left(
                    \begin{array}{cccc}
\dsp  \mathbb{K}    & \dsp \frac{\mathbb{J}_1}{z_1}  & \dots  &  \dsp  \frac{\mathbb{J}_N}{z_N}  \\
\dsp  \mathbb{L}_1 & \dsp \frac{\mathbb{D}_{11}}{z_1} & \cdots &  \dsp       \frac{\mathbb{D}_{1N}}{z_N}\\
\dsp  \vdots  & \vdots & \ddots & \vdots \\
\dsp  \mathbb{L}_N & \dsp  \frac{\mathbb{D}_{N1}}{z_1} & \cdots & \dsp \frac{\mathbb{D}_{NN}}{z_N} \\
                    \end{array}
                  \right) ,
\end{equation}
and complemented with periodic boundary conditions for $p^0$
and $\{\Phi_j^0\}_{1\leq j\leq N}$.
The matrices $\mathbb{J}_i$, $\mathbb{K}$, $\mathbb{D}_{ji}$ and $\mathbb{L}_j$
are defined by their entries
\begin{gather}
\{ \mathbb{J}_i \}_{lk} = \frac{1}{|Y_F |}\int_{Y_F } \mathbf{v}^{i,k}(y) \cdot \mathbf{e}^l \, dy ,\notag \\
\{ \mathbb{K} \}_{lk} = \frac{1}{|Y_F |} \int_{Y_F } \mathbf{v}^{0,k}(y) \cdot \mathbf{e}^l \, dy ,\notag\\
\{ \mathbb{D}_{ji} \}_{lk} = \frac{1}{|Y_F |} \int_{Y_F } n_j^0 (y) \Big( \mathbf{v}^{i,k}(y)
+ \sum_{m=1}^N  K_{jm}(y) \frac{z_m}{\pej}
\left( \delta_{im} \mathbf{e}^k + \nabla_y \theta^{i,k}_m (y) \right) \Big) \cdot \mathbf{e}^l \ dy ,\notag \\
\{ \mathbb{L}_j \}_{lk} = \frac{1}{|Y_F |} \int_{Y_F } n_j^0 (y) \Big( \mathbf{v}^{0,k}(y)
+ \sum_{m=1}^N  K_{jm}(y) \frac{z_m}{\pej}
\nabla_y \theta^{0,k}_m (y) \Big) \cdot \mathbf{e}^l \ dy .\notag
\end{gather}
Furthermore, $\mathcal{M}$ is symmetric positive definite,
which implies that the homogenized equations (\ref{hom1})-(\ref{hom2}) have a
unique solution.
\end{proposition}

\begin{remark}
The symmetry of $\mathcal{M}$ is equivalent to the famous Onsager's reciprocal relations.
In the ideal case, the symmetry of the tensor $\mathcal{M}$ was proved in \cite{LC:06}, \cite{AMP}.
\end{remark}

\begin{proof}
The conservation law (\ref{hom1}) is just a rewriting of (\ref{VAREP51})
and (\ref{Diffg1}). The constitutive equation (\ref{hom2}) is an immediate
consequence of the definitions (\ref{qhom2}) and (\ref{qhom3}) of the
homogenized fluxes, taking into account the decomposition (\ref{micro1})-(\ref{micro3}).

We now prove that $\mathcal{M}$ is positive definite.
For any collection of vectors $\lambda^0,\{\lambda^i\}_{1\leq i\leq N}\in\RR^d$
let us introduce the following linear combinations of the cell solutions
\begin{gather}
\mathbf{v}^{\lambda} = \sum_{k=1}^d \left( \lambda^0_k \mathbf{v}^{0,k} +
\sum_{i=1}^N \lambda^i_k \mathbf{v}^{i,k} \right) , \quad
\theta^{\lambda}_j = \sum_{k=1}^d \left(\lambda^0_k \theta^{0,k}_j +
\sum_{i=1}^N \lambda^i_k \theta_j^{i,k} \right) , \label{vlambda}
\end{gather}
which satisfy
\begin{gather}
- \Delta_y \mathbf{v}^{\lambda}(y) + \nabla_y \pi^{\lambda}(y) = \lambda^0 +
\sum_{j=1}^N z_j n_j^0(y) \left(
\lambda^j  + \nabla_y \theta^{\lambda}_j(y) \right) \; \mbox{ in }  Y_F \label{StokesAuxl}\\
\div_y \mathbf{v}^{\lambda}(y) =0 \quad \mbox{ in }  Y_F , \quad \mathbf{v}^{\lambda}(y) =0 \quad \mbox{ on } S , \label{divAuxl}\\
-\div_y \left( n^0_i (y) \left( \sum_{j=1}^N z_j K_{ij}
(\lambda^j +\nabla_y \theta^{\lambda}_j(y)) + \pei \mathbf{v}^{\lambda}(y) \right) \right) =0  \; \mbox{ in }  Y_F \label{DiffAuxl}\\
\sum_{j=1}^N z_j K_{ij}
(\lambda^j +\nabla_y \theta^{\lambda}_j(y)) \cdot \nu =0  \; \mbox{ on } S, \label{bcAuxl}
\end{gather}
Multiplying the Stokes equation (\ref{StokesAuxl}) by $\mathbf{v}^{\lambda}$,
the convection-diffusion equation (\ref{DiffAuxl}) by $\theta_j^{\lambda}$ and
summing up, we obtain
\begin{gather}
\int_{Y_F} \left( | \nabla_y \mathbf{v}^{\lambda} (y) |^2  +
\sum_{i,j=1}^N \frac{z_iz_j}{\pei} n_i^0(y) K_{ij}(y) (\nabla_y \theta^{\lambda}_j(y)
+ \lambda^j ) \cdot (\nabla_y \theta^{\lambda}_i(y) + \lambda^i ) \right) dy \notag\\
= \int_{Y_F} \lambda^0\cdot \mathbf{v}^{\lambda} \, dy + \sum_{i=1}^N
\int_{Y_F} z_i n_i^0 \lambda^i\cdot \mathbf{v}^{\lambda} \, dy
+ \sum_{i,j=1}^N \int_{Y_F} \frac{z_iz_j}{\pei} n_i^0 K_{ij}
(\nabla_y\theta^{\lambda}_j + \lambda^j) \cdot \lambda^i \, dy \notag\\
= \mathbb{K} \lambda^0\cdot\lambda^0 + \sum_{i=1}^N \mathbb{J}_i \lambda^i\cdot\lambda^0
+ \sum_{i,j=1}^N z_i \lambda^i\cdot \mathbb{D}_{ij} \lambda^j
+ \sum_{i=1}^N z_i \lambda^i\cdot \mathbb{L}_{i} \lambda^0
= \mathcal{M} (\lambda^0,\{z_i\lambda^i\})^T \cdot (\lambda^0,\{z_i\lambda^i\})^T .\notag
\end{gather}
The left hand side of the above equality is positive.
This proves the positive definite character of $\mathcal{M}$.

Following a computation of \cite{AMP} in the ideal case, we prove the
symmetry of $\mathcal{M}$. For another set of vectors
$\tlambda^0,\{\tlambda^i\}_{1\leq i\leq N}\in\RR^d$, we define
$\mathbf{v}^{\tlambda}$ and $\theta^{\tlambda}_j$ by (\ref{vlambda}).
Multiplying the Stokes equation for $\mathbf{v}^{\lambda}$ by $\mathbf{v}^{\tlambda}$
and the convection-diffusion equation for $\theta_j^{\tlambda}$ by $\theta_j^{\lambda}$
(note the skew-symmetry of this computation), then adding the two variational
formulations yields
\begin{gather}
\int_{Y_F} \nabla_y \mathbf{v}^{\lambda} \cdot \nabla_y \mathbf{v}^{\tlambda} \, dy +
\sum_{i,j=1}^N \int_{Y_F} \frac{z_iz_j}{\pei} n_i^0 K_{ij} \nabla_y \theta^{\tlambda}_j
\cdot \nabla_y \theta^{\lambda}_j \, dy = \notag \\
\int_{Y_F} \lambda^0\cdot \mathbf{v}^{\tlambda} \, dy +
\sum_{j=1}^N \int_{Y_F} z_j n_j^0 \lambda^j\cdot \mathbf{v}^{\tlambda} \, dy -
\sum_{i,j=1}^N \int_{Y_F} \frac{z_iz_j}{\pei} n_i^0 K_{ij} \tlambda^j\cdot
\nabla_y\theta^{\lambda}_i \, dy .\label{vfsym}
\end{gather}
The diffusion tensor appearing in the left hand side of (\ref{vfsym})
is precisely equal to $\tilde K$, defined by (\ref{local.tensor}),
which is symmetric.
Therefore, the left hand side of (\ref{vfsym}) is symmetric in $\lambda,\tlambda$.
Exchanging the last term in (\ref{vfsym}), we deduce by symmetry
\begin{gather}
\int_{Y_F} \lambda^0\cdot \mathbf{v}^{\tlambda} \, dy +
\sum_{j=1}^N \int_{Y_F} z_j n_j^0 \lambda^j\cdot \mathbf{v}^{\tlambda} \, dy +
\sum_{i,j=1}^N \int_{Y_F} \frac{z_iz_j}{\pei} n_i^0 K_{ij} \lambda^j\cdot
\nabla_y\theta^{\tlambda}_i \, dy \notag \\
= \int_{Y_F} \tlambda^0\cdot \mathbf{v}^{\lambda} \, dy +
\sum_{j=1}^N \int_{Y_F} z_j n_j^0 \tlambda^j\cdot \mathbf{v}^{\lambda} \, dy +
\sum_{i,j=1}^N \int_{Y_F} \frac{z_iz_j}{\pei} n_i^0 K_{ij} \tlambda^j\cdot
\nabla_y\theta^{\lambda}_i \, dy ,\notag
\end{gather}
which is equivalent to the desired symmetry
$$
\mathcal{M} (\tlambda^0,\{z_i\tlambda^i\})^T \cdot (\lambda^0,\{z_i\lambda^i\})^T= \mathcal{M} (\lambda^0,\{z_i\lambda^i\})^T \cdot (\tlambda^0,\{z_i\tlambda^i\})^T .
$$
\end{proof}


\section{Existence of solutions to the MSA variant of Poisson-Boltzmann equation}
\label{secexist}

{  The goal of this section is to prove Theorem \ref{thm.exist3},
i.e., the existence of solutions to system (\ref{BP0}), the MSA variant of
Poisson-Boltzmann equation.}
These solutions are the so-called equilibrium solutions computed in
Section \ref{sec3}. In a slightly different setting (two species
and a linear approximation of $p(\xi)$) and with a different
method (based on a saddle point approach in the two variables,
potential and concentrations), a previous existence result was
obtained in \cite{EJL}.

To simplify the notations we shall drop all $\ep$- or $0$-indices.
{  In the same spirit, the pore domain is denoted $\Omega_p$, a subset
of the full domain $\Omega$. To simplify we
denote by $\p\Omega_p$ the solid boundary of $\Omega_p$, which
should rather be $\partial\Omega_p \setminus \partial\Omega$ since
we impose periodic boundary conditions on $\partial\Omega$.
With our simplified notations, Theorem \ref{thm.exist3} is restated
below as Theorem \ref{thm.main} and the Poisson-Boltzmann equation reads}
\begin{equation}
\label{BP000}
\left\{ \begin{array}{ll}
\dsp - \Delta \Psi =\beta  \sum_{j=1}^N z_j n_j(x) \; \mbox{ in } \ \Omega_p , &  \\
\dsp  \nabla \Psi \cdot \nu = -N_\sigma \Sigma^* \ \mbox{ on } \, \partial\Omega_p ,\
 \Psi \; \mbox{ is } \Omega-\mbox{periodic},&
\end{array} \right.
\end{equation}
where, in view of (\ref{NJ00}), the equilibrium concentrations are
\begin{equation}
\label{NJ0HS}
n_j = \frac{n_j^0(\infty) \gamma_j^0(\infty)}{\gamma^{HS}} \exp \{ - z_j \Psi + \frac{L_B \Gamma \Gamma_c z_j^2}{1+ \Gamma \Gamma_c \sigma_j}  \} .
\end{equation}
We recall that the MSA screening parameter $\Gamma$ is defined by
\begin{equation}
\label{Gammaeps}
(\Gamma)^2 = \sum_{j=1}^N \frac{n_j z_j^2}{(1+ \Gamma_c \Gamma \sigma_j)^2}.
\end{equation}
and the hard sphere part of the activation
coefficient is given by
\begin{equation}
\label{gammaeps}
\gamma^{HS} = \exp \{ p(\xi) \} \; \mbox{ with } \quad
p(\xi) = \xi \frac{8-9\xi + 3\xi^2}{(1-\xi)^3}  \; \mbox{ and } \quad
\xi =  \xi_c \sum_{j=1}^N n_j (\frac{\sigma_j}{\sigma_c})^3 ,
\end{equation}
where $\xi \in [0,1)$ is the solute packing fraction and $\xi_c$ its
characteristic value defined by (\ref{xic}).

{
Let us now explain our strategy to solve the boundary value problem
(\ref{BP000}) coupled with the algebraic equations (\ref{NJ0HS}),
(\ref{Gammaeps}) and (\ref{gammaeps}). In a first step (Lemmas
\ref{lem.xi} and \ref{lem.gamma2}) we eliminate the algebraic
equations and write a nonlinear boundary value problem (\ref{BP12})
for the single unknown $\Psi$. In a second step
we introduce a truncated or "cut-off" problem (\ref{BP12cutoff})
which is easily solved by a standard energy minimization since
the nonlinearity has been truncated.
The third and most delicate step is to prove a maximum principle
for these truncated solutions (Proposition \ref{Linft}) which,
in turns, imply our desired existence result.}

{   In a first step we eliminate $\xi$ as a function
of $(\Psi,\Gamma)$ and then $\Gamma$ as a function of $\Psi$.}
From (\ref{sol.xi0}), for given potential $\Psi$ and screening parameter $\Gamma$,
the solute packing fraction $\xi$ is a solution of the algebraic equation
\begin{equation}
\label{sol.xi}
\xi =  \exp \{ -p(\xi) \} \xi_c \sum_{j=1}^N (\frac{\sigma_j}{\sigma_c})^3
n_j^0(\infty) \gamma_j^0(\infty) \exp \left\{ - z_j \Psi +
\frac{L_B \Gamma \Gamma_c z_j^2}{1+ \Gamma \Gamma_c \sigma_j} \right\} .
\end{equation}

\begin{lemma}
\label{lem.xi}
For given values of $\Psi$ and $\Gamma$, there exists
a unique solution $\xi\equiv\xi(\Psi,\Gamma)\in[0,1)$ of (\ref{sol.xi}).
Furthermore, this solution depends smoothly on $\Psi,\Gamma$ and
is increasing with $\Gamma$.
\end{lemma}

\begin{proof}
{  One can check that $p(\xi)$ is an increasing function of $\xi$ on
$[0,1)$ with range $\RR^+$ since}
$$
p^\prime(\xi) = \frac{8-2\xi}{(1-\xi)^4} .
$$
The existence and uniqueness follows from the strict decrease of
the function $\exp \{ -p(\xi) \}$ from 1 to 0, {  while the left hand ‪side
$\xi$ of (\ref{sol.xi})} increases from 0 to 1. Since the function
$$
\Gamma \to
\frac{L_B \Gamma \Gamma_c z_j^2}{1+ \Gamma \Gamma_c \sigma_j}
$$
is increasing, so is the solution $\xi$ of (\ref{sol.xi}) as a function of $\Gamma$.
\end{proof}

Once we know $\xi\equiv\xi(\Psi,\Gamma)$, the MSA screening parameter
$\Gamma$ satisfies the following algebraic equation (see (\ref{GammaL0}))
\begin{gather}
(\Gamma)^2 =  \sum^N_{j=1} n_j^0(\infty) \gamma_j^0(\infty)
\frac{z_j^2}{(1+\Gamma \Gamma_c \sigma_j)^2}
\exp \left\{ - z_j \Psi +
\frac{L_B \Gamma \Gamma_c z_j^2}{1+ \Gamma \Gamma_c \sigma_j}
- p\left(\xi(\Psi,\Gamma)\right) \right\} .
\label{GammaL2}
\end{gather}
We now prove that the algebraic equation (\ref{GammaL2}) admits
a unique solution $\Gamma(\Psi)$ under a mild assumption.

\begin{lemma}
\label{lem.gamma2}
For any value of $\Psi$, there always exists at least one solution
$\Gamma\equiv\Gamma(\Psi)$ of the algebraic equation (\ref{GammaL2}).
Furthermore, under the following assumption on the physical parameters
\begin{equation}
\label{Bound1}
L_B < (6+4\sqrt{2})  \min_{1\leq j\leq N} \frac{\sigma_j}{z_j^2}
\quad \mbox{ with } \quad 6+4\sqrt{2} \approx 11.656854 ,
\end{equation}
the solution $\Gamma(\Psi)$ is unique and is a differentiable function of $\Psi$.
\end{lemma}

\begin{proof}
Existence of a solution is a consequence of the fact that, as functions
of $\Gamma$, the left hand side of (\ref{GammaL2}) spans $\RR^+$ while
the right hand side remains positive and bounded on $\RR^+$.

Denote by $F(\Gamma)$ the difference between the left and
the right hand sides of (\ref{GammaL2}).
Let us show that (\ref{Bound1}) implies that $F$ is an increasing function on $\RR^+$, and, moreover, $F'(\Gamma)>0$.
To this end we
use the trick $2\Gamma=2(\Gamma)^2/\Gamma$ and compute the derivative
\begin{gather}
F^\prime(\Gamma) = \sum^N_{j=1} n_j^0(\infty) \gamma_j^0(\infty)
\frac{z_j^2 \exp \big\{ - z_j \Psi + \frac{L_B \Gamma \Gamma_c z_j^2}{1+ \Gamma \Gamma_c \sigma_j}
- p(\xi) \big\}}{(1+\Gamma \Gamma_c \sigma_j)^2}
\left( \frac{2}{\Gamma} - \Gamma_c \frac{L_B z_j^2 - 2 \sigma_j (1+ \Gamma \Gamma_c \sigma_j )}{(1+ \Gamma \Gamma_c \sigma_j )^2} \right) \notag \\
+ \frac{\partial\xi}{\partial\Gamma} \, p^\prime(\xi) \sum^N_{j=1} n_j^0(\infty) \gamma_j^0(\infty)
\frac{z_j^2}{(1+\Gamma \Gamma_c \sigma_j)^2}
\exp \{ - z_j \Psi +
\frac{L_B \Gamma \Gamma_c z_j^2}{1+ \Gamma \Gamma_c \sigma_j}  - p(\xi)\} .
\label{eq.derivF}
\end{gather}
Lemma \ref{lem.xi} shows that $\partial\xi / \partial\Gamma>0$,
so the second line of (\ref{eq.derivF}) is positive.
Introducing $x=\Gamma \Gamma_c \sigma_j$, the sign of each term in the sum
of the first line of (\ref{eq.derivF})
is exactly that of the polynomial $P(x)= 4x^2 + (6-L_B z_j^2/ \sigma_j ) x +2$.
A simple computation shows that $P(x)$ has no positive roots (and thus is positive
for $x\geq0$) if and only if (\ref{Bound1}) holds true.

Since, $F(0)<0$ and $\lim_{+\infty} F(\Gamma)=+\infty$, the inequality $F'(\Gamma)>0$ yields the
existence and uniqueness of the root $\Gamma$ such that $F(\Gamma)=0$.
Then, a standard application of the implicit function theorem leads to the
differentiable character of $\Gamma(\Psi)$.

\end{proof}

\begin{remark}
\label{rem.multi}
The bound (\ref{Bound1}) is a sufficient, but not a necessary, condition for
uniqueness of the root $\Gamma(\Psi)$, solution of (\ref{GammaL2}). There
are other criteria (not discussed here) which ensure the uniqueness of $\Gamma(\Psi)$.
However there are cases when multiple solutions do exist: it is interpreted
as a phase transition phenomenon and it was studied, e.g., in \cite{joubaud}.
\end{remark}

In view of Lemma \ref{lem.gamma2} the solute packing fraction is now a
nonlinear function of the potential $\Psi$ that we denote by
$$
\tilde\xi(\Psi) \equiv \xi\Big(\Psi,\Gamma(\Psi)\Big) .
$$
{   As a result of our first step, the electrostatic equation
(\ref{BP000}) reduces to the following Poisson-Boltzmann equation which
is a nonlinear partial differential equation for the sole unknown $\Psi$}
\begin{equation}
\label{BP12}
\left\{ \begin{array}{ll}
\dsp - \Delta \Psi =\beta  \sum_{j=1}^N z_j n_j^0(\infty) \gamma_j^0(\infty) \exp \left\{ - z_j \Psi + \frac{L_B \Gamma(\Psi) \Gamma_c z_j^2}{1+ \Gamma(\Psi) \Gamma_c \sigma_j} -p(\tilde\xi(\Psi) ) \right\} \; \mbox{ in } \ \Omega_p , &  \\
\dsp \nabla \Psi \cdot \nu = -N_\sigma \Sigma^* \ \mbox{ on } \, \partial\Omega_p , \quad
 \Psi \; \mbox{ is } \Omega-\mbox{periodic}.&
\end{array} \right.
\end{equation}
{  Recall that $N_\sigma>0$ is a parameter and that $\Sigma^*(x)$
is assumed to be a $\Omega$-periodic function in $L^\infty(\p \Omega_p)$.
Our goal is to prove existence of at least one solution to problem (\ref{BP12}).
The main difficulty is the non-linearity of the right hand side which is
growing exponentially fast at infinity.}
Recalling definition (\ref{E_j}) of $E_j(\Psi)$, the right hand side of (\ref{BP12})
is the nonlinear function $\Phi$ defined by its derivative
\begin{equation}
\label{pres}
\Phi^\prime (\Psi) = \beta\sum^N_{j=1} E_j^\prime(\Psi) .
\end{equation}
In the ideal case, Remark \ref{rem.E_j} tells us that
$E_j(\Psi)=n_j^0(\infty) \exp \{ - z_j \Psi \}$.
We are thus lead to introduce
\begin{equation}
\label{idealg}
g(\psi) = \sum^N_{j=1} n_j^0(\infty)\gamma_j^0(\infty) \exp \left\{ - z_j \psi \right\} ,
\quad \psi\in \mathbb{R} ,
\end{equation}
which is a strictly convex function.
{  In the ideal case, we have $\Phi(\Psi) = \beta g (\Psi)$ and
the existence and uniqueness of a solution of (\ref{BP12}) is more or less
standard thanks to a monotonicity argument (see \cite{L:06}, \cite{PoissBoltz:12}).}
For the MSA model our strategy of proof is different since $\Phi$ is not
anymore convex. We rely on a truncation argument, $L^\infty$-bounds and
still some monotonicity properties of part of $\Phi^\prime$. Our proof
requires a smallness condition on the characteristic value $\xi_c$.

{   The second step of our proof introduces
a truncation operator at the level $M>0$ defined, for any function $\varphi$, by}
$$
T_M(\varphi) = \left\{ \begin{array}{ll}
- \frac{M}{z_N} & \mbox{ if } \varphi <- \frac{M}{z_N} , \\
\varphi & \mbox{ if } - \frac{M}{z_N} \leq \varphi \leq \frac{M}{|z_1|} , \\
\frac{M}{|z_1|} & \mbox{ if } \varphi > \frac{M}{|z_1|} .
\end{array} \right.
$$
Note that this truncation is not symmetric since the growth condition at
$\pm\infty$ of $\Phi$ and $g$ are not symmetric too.
We define a "cut-off" function $\Phi_M$ by its derivative
\begin{equation}
\label{cutoof}
\Phi_M^\prime (\Psi) = \Phi^\prime \circ T_M (\Psi) ,
\end{equation}
and solve the associated "cut-off" problem
\begin{equation}
\label{BP12cutoff}
\left\{ \begin{array}{ll}
\dsp - \Delta \Psi_M =- \Phi_M^\prime (\Psi_M) & \mbox{ in } \ \Omega_p , \\
\dsp \nabla \Psi_M \cdot \nu = -N_\sigma \Sigma^* & \mbox{ on } \, \partial\Omega_p ,\
 \Psi_M \; \mbox{ is } \Omega-\mbox{periodic}.
\end{array} \right.
\end{equation}
Note that $\Phi_M^\prime(\Psi)$ is a bounded Lipschitz function and its primitive
$\Phi_M(\Psi)$ is a coercive $C^1$-function, with a linear growth at infinity.
Therefore, for $\Sigma^* \in L^\infty (\p \Omega_p)$ and $M$ sufficiently large,
the corresponding functional
$$
J(\psi) = \frac{1}{2} \int_{\Omega_p} | \nabla \psi |^2 + \int_{\Omega_p} \Phi_M (\psi)
+ N_\sigma \int_{\p \Omega_p} \Sigma^* \psi
$$
is lower semi-continuous with respect to the weak topology of $H^1$ and coercive on $H^1$.
Then the basic calculus of variations yields existence of at least one solution for
problem (\ref {BP12cutoff}). Furthermore, for smooth domains, $\Psi_M$ belongs to
$W^{2,q}(\Omega_p)$ for all $q<+\infty$.

{   The third step of our proof amounts
to prove an $L^\infty$- estimate for $\Psi_M$ such that,
for $M$ sufficiently large, it implies $\Phi_M (\Psi_M) = \Phi (\Psi_M)$
and, consequently, existence of at least one solution for problem (\ref{BP12}).}
We start by some simple lemmas giving bounds on the solute packing fraction $\xi$.

\begin{lemma}
\label{propertiesxi}
Let $p(\xi)$, $\xi\equiv\xi(\Psi,\Gamma)$ and $g(\psi)$
be given by (\ref{gammaeps}), (\ref {sol.xi}) and (\ref{idealg}) respectively.
Then we have
\begin{equation}
\label{xi.bound1}
\mathcal{A}_{\min} g(\Psi) \leq \xi e^{p(\xi)} \leq \mathcal{A}_{\max} g(\Psi) ,
\end{equation}
with
\begin{gather*}
\mathcal{A}_{\min}  =  \xi_c \min_r (\frac{\sigma_r}{\sigma_c})^3 , \quad
\mathcal{A}_{\max} =  \xi_c \max_r (\frac{\sigma_r}{\sigma_c})^3 e^{L_B \max_j \frac{z_j^2}{\sigma_j}} .
\end{gather*}
Let $\xi_0$ be the unique solution of $x \exp \{p(x)\} = \mathcal{A}_{\min} g_m $
where $g_m$ is the minimal value of $g(\psi)$. Then we have
\begin{equation}
\label{xi.bound2}
\xi_0 \leq \xi \leq \mathcal{A}_{\max} g(\Psi).
\end{equation}
\end{lemma}

\begin{proof}
Formula (\ref{sol.xi}) yields
$$
\xi \exp \{ p(\xi) \} = \xi_c \sum_{j=1}^N (\frac{\sigma_j}{\sigma_c})^3
n_j^0(\infty)\gamma_j^0(\infty) \exp \left\{ - z_j \Psi +
\frac{L_B \Gamma \Gamma_c z_j^2}{1+ \Gamma \Gamma_c \sigma_j} \right\} .
$$
Since
$$
0 < \frac{L_B \Gamma \Gamma_c z_j^2}{1+ \Gamma \Gamma_c \sigma_j}
\leq L_B \max_j \frac{z_j^2}{\sigma_j} ,
$$
we deduce the bound (\ref{xi.bound1}). The other bound (\ref{xi.bound2}) is
then a consequence of the fact that $p(\xi)\geq0$ and $\xi\to\xi e^{p(\xi)}$
is increasing.
\end{proof}

\medskip

For the sequel it is important to find a bound for $\xi$ which is independent
of $\xi_c$, small as we wish, at least for large values of the potential $\Psi$.

\begin{lemma}
\label{propertiesxi2}
Let $\xi\equiv\xi(\Psi,\Gamma)$ be the unique solution of (\ref{sol.xi}).
There exists a threshold $0<\xi^{\rm cr}<1$ such that, for any number $q\geq1$, there
exist positive values $\xi_{\min}, \xi_{\max} >0$ such that, for any characteristic value
$0<\xi_c<\xi^{\rm cr}/q$,
$$
\xi \geq \xi_{\min} \; \mbox{ if }  \Psi < \frac{-1}{z_N} \log \frac{1}{q\xi_c} , \quad
\xi \geq \xi_{\max} \; \mbox{ if }  \Psi > \frac{1}{|z_1|} \log \frac{1}{q\xi_c} .
$$
\end{lemma}

\begin{remark}
The point in Lemma \ref{propertiesxi2} is that the lower bounds $\xi_{\min}, \xi_{\max}>0$
are independent of $\xi_c$ (but they depend on $q$), on the contrary of $\xi_0$ in Lemma \ref{propertiesxi}. In the proof of Proposition \ref{Linft} the number $q$ will be chosen
as $O(1)$ with respect to $\xi_c$.
\end{remark}

\begin{proof}
We improve the lower bound for equation (\ref{GammaL2}) when the potential
is very negative $\Psi < (\log(q\xi_c))/z_N$.
From (\ref{xi.bound1}) we deduce for small $\xi_c$
\begin{gather*}
\xi e^{p(\xi)} \geq \mathcal{A}_{\min} g(\Psi) = n_N^0(\infty) \gamma_N^0(\infty)
\min_r (\frac{\sigma_r}{\sigma_c})^3 \frac{\xi_c}{q\xi_c}
\left( 1 + O(\xi_c^{1-z_{N-1} /z_N }) \right) \\ \geq
\frac{1}{2q} n_N^0(\infty) \gamma_N^0(\infty) \min_r (\frac{\sigma_r}{\sigma_c})^3 = O(1),
\end{gather*}
where the lower bound is independent of $\xi_c$. The conclusion follows by
defining $\xi_{\min}$ as the unique solution of
\begin{gather*}
\xi_{\min} e^{p(\xi_{\min})} = \frac{1}{2q} n_N^0(\infty) \gamma_N^0(\infty)
\min_r (\frac{\sigma_r}{\sigma_c})^3 .
\end{gather*}
Note that $\xi_{\min}$ is uniformly bounded away from 0 for small
$\xi_c$ since $\gamma_N^0(\infty) = O(1)$ by virtue of Remark \ref{rem.xic}.

The proof of the estimate for large values $\Psi > (\log(q\xi_c))/z_1$ is analogous.
\end{proof}

Now the upper bound in Lemma \ref{propertiesxi} implies
that for $\xi=\xi(\Psi,\Gamma)$ we have
\begin{gather}
\label{boundxi}
\frac{\xi_{\min}}{ g(\Psi) \xi_c \max_r (\frac{\sigma_r}{\sigma_c})^3} e^{- L_B \max_j \frac{z_j^2}{\sigma_j}} \leq e^{-p(\xi)} , \quad \mbox{for} \quad \Psi < \frac{-1}{z_N} \log \frac{1}{q\xi_c}.
\end{gather}
Indeed, by (\ref{xi.bound1}) and Lemma \ref{propertiesxi2}
$$
e^{-p(\xi)}\ \geq \ \frac{\xi}{\mathcal{A}_{\rm max}g(\Psi)}
\ \geq\ \frac{\xi_{\rm min}}{g(\Psi)\xi_c\max_r\big(\frac{\sigma_r}{\sigma_c}\big)^3}\,
e^{-L_B \max_j \frac{z_j^2}{\sigma_j}}.
$$

For the purpose of comparison we introduce the following auxiliary Neumann problem
\begin{equation}
\label{BU1}
\left\{  \begin{array}{ll}
\dsp - \Delta U = \frac{1}{|\Omega_p|} \int_{\p\Omega_p} N_\sigma \Sigma^* \, dS
& \mbox{ in } \ \Omega_p , \\
\dsp \nabla U  \cdot \nu = - N_\sigma \Sigma^* & \mbox{ on } \, \p\Omega_p , \\
U \; \mbox{ is } \Omega-\mbox{periodic} , & \int_{\Omega_p} U (x) \, dx =0.
\end{array} \right.
\end{equation}
Remark that (\ref{BU1}) admits a solution $U\in H^1_\#(\Omega_p)$ since the
bulk and surface source terms are in equilibrium. Furthermore, the zero
average condition of the solution gives its uniqueness.  It is known that
$U$ is continuous and achieves its minimum and maximum in $\overline\Omega_p$.
Define
$$
\overline{\sigma} = \frac{1}{|\Omega_p|} \int_{\p \Omega_p}
N_\sigma \Sigma^* \, dS \ , \quad
U_{\min} =\min_{x\in \overline\Omega_p} U(x) \ \mbox{ and } \quad
U_{\max} =\max_{x\in \overline\Omega_p} U(x) .
$$
Then our $L^\infty$-bound reads as follows.

\begin{proposition}
\label{Linft}
Let $\Psi_M$ be a solution for the cut-off problem (\ref{BP12cutoff}) and take
$$
M = \log \frac{1}{\xi_c} .
$$
Under assumption (\ref{Bound1}),
there exists a critical value $\xi^{\rm cr}>0$ such that, for any
$\xi_c \in (0, \xi^{\rm cr} )$, the solution $\Psi_M$ of (\ref{BP12cutoff}) satisfies the following bounds
\begin{gather}
-\frac{M}{z_N} \leq \Psi_M (x) \leq \frac{M}{|z_1|} .
\label{Linfty2}
\end{gather}
\end{proposition}

\begin{proof}
We write the variational formulation for $\Psi_M- U$ for any
smooth $\Omega$-periodic function $\varphi$. Taking into account the
definition $\Phi_M^\prime (\Psi_M) = \Phi^\prime (T_M (\Psi_M))$,
it reads
\begin{gather}
\int_{\Omega_p} \nabla (\Psi_M - U) \cdot\nabla \varphi \, dx  \notag \\
- \beta \sum_{j=1}^N z_j n_j^0(\infty) \gamma_j^0(\infty) \int_{\Omega_p}
\big( e^{-z_j T_M (\Psi_M)} -e^{-z_j T_M(U-C) } \big)
e^{L_B \frac{z_j^2 \Gamma_c \Gamma (T_M (\Psi_M))}{1+ \Gamma_c \Gamma (T_M(\Psi_M)) \sigma_j } - p(\xi(T_M (\Psi_{ M})))} \varphi \, dx  \notag \\
+ \int_{\Omega_p} \big( -\beta \sum_{j=1}^N z_j n_j^0(\infty) \gamma_j^0(\infty) e^{-z_j T_M(U-C) } e^{L_B \frac{z_j^2 \Gamma_c \Gamma (T_M (\Psi_M))}{1+ \Gamma_c \Gamma (T_M(\Psi_M)) \sigma_j } - p(\xi(T_M (\Psi_{M})))} + \overline{\sigma} \big) \varphi \, dx =0 .
\label{VVP}
\end{gather}
We take $\varphi (x) = (\Psi_M (x) -U(x) + C)^-$, where $C$ is a constant to be determined
and, as usual, the function $f^-=\min(f,0)$ is the negative part of $f$. The first term in (\ref{VVP}) is thus non-negative.

By monotonicity of $v\to -z_j \exp \{ -z_j T_M(v) \}$ the second term of (\ref{VVP}) is non-negative.
To prove that the third one is non-negative too (which would imply that $\varphi\equiv0$),
it remains to choose $C$ in such a way that the coefficient $Q$ in front of $\varphi$
in the third term is non-positive.

For a given number $q\geq1$ (to be defined later, independent of $\xi_c$) we define constants
$$
\tilde M = \log \frac{1}{q\xi_c} \leq
M = \log \frac{1}{\xi_c}
$$
and we choose $C = U_{\max} + \tilde M /z_N$.
Since $\varphi\neq0$ if and only if $\Psi_M<U-C$, we restrict the following
computation to these negative values of $\Psi_M$. In such a case, we have
$\Psi_M<\frac{-1}{z_N} \log \frac{1}{q\xi_c}$ (the same is true
for $T_M(\Psi_M)$) so we can apply (\ref{boundxi})
from Lemma \ref{propertiesxi2}.
Then, since $-M/z_N\leq T_M(U-C) \leq -\tilde M/z_N$,
we bound the coefficient $Q$ (decomposing the indices in $j^-$ for
negative valencies and $j^+$ for positive ones)
\begin{gather}
Q = \overline{\sigma} - \beta \sum_{j=1}^N z_j n_j^0(\infty) \gamma_j^0(\infty) e^{-z_j T_M(U-C) }
e^{L_B \frac{z_j^2 \Gamma_c \Gamma(T_M(\Psi_M))}{1+ \Gamma_c \Gamma(T_M(\Psi_M)) \sigma_j}
- p(\xi(T_M (\Psi_{M})))} \leq  \notag \\
\overline{\sigma} - \beta \sum_{j\in j^-} z_j n_j^0(\infty) \gamma_j^0(\infty)
e^{L_B \max_j \frac{z_j^2}{\sigma_j}} - \beta \sum_{j\in j^+} z_j n_j^0(\infty) \gamma_j^0(\infty)
e^{z_j \tilde M /z_N} e^{- p( \xi(T_M(\Psi_M)))}
\underbrace{\leq}_{\mbox{using (\ref{boundxi})}}\notag \\
\overline{\sigma} - \beta \sum_{j\in j^-} z_j n_j^0(\infty) \gamma_j^0(\infty)
e^{L_B \max_j \frac{z_j^2}{\sigma_j}} - \beta \sum_{j\in j^+} z_j n_j^0(\infty) \gamma_j^0(\infty)
\frac{\xi_{\min} e^{z_j \tilde M /z_N} e^{- L_B \max_j \frac{z_j^2}{\sigma_j}}}{ g(T_M (\Psi_M)) \xi_c \max_r (\frac{\sigma_r}{\sigma_c})^3}   .
\label{bound11}
\end{gather}
Next, for small $\xi_c$ (i.e. very negative values of $\Psi_M$),
the function $g(T_M(\Psi_M))$ is decreasing (and
equivalent to $n_N^0(\infty)e^{-z_N T_M(\Psi_M)}$ at $-\infty$)
$$
g(T_M(\Psi_M)) \leq g(-M/z_N) .
$$
Thus
\begin{gather}
\sum_{j\in j^+} z_j n_j^0(\infty) \gamma_j^0(\infty) \frac{e^{z_j \tilde M /z_N}}{g(T_M(\Psi_M))}
\geq \frac{1}{g(-M/z_N)} \sum_{j\in j^+} z_j n_j^0(\infty) \gamma_j^0(\infty) e^{z_j \tilde M /z_N} \notag \\
\geq z_N \frac{\xi_c (1+o(1)) }{q\xi_c (1+o(1))} = \frac{z_N}{q} (1+o(1)) .
\label{bound12}
\end{gather}
We insert inequality (\ref{bound12}) into the last term in (\ref{bound11})
which yields
\begin{equation}
\label{bound13}
Q \leq \overline{\sigma} - \beta \sum_{j\in j^-} z_j n_j^0(\infty) \gamma_j^0(\infty)
e^{L_B \max_j \frac{z_j^2}{\sigma_j}} - \beta
\frac{z_N \xi_{\min}}{q \xi_c \max_r (\frac{\sigma_r}{\sigma_c})^3}
e^{- L_B \max_j \frac{z_j^2}{\sigma_j}}(1+o(1)) .
\end{equation}
Then, recalling that $\xi_{\min}$ and $\gamma_j^0(\infty)$ are $O(1)$ for
small $\xi_c$, it follows that, for given $q\geq1$, there exists
$\xi^{\rm cr} < 1$ such that, for any $0< \xi_c \leq \xi^{\rm cr}$,
the expression on the right hand side of (\ref{bound13}) is negative.

Now we conclude that $\varphi = (\Psi_M-U+C)^- =0$, which implies
$\psi_M \geq U-U_{\max} - \frac{1}{z_N} \log \frac{1}{q\xi_c}$.
Choosing $q$ sufficiently large so that $\frac{1}{z_N} \log q\geq U_{\max}-U_{\min}$,
we deduce the lower bound $\psi_M \geq -M/z_N$ in (\ref{Linfty2}).

An analogous calculation gives the upper bound in (\ref{Linfty2}) and the Proposition is proved.
\end{proof}

{   As a conclusion of our three steps of the proof, we can
state the final result which is Theorem \ref{thm.exist3}, stated in
the simplified notations of this section.}

\begin{theorem}
\label{thm.main}
Let $\Sigma^* \in L^\infty (\p \Omega_p)$.
Under assumption (\ref{Bound1}) and for small enough
$\xi_c \in (0, \xi^{\rm cr})$, there exists a
solution of the Poisson-Boltzmann problem (\ref{BP12}),
$\Psi \in H^1(\Omega_p)\cap L^\infty(\Omega_p)$.
In particular, $n_j$ satisfies a uniform lower bound
$n_j(x) \geq C>0$ in $\Omega_p$.
\end{theorem}

\begin{proof}
Proposition \ref{Linft} implies that $T_M(\Psi_M)=\Psi_M$,
so $\Phi_M^\prime (\Psi_M)= \Phi^\prime (\Psi_M)$, which proves that
$\Psi_M$ solves the original Poisson-Boltzmann problem (\ref{BP12}).
\end{proof}

\begin{remark}
Note that the assumptions (\ref{Bound1}) and $\xi_c$ small enough
are completely independent of the scaling of the domain $\Omega_p$
and thus of $\ep$. Therefore, Theorem \ref{thm.main} applies
uniformly with respect to $\ep$ in the porous medium $\Omega_\ep$,
as stated in Theorem \ref{thm.exist3}.
\end{remark}

\begin{remark}
Of course, further regularity of $\Psi$ can be obtained by standard elliptic regularity
in (\ref{BP12}). For example, assuming
$\Sigma^* \in C^\infty (\partial\Omega_p)$, the right hand side of equation (\ref{BP12}) is
bounded and using the smoothness of the geometry, we conclude that
$\Psi \in W^{2,q} (\Omega_p)$ for every $q<+\infty$. By bootstrapping,
we obtain that $\Psi \in C^\infty (\overline \Omega_p)$.
\end{remark}

\section{Numerical results}
\label{secnum}

We perform two-dimensional numerical computations with the {\bf FreeFem++} package \cite{FreeFM}.
The goal of this section is to compute the effective coefficients constituting
the Onsager homogenized tensor (\ref{Onsager}), to study their variations in terms
of some physical parameters (concentration, pore size and porosity) and to
make comparisons with the ideal case studied in \cite{ABDMP} in a realistic model of porous media.
We use the same unit cell geometries and complete the same test cases as in \cite{ABDMP}. It corresponds
to a simple model of geological montmorillonite clays.

The linearization of the electrokinetic equations (see Section \ref{sec4}) allows us
to decouple the computation of the electrostatic potential from those of the cell problems.

In a first step, we compute the solution $\Psi^0$ of the nonlinear Poisson-Boltzmann equation \eqref{BP122} with the associated
hard sphere term $\gamma^{HS}$ and MSA screening parameter $\Gamma$, from which we infer
the activity coefficients $\gamma^0_j$ and the concentrations $n^0_j$.

Second, knowing the $n_j^0$'s, and thus the MSA screening parameter $\Gamma$,
we compute the hydrodynamic interaction terms $\Omega_{ij}$ (\ref{Omegac})-(\ref{OmegaHS}) and the electrostatic relaxation terms $\mathcal{R}_{ij}$ (\ref{Relax}).
In turn it yields the value of the tensor $K_{ij}$ given by (\ref{lin2}).
The concentrations $n_j^0$ and the tensor $K_{ij}$ play the role of coefficients in
the cell problems \eqref{StokesAux0}-\eqref{bcAux0} and \eqref{divAuxi}-\eqref{bcAuxi}.
Thus, we can now compute their solutions which are used to
evaluate the various entries of the effective tensor \eqref{Onsager} according to the formula from Proposition \ref{prop.eff}.
In all figures we plot the adimensionalized entries of the effective tensors \eqref{Onsager}.
However, when the concentrations are involved, we plot them in their physical units, namely we use the dimensional quantity
\begin{equation}
\label{def.conc}
n_j^*(\infty) = n_c \, n_j^0(\infty) .
\end{equation}
For large pores (compared to the Debye length) the electrostatic potential is varying
as a boundary layer close to the solid boundaries. In such a case, the mesh is refined
close to those boundaries (see e.g. Figure \ref{fig.ellipsoid}).
The total number of degrees of freedom is around 18000 (depending on the infinite dilution concentration $n_j^*(\infty)$).

The nonlinear Poisson-Boltzmann equation \eqref{BP122} is solved with
Lagrange P2 finite elements and a combination of a Newton-Raphson algorithm
and a double fixed point algorithm. The Newton-Raphson algorithm is used to
solve the Poisson-Boltzmann equation at fixed values of the MSA coefficients
$\gamma^{HS}$ and $\Gamma$. The double fixed point algorithm is performed
on these values of $\gamma^{HS}$ and $\Gamma$. It starts with the initial values
$\gamma^{HS}= 1$ and $\Gamma = 0$ which correspond to the ideal case.

Let $n=1,2,...,n_\fin$ be the iteration number of the first level of the fixed point algorithm (the outer loop)
which update the electrokinetic potential from the previous value $\Psi^{(n-1)}$
to the new value $\Psi^{(n)}$, keeping $\gamma^{HS}_{(n-1)}$ fixed.
We first solve the Poisson-Boltzmann equation with these initial values and
a MSA screening parameter initialized to $\Gamma^{(n-1)}_{(0)}$. Let us note
$\Gamma^{(n-1)}_{(k-1)}$ the generic term at iteration $k$. Here, the iteration number
$k=1,2,...,k_\fin$ refers to the second level (inner loop) of the double fixed point algorithm.
It yields the electrokinetic potential $\Psi^{(n-1)}_{(k-1)}$ and, through
(\ref{Gammaeps0}), the new $\Gamma^{(n-1)}_{(k)}$ value which allows us to
iterate in $k$. The inner iterations are stopped when the wished accuracy is reached
at $k=k_\fin$.

From this new electrokinetic potential $\Psi^{(n-1)}_{(k_\fin)}$, we determine
the species concentrations and, through (\ref{Hardsphere}), the solute
packing fraction $\xi^{(n-1)}$. At this stage, a new  hard sphere term
$\gamma^{HS}_{(n)}$ is defined and we start a new iteration of the outer loop.
The outer loop is broken when the wished accuracy is reached at $n=n_\fin$.

\begin{figure}[!t]
\centering
    \includegraphics[width=0.5\textwidth]{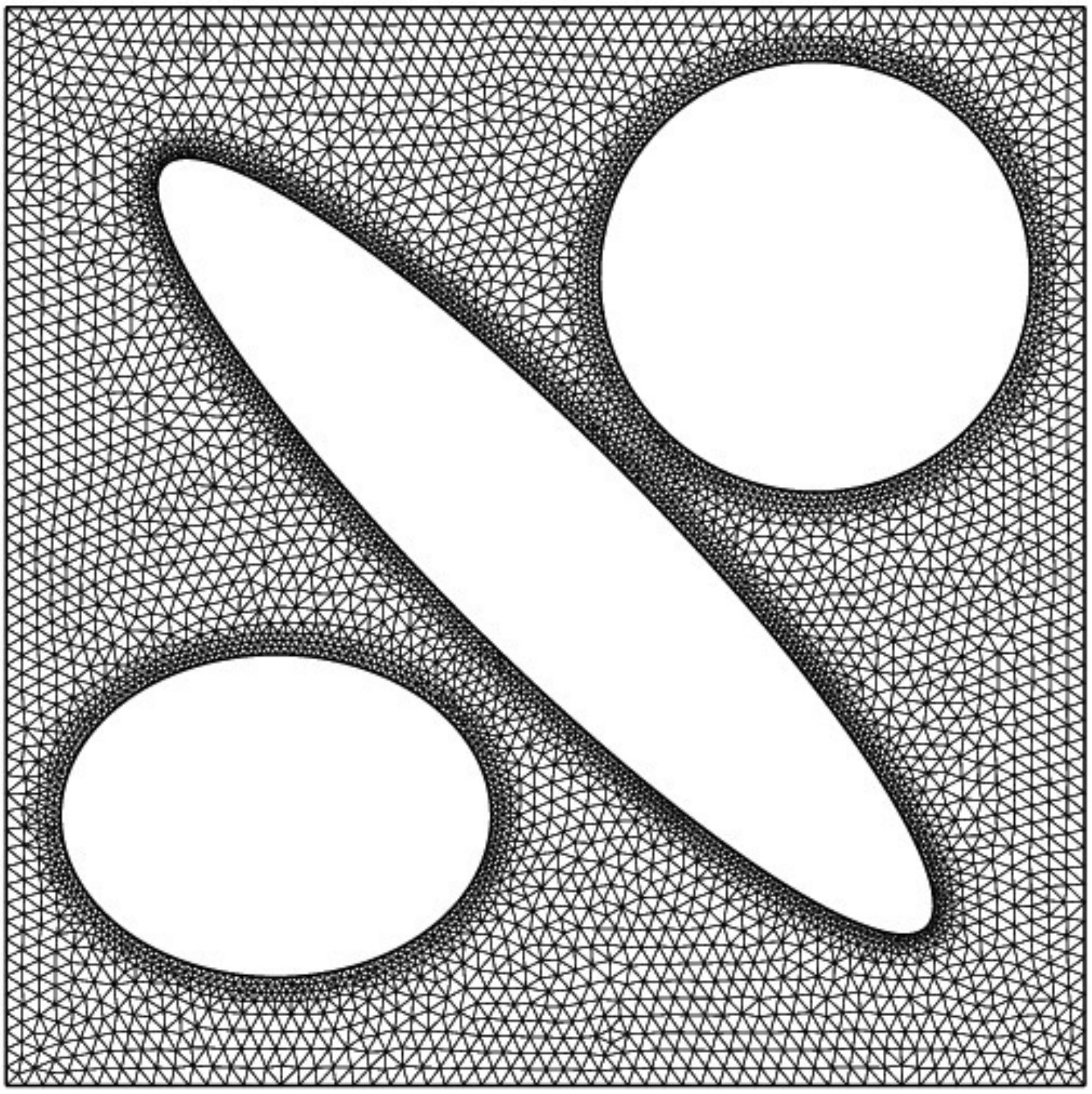}
   \caption{{\em Mesh for a periodicity cell with ellipsoidal inclusions (porosity is equal to $0.62$)}}
\label{fig.ellipsoid}
\end{figure}

All the following computations are ran for an aqueous solution of NaCl at $298$~K (Kelvin), where species $j=1$
is the cation Na$^{+}$ ($z_1=1$) with diffusivity $D_{1}^0  = 13.33 $e${-10}  \, $m$^2$/s and species $j=2$ the anion Cl$^{-}$ ($z_2=-1$) with $D_{2}^0 = 20.32 $e${-10}  \, $m$^2$/s
(note that this is the opposite convention of the previous sections where $z_1<0<z_2$).
The hard sphere diameters of the two species are considered equal to $3.3e-10 \, m$.
This model of NaCl electrolyte solution is able to reproduce both the equilibrium (activity coefficients, osmotic pressure) and the transport coefficients (conductivity, Hittorf transference number \cite{RS70}, self and mutual diffusion coefficient of the electrolyte) up to molar concentrations.
 The infinite dilution concentrations of the species are considered equal, $n_1^0(\infty) = n_2^0(\infty)$, and the characteristic concentration is $n_c=0.1 mole/l$.

The dynamic viscosity $\eta$ is equal to $0.89$e$-3\, $kg/(m$\,$s).
Instead of using the formula of Table \ref{Data} for defining the Debye length,  we use the following definition (as in the introduction)
$$
\lambda_D= \sqrt{\frac{\mathcal{E} k_B T}{ e^2 \sum_{j=1}^N n_j z_j^2 }} ,
$$
which differs by a factor of $\sqrt2$ in the present case of two monovalent ions.
Other physical values are to be found in Table \ref{Data}.
Following \cite{ABDMP} two model geometries are considered.
The first one features ellipsoid solid inclusions (see Figure \ref{fig.ellipsoid}),
for which we perform variations of concentrations from $10^{-3}$ to  1$\,$mol/l
and variations of the pore size ($3 \leq \ell \leq 50\,$nm).
The second one is a rectangular model (see Figure \ref{fig.rectangle}) which allows us to perform porosity variation.

\begin{figure}[!t]
\centering
    \includegraphics[width=0.3\textwidth]{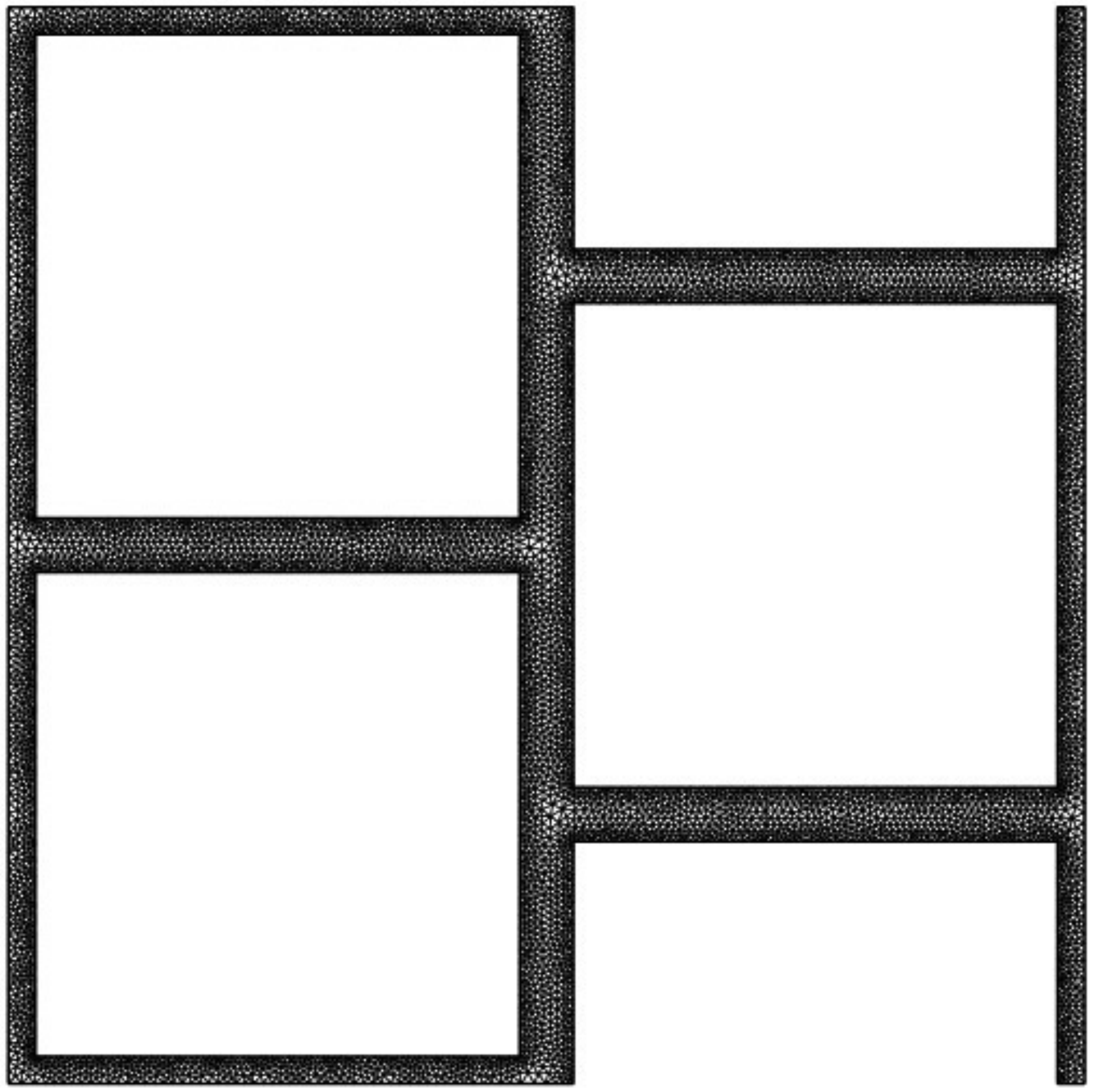}
    \includegraphics[width=0.3\textwidth]{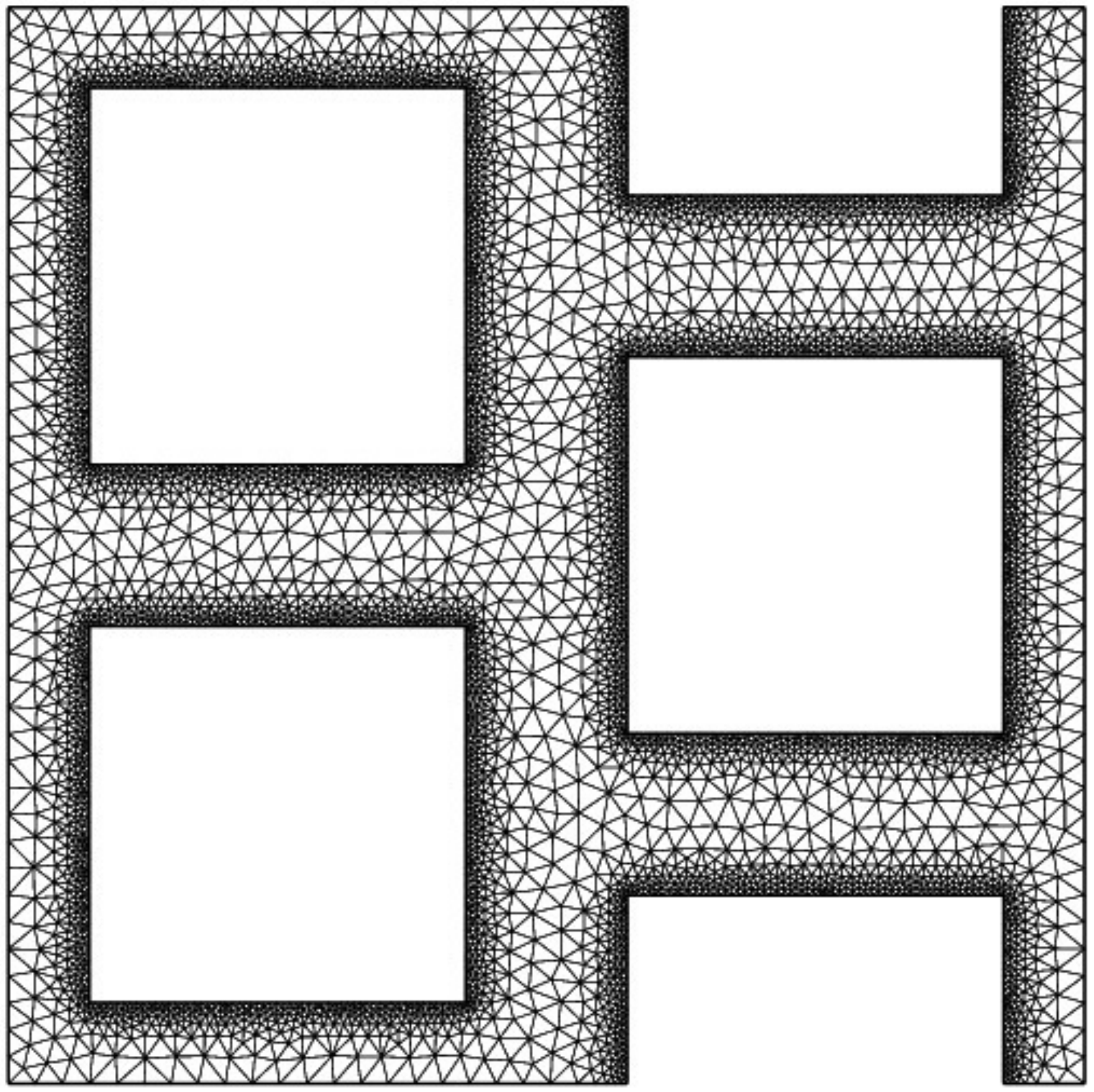}
    \includegraphics[width=0.3\textwidth]{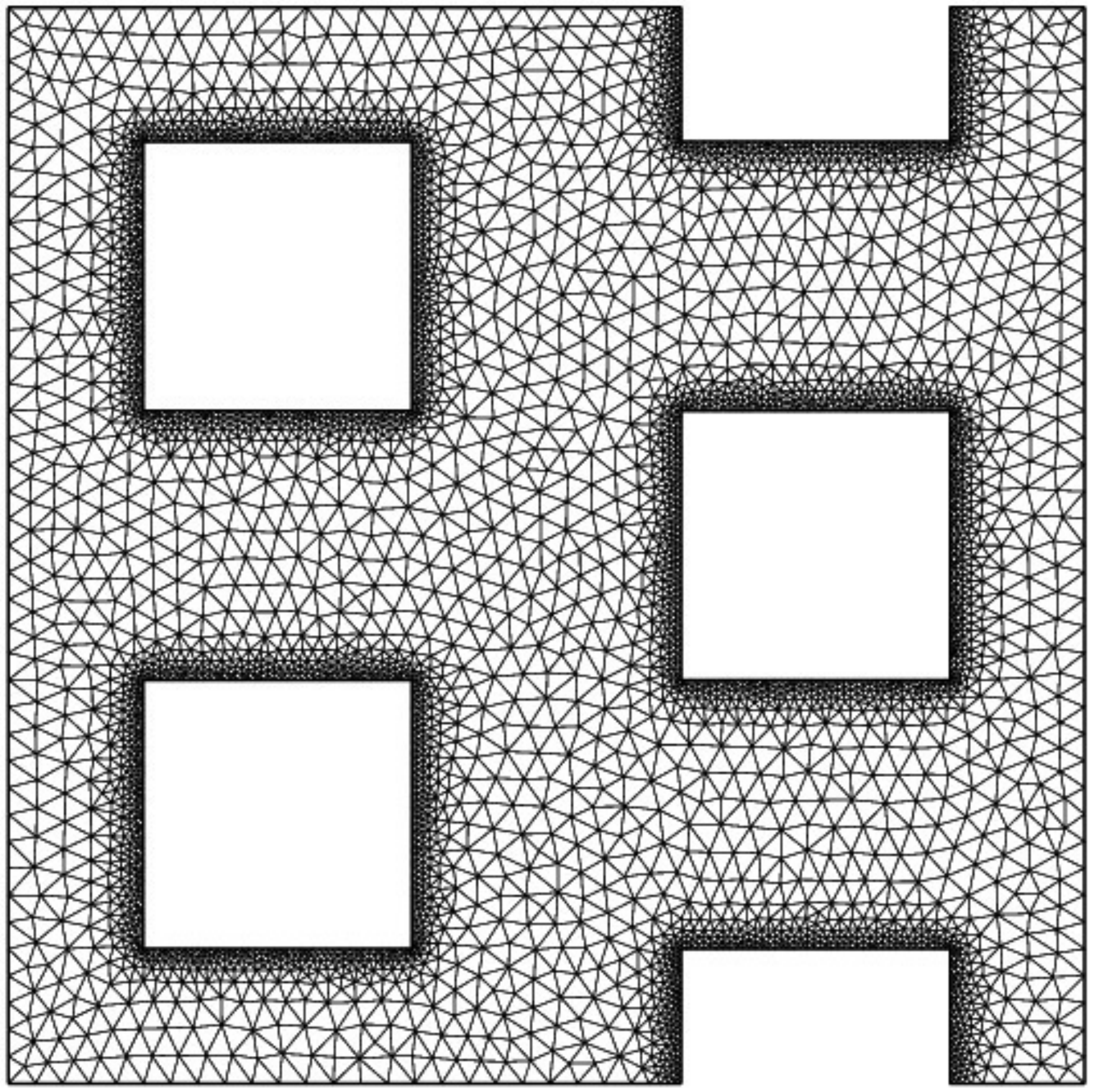}
   \caption{{\em Meshes for three different porosities ($0.19$, $0.51$ and $0.75$) of a periodic cell with rectangular inclusions}}
\label{fig.rectangle}
\end{figure}

\subsection{Variation of the concentration}

Here we consider the geometry with ellipsoidal inclusions (Figure \ref{fig.ellipsoid}).
We vary the infinite dilution concentrations $n_j^0(\infty)$ in the range $(10^{-2}, 10)$
or, equivalently through \eqref{def.conc}, the dimensional infinite dilution concentrations
$n_j^*(\infty)$ varies from $10^{-3}$ to 1$\,$mol/l. The pore size is $\ell$=50$\,$nm.
Varying proportionally all values of $n_j^0(\infty)$ is equivalent to varying
the parameter $\beta$ in the Poisson-Boltzmann equation \eqref{BP122}.

As can be checked on Figure \ref{fig.cmoy}, except for very small concentrations,
the cell-average of the concentrations $|Y_F|^{-1}\int_{Y_F} n_j(y)\,dy$ is almost
equal to the infinite dilution concentrations $n_j^0(\infty)$. This is clear in
the ideal case, but in the MSA case the cell-average of the concentrations is
slightly smaller than the infinite dilution concentrations for large concentrations.
It is a manifestation of the packing effect which forbids the boundary layer to
be too thin in the MSA setting. The behavior of Figure \ref{fig.cmoy} (bottom)
which represents the Donnan effect
was expected.  For small dilutions the MSA concentration is higher than the ideal
one because the electrolyte is in the attractive electrostatic regime so that there is a tendancy of
incorporating anions. It is the opposite for large dilutions : the electrolyte is in the repulsive
hard sphere regime and the excluded volumes expel the anions.

\begin{figure}[!t]
\centering
    \includegraphics[width=0.6\textwidth,angle=270]{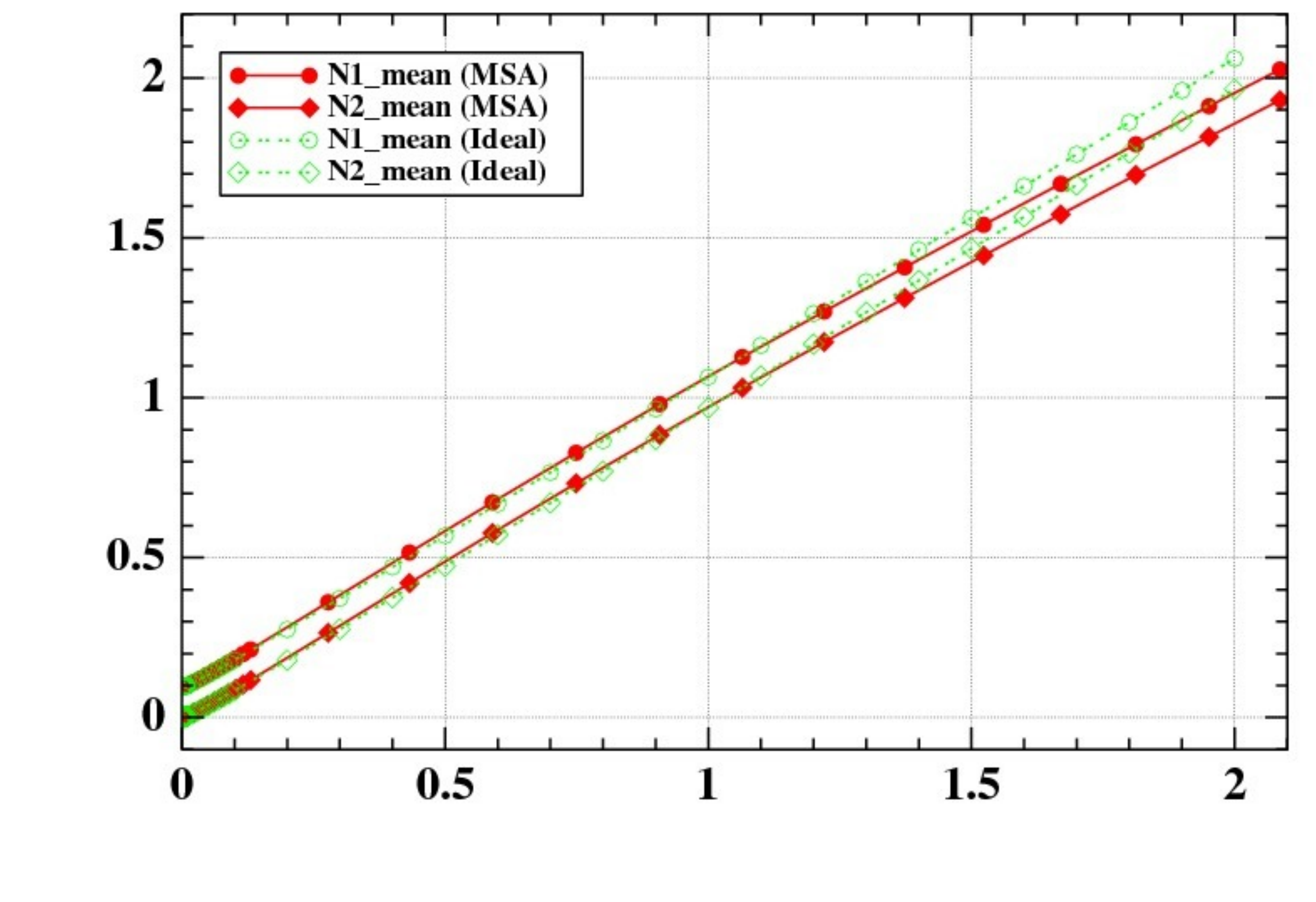}
    \includegraphics[width=0.8\textwidth]{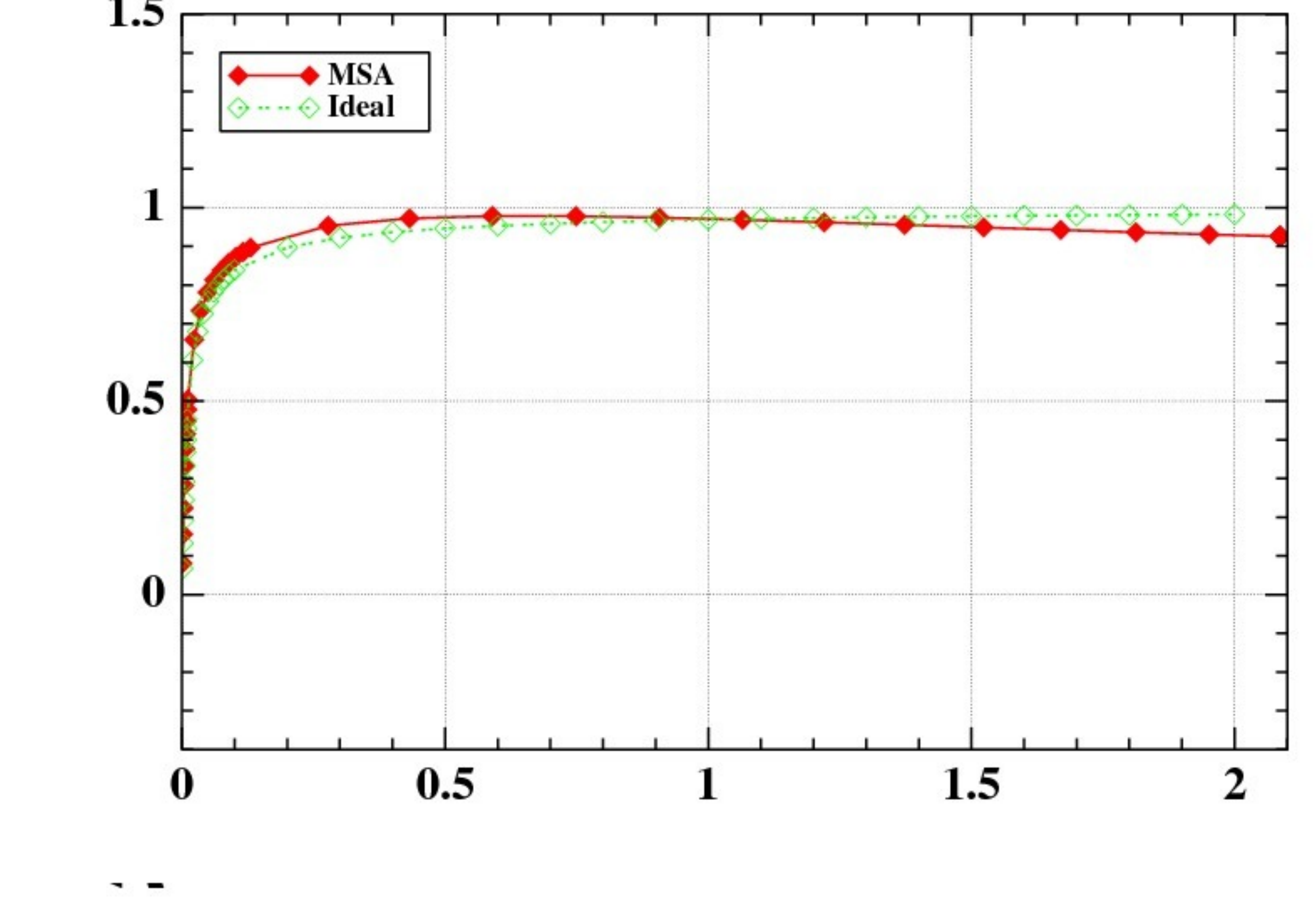}
   \caption{{\em Averaged cell concentrations $\textrm{Nj\_mean}=|Y_F|^{-1}\int_{Y_F} n_j(y)\,dy$ (top) and rescaled averaged cell anion concentration $\textrm{N2\_mean}/n_2^0(\infty)$ (bottom)
as a function of the dimensional (mol/l) infinite dilution concentrations $n_j^*(\infty)$}}
\label{fig.cmoy}
\end{figure}

Since the permeability tensor $\mathbb{K}$ depends on the pore size $\ell$,
we renormalize its entries by dividing them by the corresponding
ones for a pure filtration problem (computed through the usual Stokes
cell problems \cite{hornung}).
The resulting relative permeability coefficients
are plotted on Figure \ref{fig.perm}: the smaller the infinite dilution concentration,
the smaller the permeability. We clearly see an asymptotic limit of the relative
permeability tensor not only for high concentrations
but also for low concentrations. In the latter regime, the hydrodynamic
flux is reduced: the electrostatic attraction of the counterions with respect to the surface slows down the fluid motion. This effect is not negligible
because the Debye layer is important.
The MSA model differs from the ideal case. The curve is qualitatively the same but the electrostatic reduction of the Darcy flow is more important. Non ideality diminishes the mobility of the counterions at the vicinity of the surface so that the electrostatic interactions in the double layers are more pronounced.

\begin{figure}[!t]
\centering
 \includegraphics[width=0.7\textwidth]{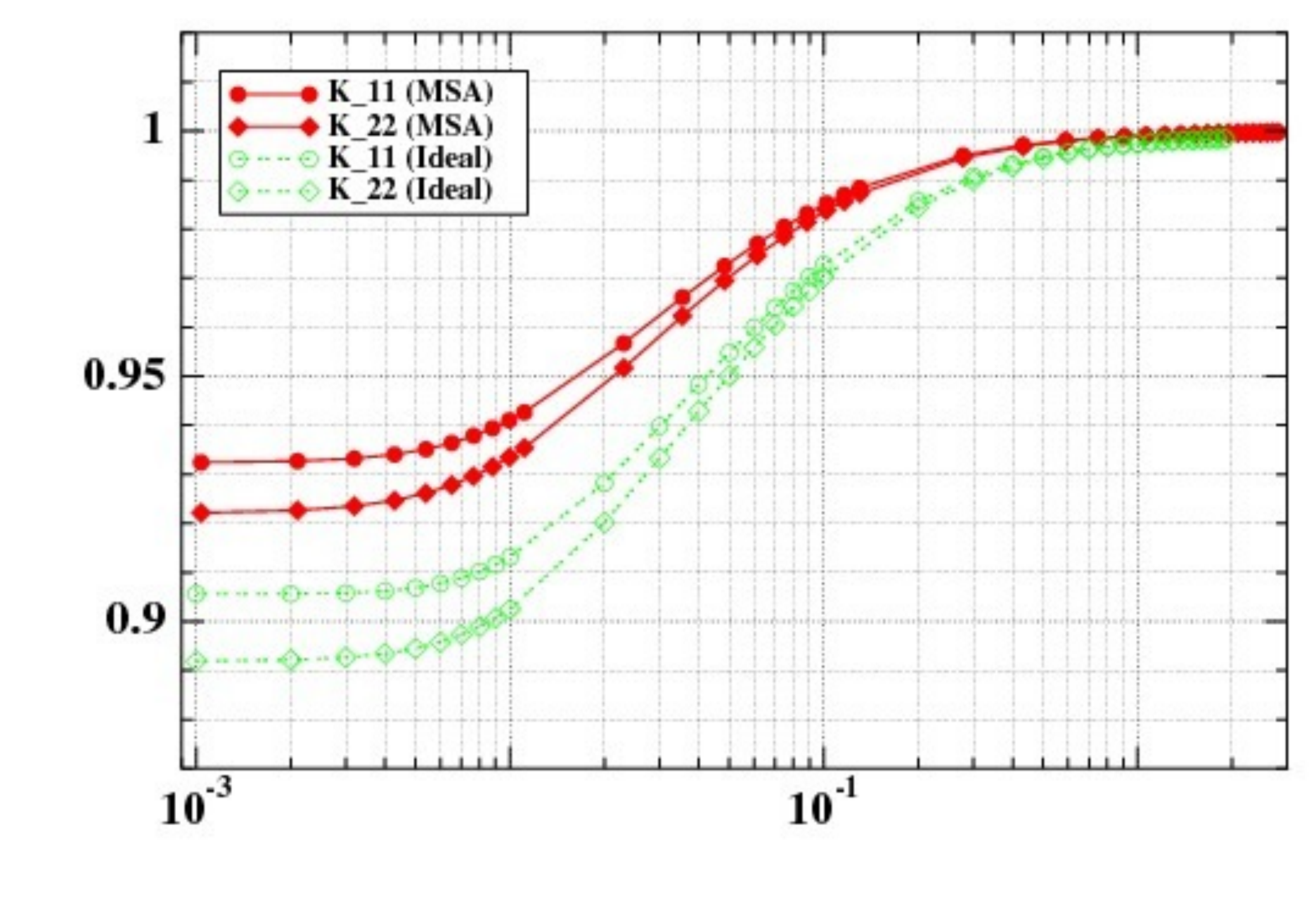}
   \caption{{\em Diagonal entries of the relative permeability tensor, $\mathbb{K}_{11}$ and $\mathbb{K}_{22}$, as functions of the dimensional (mol/l) infinite dilution concentrations $n_j^*(\infty)$}}
\label{fig.perm}
\end{figure}

The entries of the electrodiffusion tensor $\mathbb{D}_{11}$ for the cation
are plotted on Figure \ref{fig.diff11}.
A similar behavior is obtained for the other tensor $\mathbb{D}_{22}$ for the anion.
As expected the flux increases with the infinite dilution concentration $n_j^*(\infty)$.
It is not a linear law because even at low concentration there are still counterions ; they do not appear to be very mobile,
though.
The cross-diffusion tensor $\mathbb{D}_{12}$ is displayed on Figure \ref{fig.diff12}:
for large concentrations it is of the same order of magnitude than the species
diffusion tensors $\mathbb{D}_{11}$ and $\mathbb{D}_{22}$, because of the strong electrostatic interactions
between the ions.
In all cases, the MSA model is close to the ideal one: it is only for
large concentrations that the values of the electrodiffusion tensor are
different, and smaller, for MSA compared to ideal.
{  There are probably compensating effects :} same charge correlations increase diffusion but this effect is somewhat counterbalanced by opposite charge correlations that slow down the diffusion process. Non-ideal effects could be more important in the case of further quantities such as the electric conductivity for which cross effects are additive.

The log-log plot of Figure \ref{fig.d11asympt} (where the slope of the curve
is approximately 2) shows that the electrodiffusion tensors $\mathbb{D}_{ji}$
behaves quadratically as a function of $n_j^*(\infty)$ when $n_j^*(\infty)$
becomes large. This asymptotic analysis can be made rigorous in the ideal case.
At low salt concentration, correlation effects (\textit{i.e.} non-ideality) enhance slightly diffusion. In this regime, there are no counterions. So the relaxation effect is purely repulsive and diffusion is enhanced \cite{diffPRL}. At high concentration, the co-ion concentration is not negligible and there is a classical electrostatic relaxation friction.

\begin{figure}[!t]
\centering
    \includegraphics[width=0.7\textwidth]{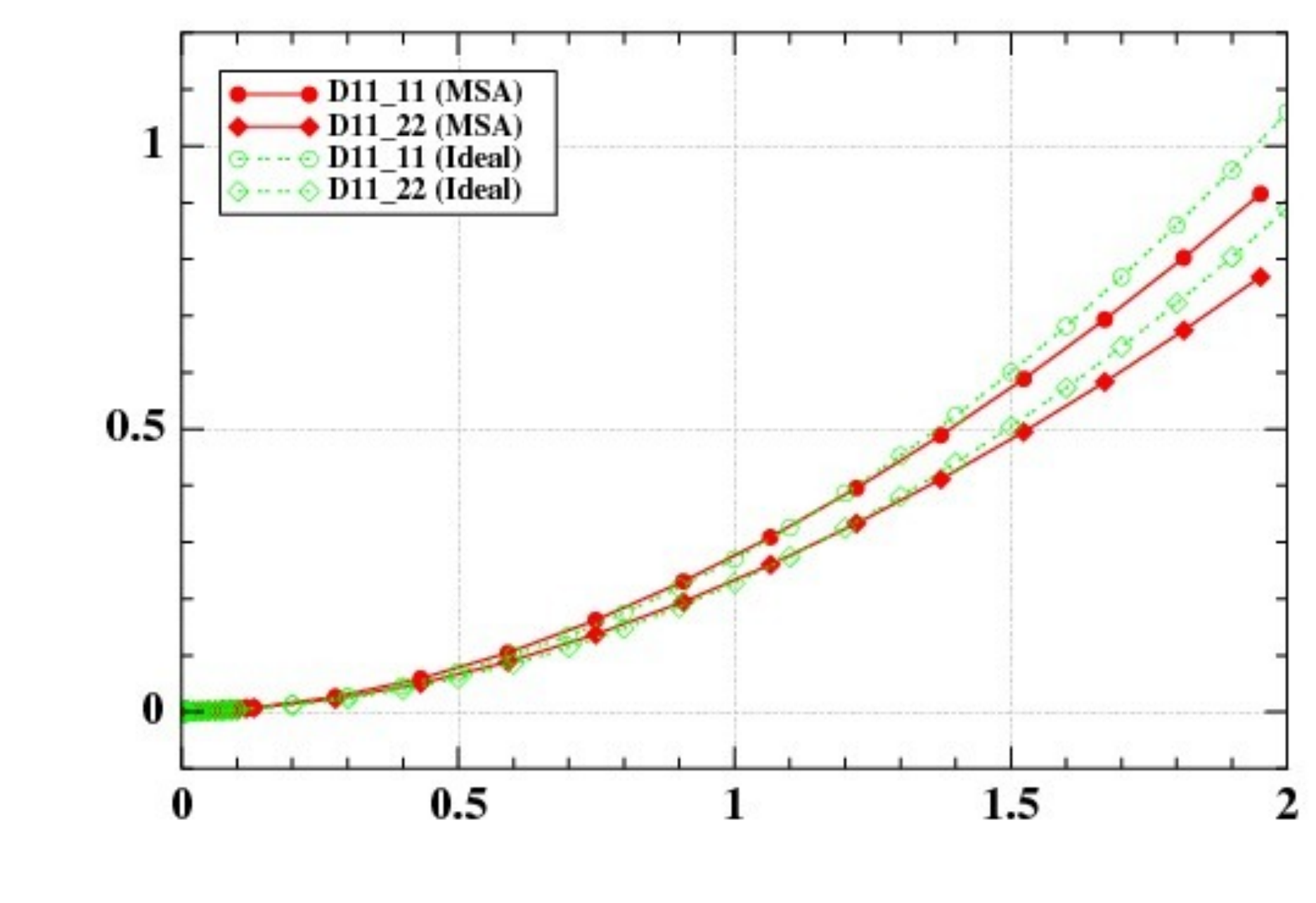}
   \caption{{\em Entries of the electrodiffusion tensor $\mathbb{D}_{11}$ for the cation,
as functions of the dimensional (mol/l) infinite dilution concentrations $n_j^*(\infty)$}}
\label{fig.diff11}
\end{figure}

\begin{figure}[!t]
\centering
    \includegraphics[width=0.7\textwidth]{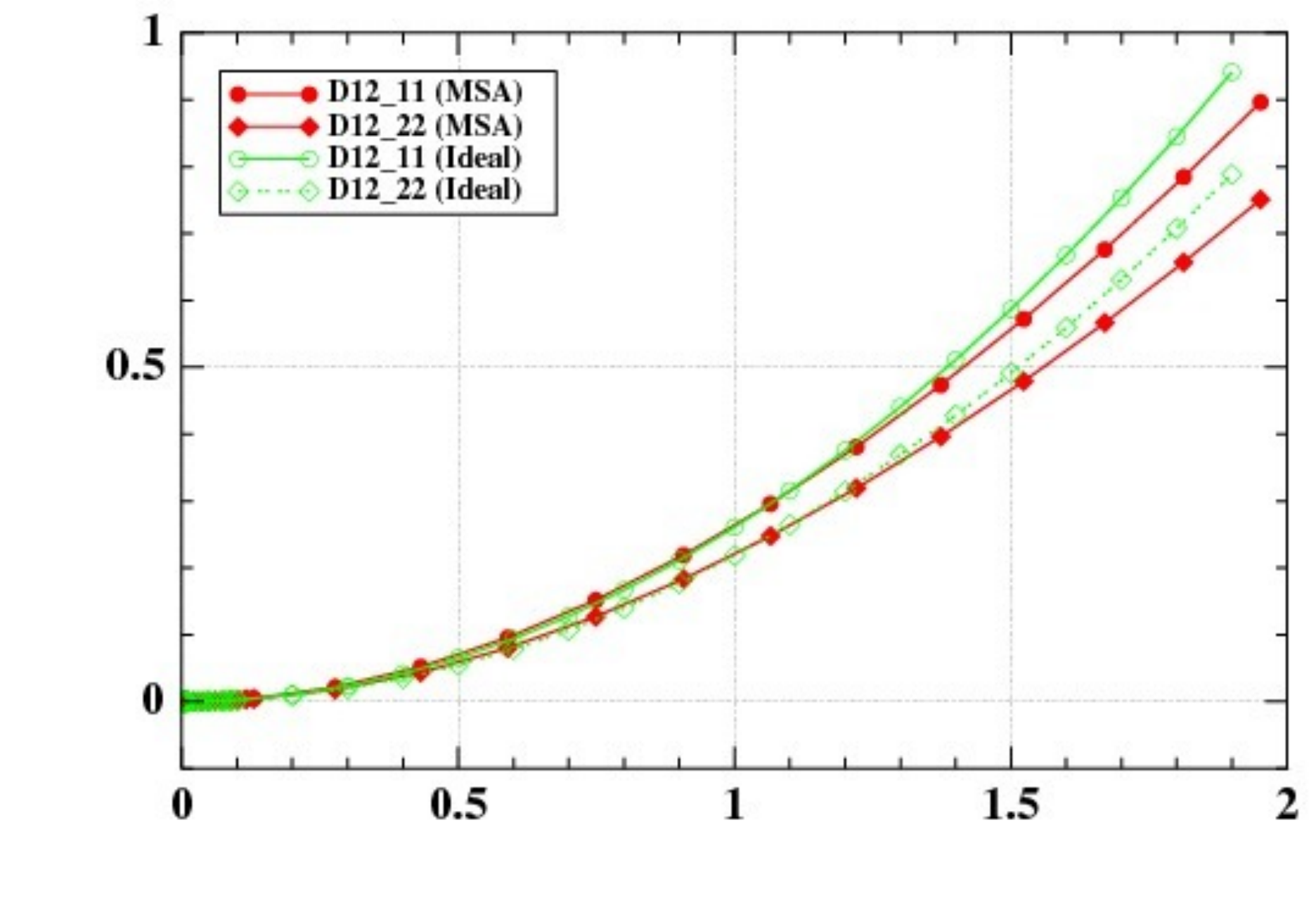}
   \caption{{\em Diagonal entries of the cross-diffusion tensor $\mathbb{D}_{12}$,
as functions of the dimensional (mol/l) infinite dilution concentrations $n_j^*(\infty)$}}
\label{fig.diff12}
\end{figure}

\begin{figure}[!t]
\centering
   \includegraphics[width=0.7\textwidth]{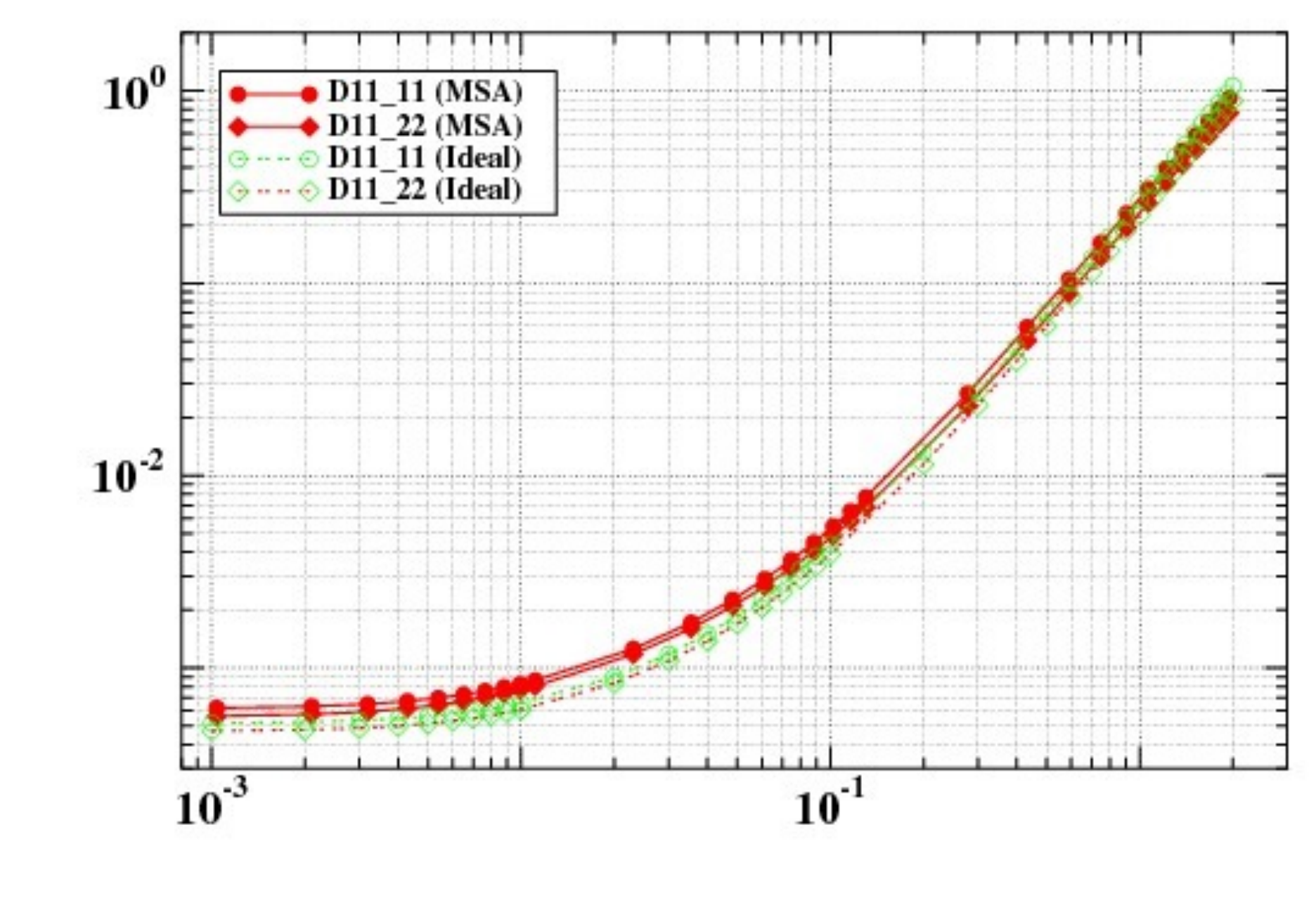}
   \caption{{\em Diagonal entries of the electrodiffusion tensor $\mathbb{D}_{11}$
as functions of the dimensional (mol/l) infinite dilution concentrations $n_j^*(\infty)$ (log-log plot)}}
\label{fig.d11asympt}
\end{figure}

The coupling tensors $\mathbb{L}_{1}$ and $\mathbb{L}_{2}$ are plotted on
Figure \ref{fig.l1l2}. The coupling is, of course, maximal for large concentrations
but the coupling tensor $\mathbb{L}_{1}$ for the cation does not vanish for very
small infinite dilution concentrations since the cell-average of the cation concentration
has a non-zero limit (required to compensate the negative surface charge) as
can be checked on Figure \ref{fig.cmoy}. The differences between the ideal and
MSA models are very limited in this logarithmic plot.

\begin{figure}[!t]
\centering
    \includegraphics[width=0.7\textwidth]{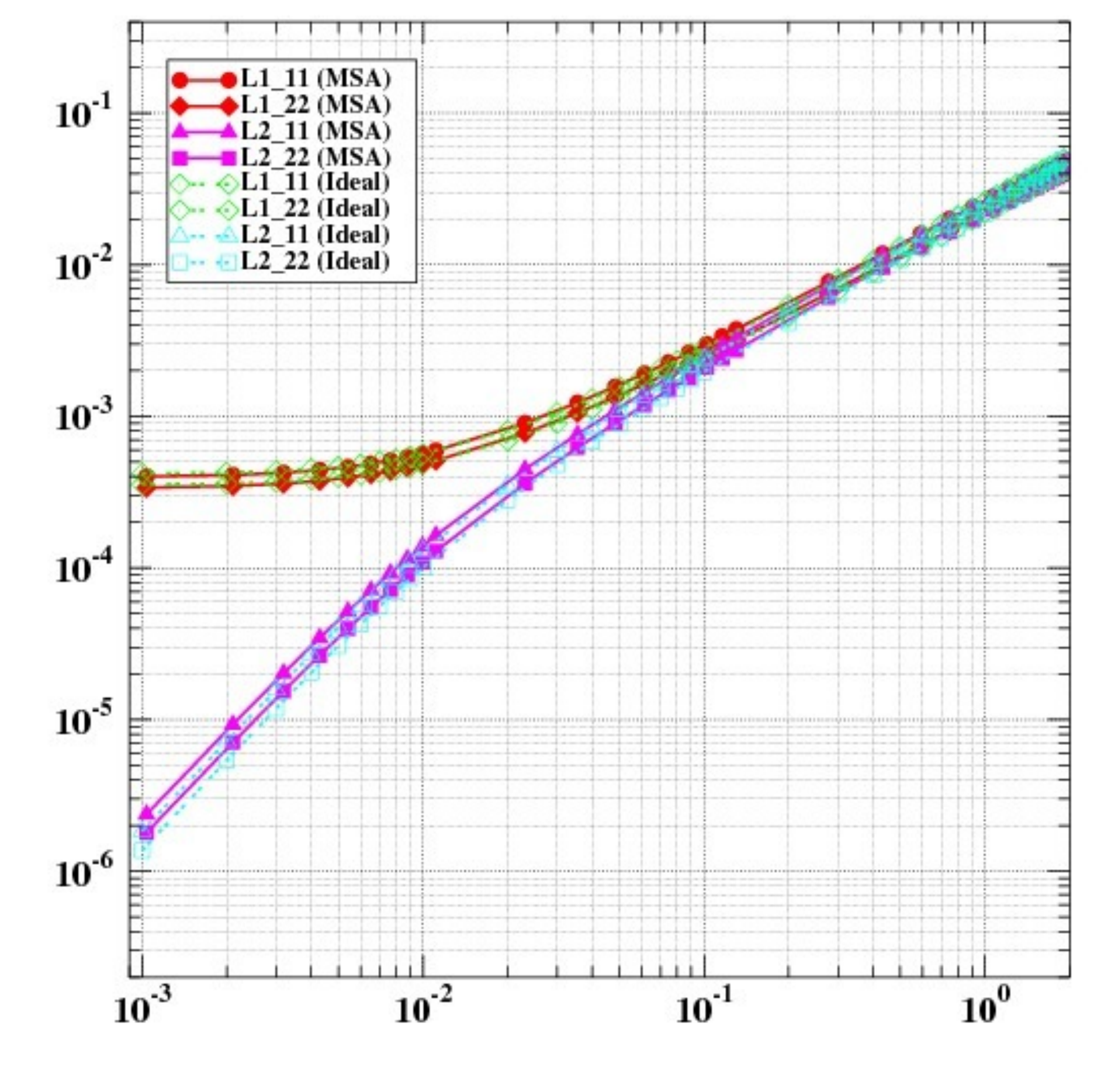}
   \caption{{\em Diagonal entries of the coupling tensors $\mathbb{L}_{1}$ and $\mathbb{L}_{2}$,
as functions of the dimensional (mol/l) infinite dilution concentrations $n_j^*(\infty)$ (log-log plot)}}
\label{fig.l1l2}
\end{figure}

\subsection{Variation of pore size }

We keep the same geometry with ellipsoidal inclusions (Figure \ref{fig.ellipsoid})
but we now vary the pore size $\ell$, which is equivalent to vary the parameter $\beta$,
defined by \eqref{def.beta}, in the Poisson-Boltzmann equation \eqref{BP122}.
It thus changes the values of the concentrations $n_j^0(y)$ which play the role
of coefficients in the cell problems \eqref{StokesAux0}-\eqref{bcAux0}
and \eqref{StokesAuxi}-\eqref{bcAuxi}. This is the only modification
which is brought into the cell problems.
{  We emphasize that varying the pore size does not change the geometry
of the unit cell, but simply changes the coefficients of the cell problems.}

The dimensional infinite dilution concentration $n_1^*(\infty)=n_2^*(\infty)$ is
$10^{-1}$ mol/l which yields a value $0.7678$ for the infinite dilute
activity coefficients $\gamma^1_0(\infty)=\gamma^2_0(\infty)$.

On Figure \ref{fig.conc-pore} we plot the cell-average of the concentrations
$|Y_F|^{-1}\int_{Y_F} n_j(y)\,dy$ as functions of the pore size $\ell$. Qualitatively, there
is a close agreement between the ideal and MSA cases, as can be checked on this logarithmic plot.
Yet, the departure from ideality is not negligible. For small pore size the
Donnan effect, which corresponds to the anion concentration, is typically 40 \% higher
than its value in the ideal case.
When the pore size goes
to infinity the averaged concentrations should converge to the infinite dilution
concentrations.

On Figure \ref{fig.perm-pore} we plot the relative permeability coefficients
with respect to the ones of the Stokes problem.
As was already observed in \cite{ABDMP}, the variation is not monotone and
there is a minimum for a pore size of roughly 20 nanometers. This effect is
less pronounced for the MSA model but the location of the $\ell$ value
where the minimum is attained is not affected.
This is the signature of a transition from a bulk diffusion regime for small pores to a surface diffusion regime (caused by large boundaries) at large pores. Globally, the counterions reduce the hydrodynamic flow because of the attraction with the surface, but this relaxation effect is less important at very large or very small pore size $\ell$. More precisely, if the pore size becomes very large, the electrostatic screening is important, as already mentioned. Thus the domain of attraction becomes very small and the lowering of the hydrodynamic flow is reduced: the permeability is increased. On the other hand, for very small pores, the counterion profile becomes more and more uniform. Consequently, there is no screening, but the hydrodynamic flow does not modify a lot the counterion distribution, since it is globally uniform and the resulting electrostatic slowdown becomes less important. The departures from ideality modelled by the MSA globally reduce the total variation of the permeability tensor because the mobility of the ions in the Debye layer is weaker and their dynamics influence less the Darcy flow.

\begin{figure}[!t]
\centering
    \includegraphics[width=0.7\textwidth]{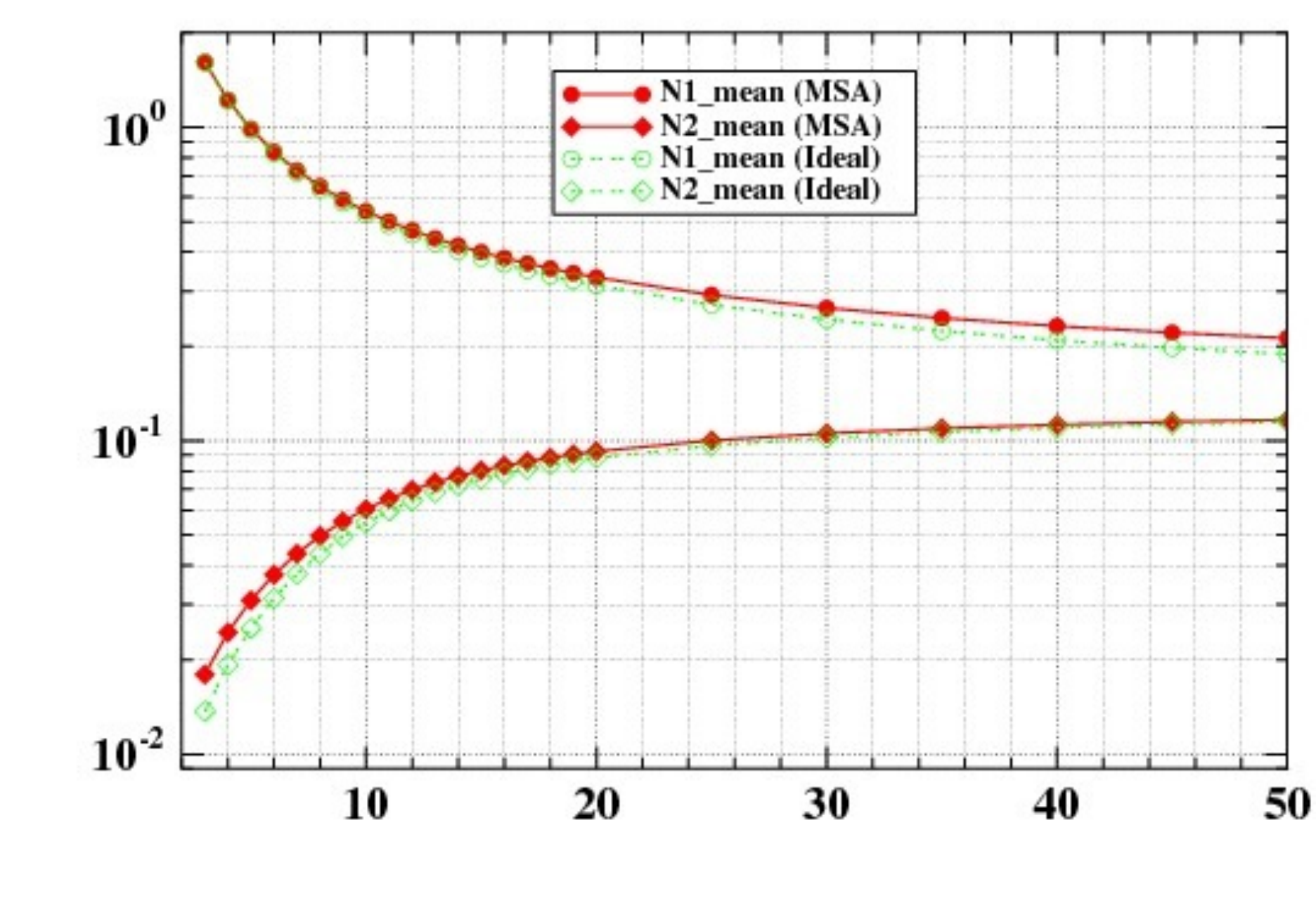}
   \caption{{\em Averaged cell concentration $\textrm{Nj\_mean}=|Y_F|^{-1}\int_{Y_F} n_j(y)\,dy$ versus pore size $\ell$ (nm)}}
\label{fig.conc-pore}
\end{figure}

\begin{figure}[!t]
\centering
    \includegraphics[width=0.7\textwidth]{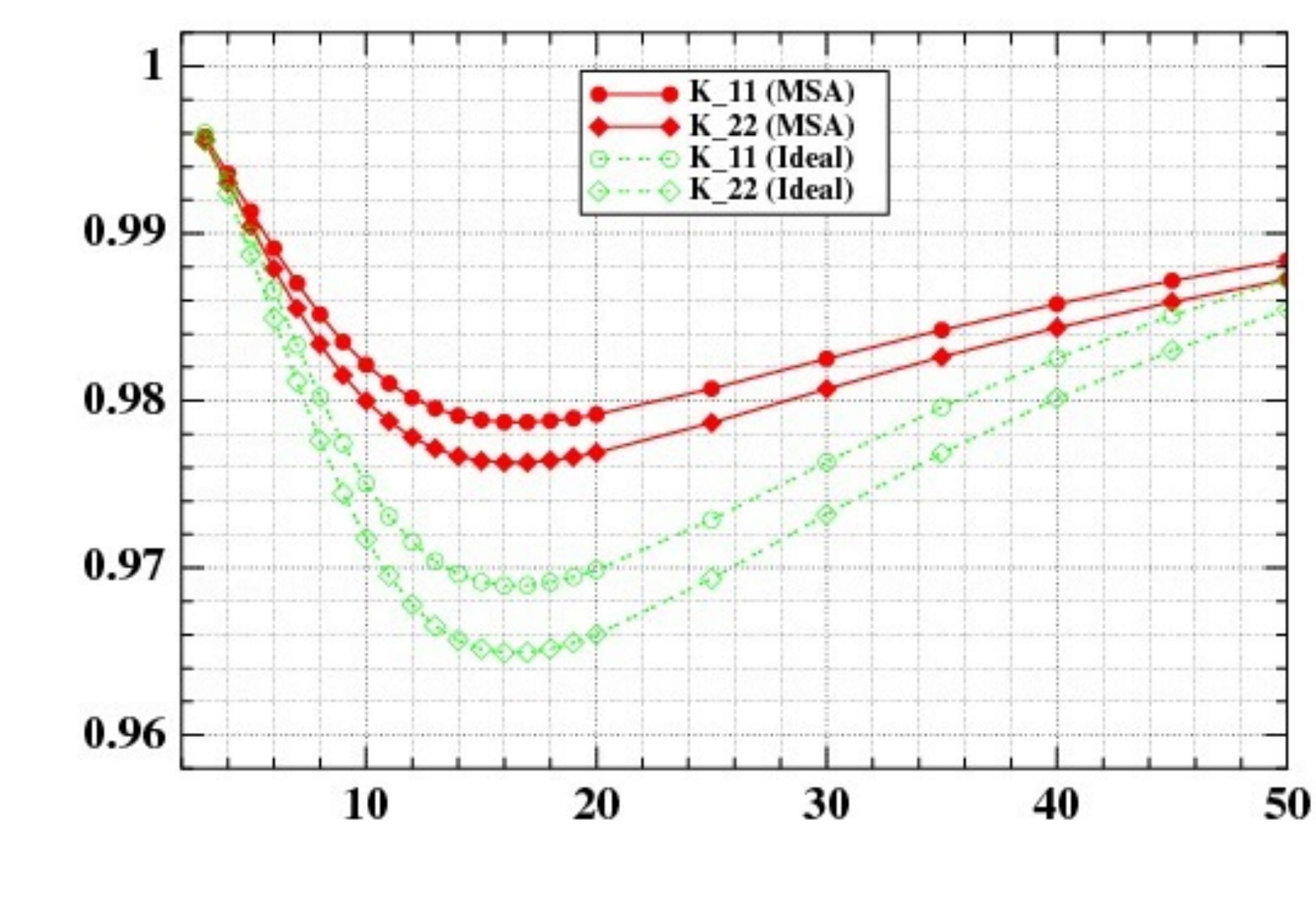}
   \caption{{\em Relative permeability coefficients $\mathbb{K}_{11}$ and $\mathbb{K}_{22}$ versus pore size $\ell$ (nm)}}
\label{fig.perm-pore}
\end{figure}

\subsection{Variation of the porosity}

Eventually we investigate the influence of the porosity on the effective tensors.
To this end we rely on the rectangular geometry where we vary the size
of the inclusions (see Figure \ref{fig.rectangle}). The infinite dilution
concentration is fixed at $n_j^0(\infty) = 1$, or $n_j^*(\infty)= 0.1\,$mol/l.
The porosity is defined as $|Y_F|/|Y|$ and takes the successive values
of $0.19 , \, 0.36, \,0.51, \,0.64, \,0.75$ in our computations.
{   Note that the porosity is independent of the pore size $\ell$
which is defined as the characteristic size of the entire periodicity cell,
i.e., the union of its fluid and solid parts.}
On Figure \ref{fig.conc-rect} we plot the cell-average of the concentrations
$|Y_F|^{-1}\int_{Y_F} n_j(y)\,dy$ as functions of the porosity. They
are almost identical between the ideal and MSA cases. When the porosity goes
to 1, meaning that there are no more solid charged walls, the averaged
concentrations should become equal, respecting the global electroneutrality.
On Figure \ref{fig.perm-rect} we check that the permeability tensor is increasing
with porosity, as expected. The same happens for the electrodiffusion tensor
$\mathbb{D}_{22}$ for the anion on Figure \ref{fig.D22-rect}.
More surprising is the behavior of the electrodiffusion tensor
$\mathbb{D}_{11}$ for the cation on Figure \ref{fig.D11-rect}: again there is a minimum
value attained for a $0.35$ value of the porosity.  This may be explained again by
a transition from a bulk diffusion regime for large porosities to a surface diffusion regime
(caused by the charged boundaries) for small porosities.

\begin{figure}[!t]
\centering
    \includegraphics[width=0.7\textwidth]{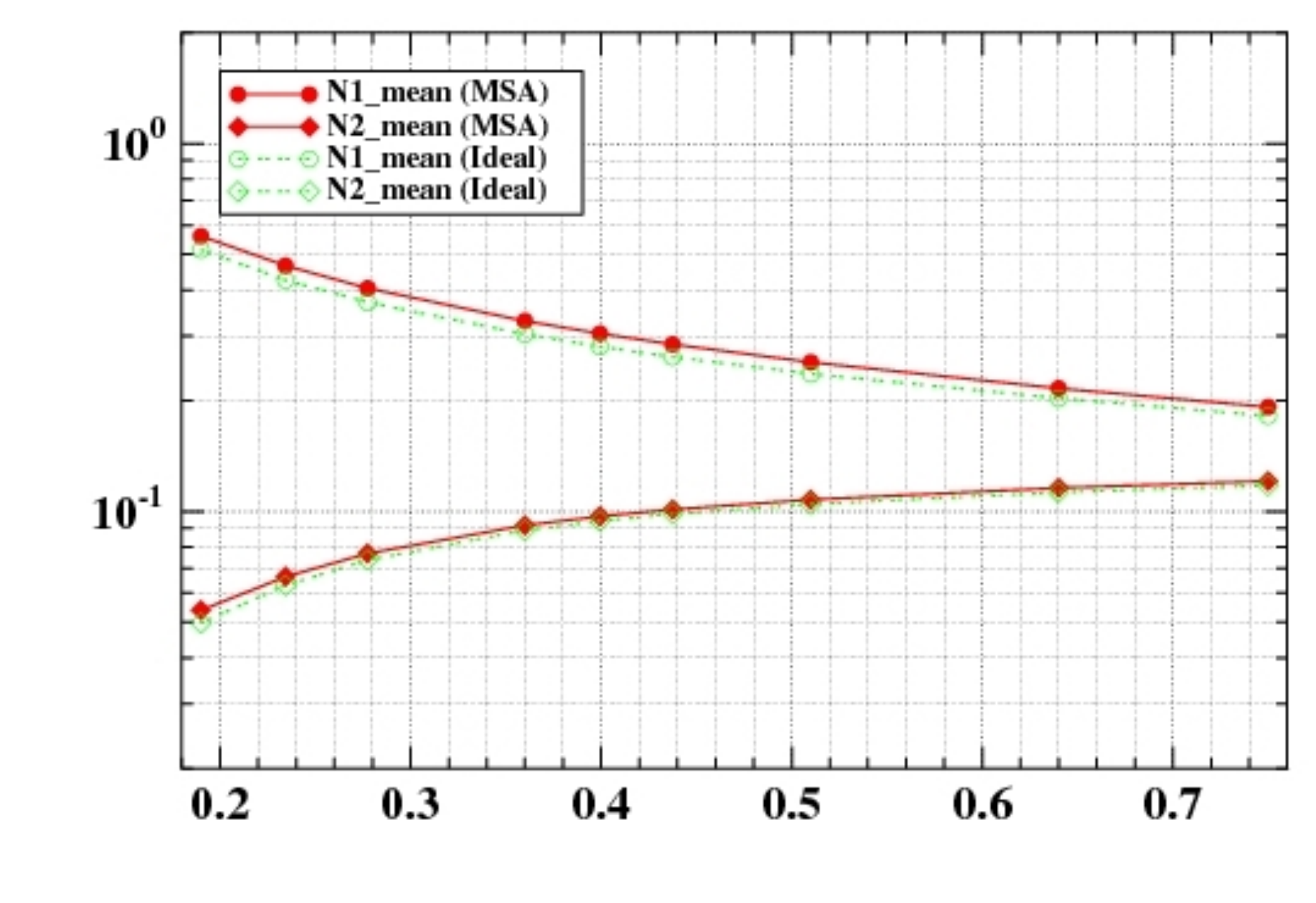}
   \caption{{\em Averaged cell concentration $\textrm{Nj\_mean}=|Y_F|^{-1}\int_{Y_F} n_j(y)\,dy$ versus porosity ($n_j^*(\infty) = 0.1 mole/l$)}}
\label{fig.conc-rect}
\end{figure}

\begin{figure}[!t]
\centering
    \includegraphics[width=0.7\textwidth]{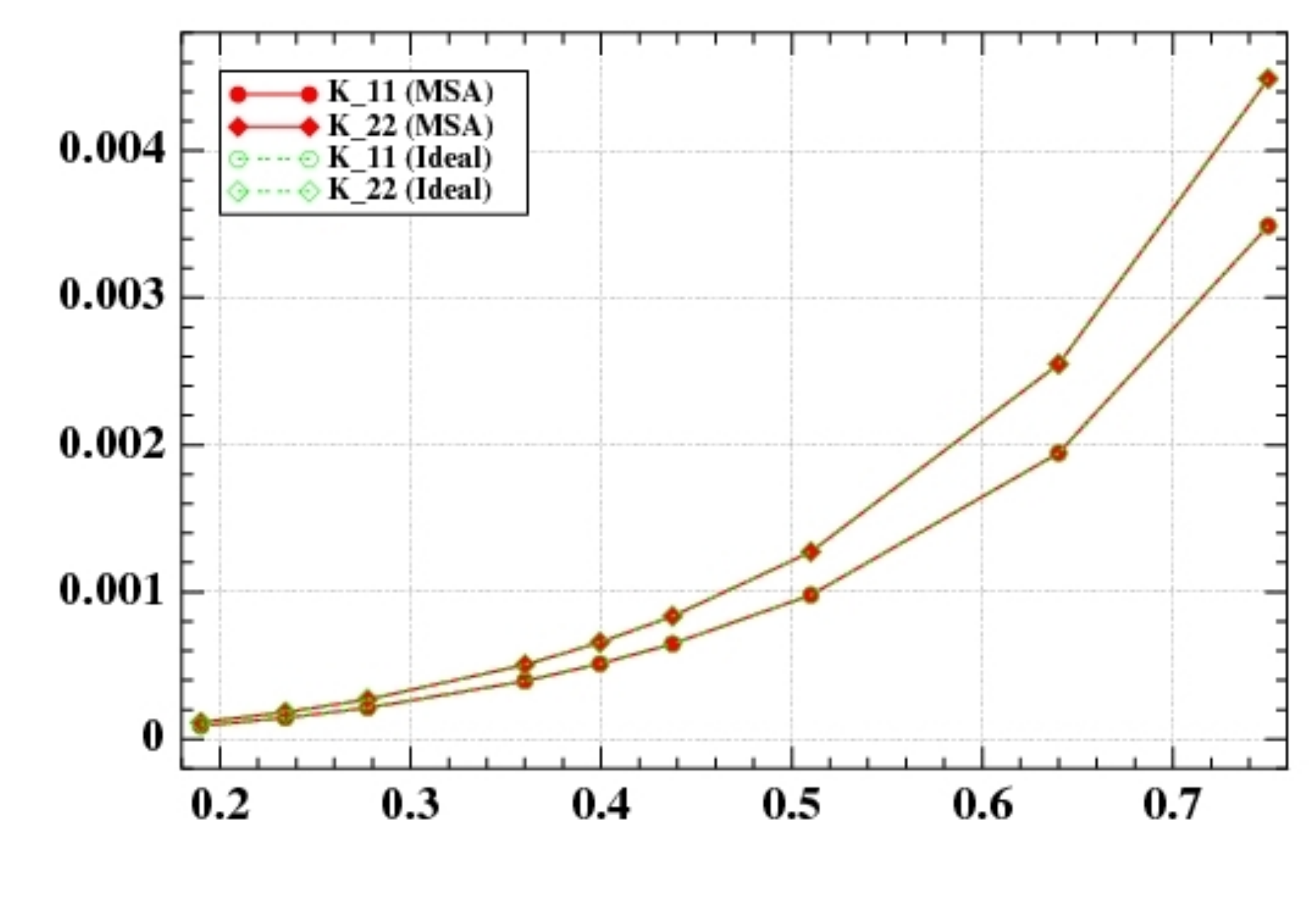}
   \caption{{\em Permeability tensor $\mathbb{K}$ versus porosity ($n_j^*(\infty) = 0.1 mole/l$)}}
\label{fig.perm-rect}
\end{figure}

\begin{figure}[!t]
\centering
    \includegraphics[width=0.7\textwidth]{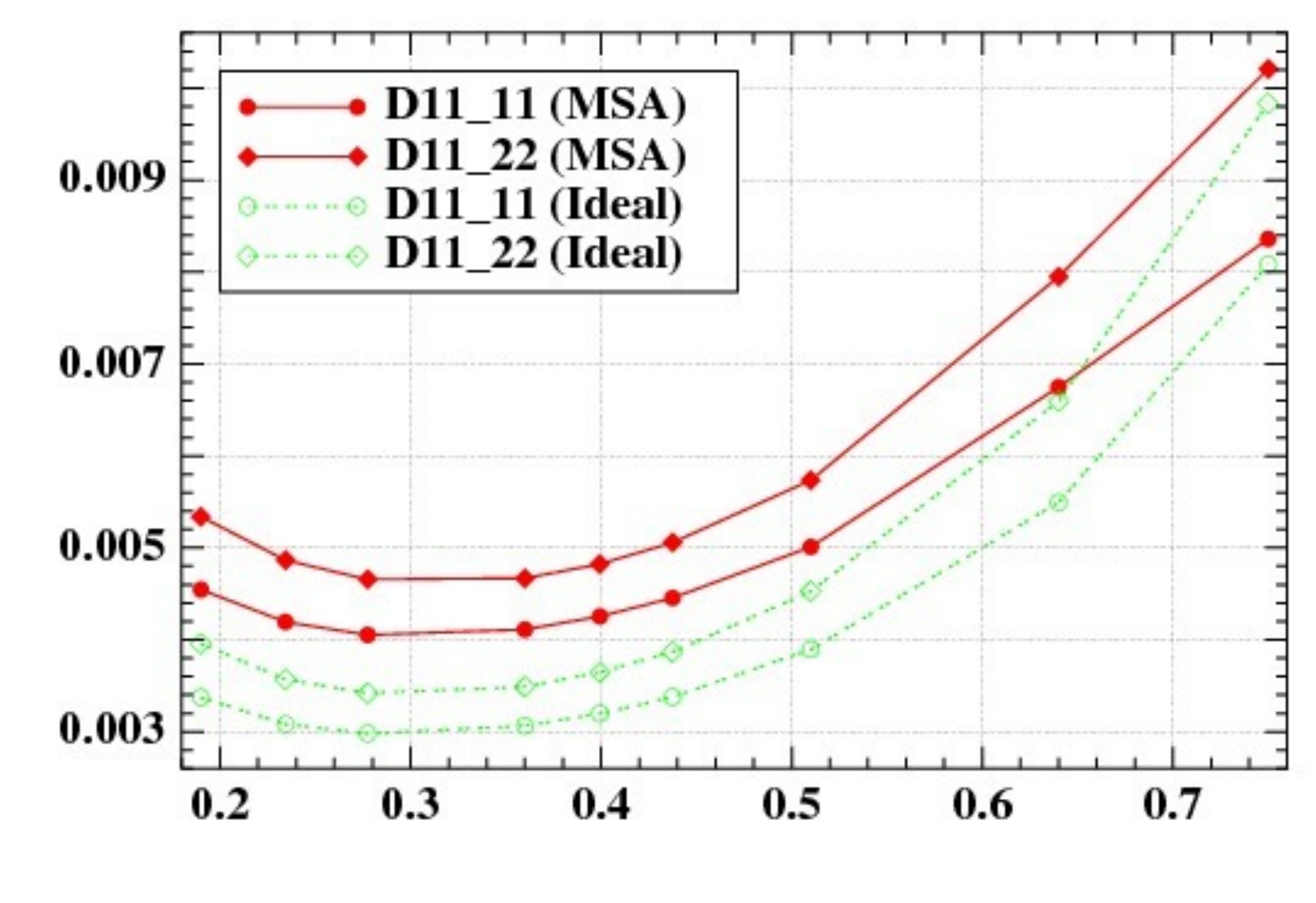}
   \caption{{\em Electrodiffusion tensor $\mathbb{D}_{11}$  for the cation versus porosity ($n_j^*(\infty) = 0.1 mole/l$)}}
\label{fig.D11-rect}
\end{figure}

\begin{figure}[!t]
\centering
    \includegraphics[width=0.7\textwidth]{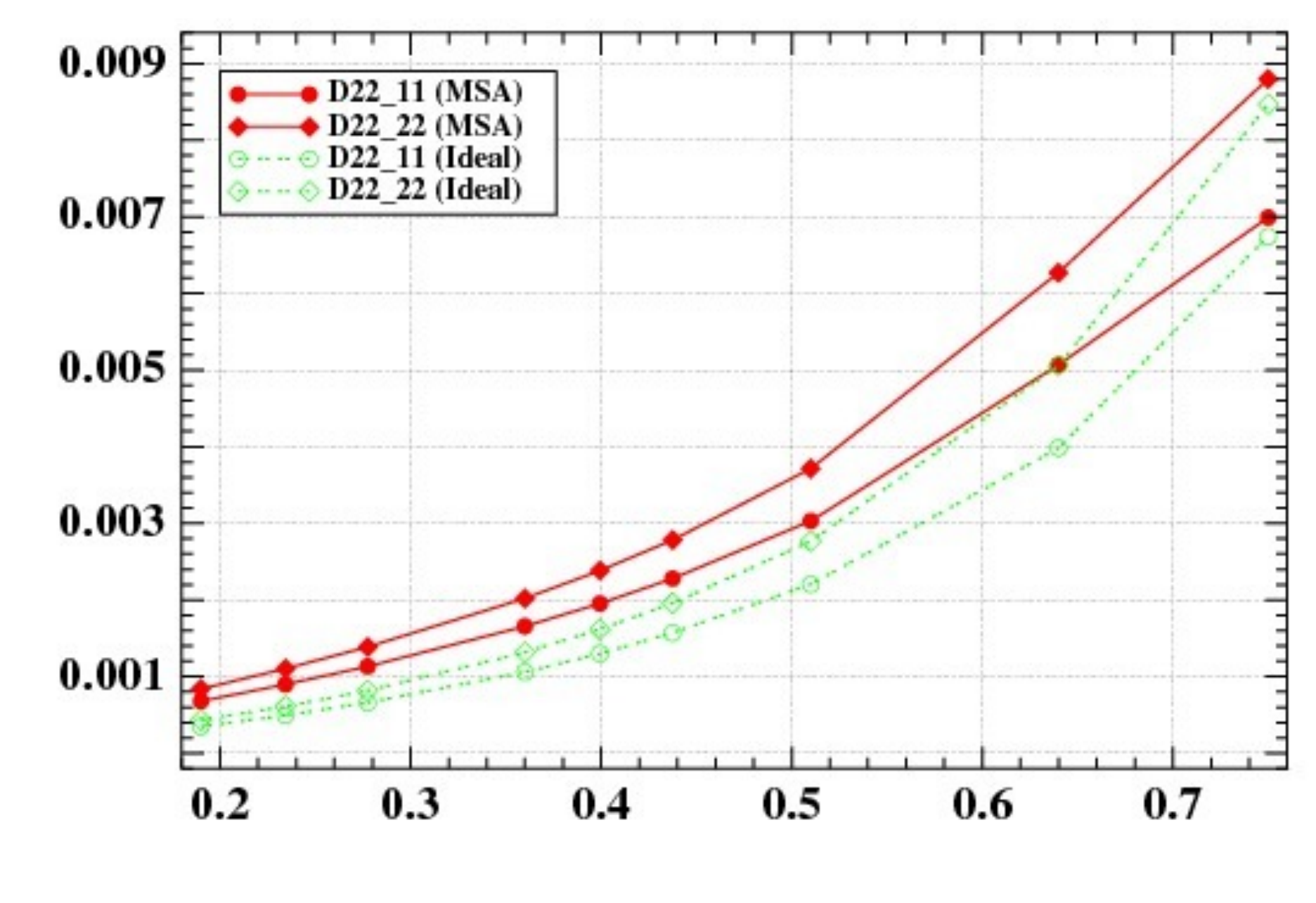}
   \caption{{\em Electrodiffusion tensor $\mathbb{D}_{22}$  for the anion versus porosity ($n_j^*(\infty) = 0.1 mole/l$)}}
\label{fig.D22-rect}
\end{figure}

The departures from ideality are found to be very important. They multiply the magnitude of diffusion by a factor of two, especially at low porosities for which the amount of anions is low. It corresponds to the case for which the relaxation effect is purely repulsive. A similar trend is obtained for the anion but the diffusion coefficient is much lower at low porosities: anions are expelled from the surface and they cannot have surface diffusion so that their transport properties are globally reduced.

\section{Conclusion}

{
We presented the homogenization (or upscaling) of the transport properties for a $N$-component electrolyte solution confined in a charged rigid porous medium. Contrary to what is commonly supposed in this domain the departures from ideality are properly taken into account thanks to a MSA-transport model, both for the equilibrium properties (activity coefficients $\gamma_j$) and for the transport quantities (Onsager coefficients $L_{ij})$. These non ideal effects are expected to be significant in most of the applications for which the electrolyte concentrations are typically molar. In the case of the equilibrium solution (in absence of external forces, apart from the surface charges on the solid wall), we prove the existence of (at least) one solution for small solute packing fractions (which corresponds to the validity of the MSA approach).

When a (small) external electric field is applied or when a (small) hydrodynamic or chemical potential gradient occurs, a rigorous homogenization procedure yields (at the linear response regime) the homogenized macroscopic laws. The effective Onsager tensor takes into account the departure from ideality, but it is still symmetric and positive definite. The significance of non-ideality has been studied by applying the results to a model of porous media (typically geological clays) for simple dissociated 1-1 electrolytes in water. It is shown that non-ideality only slightly modifies the qualitative aspects, but it can strongly modify the quantitative values, depending on the homogenized quantities.

For the equilibrium properties, it enhances the ion concentrations at low external concentration (because electrostatic attraction is predominant) and it reduces them at the opposite limit. The relative permeability tensor is increased but, in any case, it is close to the reference value calculated with a neutral solution.
The differences for the coupled diffusions and ion electrodiffusions depend on the concentrations and on the species. Similarly to bulk diffusion, the non-ideality can have an impact of the order of 50 \% for molar concentrations. Nevertheless, for some cases, there are compensating effects. It should be noted that for the model we considered the charges (ions, solid phase) were relatively low so that the differences should be magnified for highly charged media with higher valency electrolytes and higher concentrations. In that case, the result could be completely different because of the possibility of ion pairing that can change the sign of the ion charge. Nevertheless, the (relatively) simple MSA-transport theory we presented is not valid anymore in that case so that a realistic quantitative description of such complex media would require further developments.

To conclude, we showed that non-ideality can actually be important for the description of porous media.
Since most of the existing effective theories for concentrated systems are based on ideal models which
neglect the departure from ideality, the parameters that can be measured thanks to these
approaches may be wrongly estimated. In that case, they cannot be considered as robust structural quantities of the system: they are effective parameters that depend on the experimental conditions.

}


\end{document}